\RequirePackage{fix-cm}
\documentclass[smallextended]{svjour3}       
\smartqed  
\usepackage{graphicx}
\graphicspath{{./Figures/}}
\usepackage{latexsym,amsmath,amssymb,amsthm,amsfonts}
\usepackage{comment}
\usepackage{algorithm,algpseudocode}
\usepackage{xcolor}
\usepackage{subcaption}
\usepackage{amsfonts}
\usepackage{tcolorbox}
\usepackage[toc]{appendix}
\usepackage{bm}
\allowdisplaybreaks
\usepackage{float}
\usepackage{tikz}
\usepackage{amsthm}
\usepackage{hyperref}

\newtheorem{thm}{Theorem}
\newtheorem{prop}[thm]{Proposition}
\newtheorem{lem}[thm]{Lemma}

\newtheorem{defnition}[thm]{Definition}

\newcommand{\Rs}{\mathbb R}

\newcommand{\calA}{{\mathcal A}}

\newcommand{\calE}{{\mathcal E}}
\newcommand{\calF}{{\mathcal F}}

\newcommand{\calG}{{\mathcal G}}

\renewcommand{\a}{\mbox{\boldmath $a$}}

\newcommand{\one}{\mbox{\boldmath $1$}}
\newcommand{\zero}{\mbox{\boldmath $0$}}
\newcommand{\e}{\mbox{\boldmath $e$}}
\newcommand{\f}{\mbox{\boldmath $f$}}
\newcommand{\g}{\mbox{\boldmath $g$}}
\newcommand{\h}{\mbox{\boldmath $h$}}

\newcommand{\n}{\mbox{\boldmath $n$}}
\newcommand{\x}{\mbox{\boldmath $x$}}
\newcommand{\y}{\mbox{\boldmath $y$}}
\renewcommand{\u}{\mbox{\boldmath $u$}}
\renewcommand{\v}{\mbox{\boldmath $v$}}
\newcommand{\z}{\mbox{\boldmath $z$}}
\newcommand{\w}{\mbox{\boldmath $w$}}
\newcommand{\q}{\mbox{\boldmath $q$}}
\newcommand{\G}{\mbox{\boldmath $G$}}

\newcommand{\bxi}{\mbox{\boldmath $\xi$}}

\definecolor{green}{rgb}{0.0, 0.5, 0.0}

\newcommand{\Tr}{\ensuremath{^{\mr{T}}}} 
\newcommand{\mr}[1]{\ensuremath{\mathrm{#1}}} 
\newcommand{\fnc}[1]{\ensuremath{\mathcal{#1}}} 
\newcommand{\mat}[1]{\ensuremath{\mathsf{#1}}} 

\DeclareMathOperator{\diag}{diag} 

\newcommand{\DoneD}[1]{\ensuremath{\overline{\mat{D}}}_{#1}}
\newcommand{\QoneD}[1]{\ensuremath{\overline{\mat{Q}}}_{#1}}
\newcommand{\PoneD}[1]{\ensuremath{\overline{\mat{P}}}_{#1}}

\newcommand{\EoneD}[1]{\ensuremath{\overline{\mat{E}}}_{#1}}

\newcommand{\D}[1]{\ensuremath{\mat{D}}_{#1}}

\newcommand{\Q}[1]{\ensuremath{\mat{Q}}_{#1}}
\newcommand{\Pnorm}[1]{\ensuremath{\mat{P}}_{#1}}
\newcommand{\E}[1]{\ensuremath{\mat{E}}_{#1}}

\newcommand{\R}[1]{\ensuremath{\mat{\mat R}}_{#1}}
\newcommand{\I}[1]{\ensuremath{\mat{I}}_{#1}}

\newcommand{\Jdxildxm}[2]{\ensuremath{J\frac{\partial\xi_{#1}}{\partial x_{#2}}}}

\newcommand{\MJack}[1]{\ensuremath{\left[J\right]_{#1}}}
\newcommand{\MJdxildxmk}[3]{\ensuremath{\left[J\dfrac{\partial\xi_{#1}}{\partial x_{#2}}\right]_{#3}}}

\newcommand{\MSJack}[2]{\ensuremath{\left[\mat{\hat{J}}_{#1}\right]_{#2}}}
\newcommand{\Cmat}[2]{\ensuremath{\left[\mat{C}_{#1}\right]_{#2}}}
\newcommand{\Normal}[1]{\ensuremath{\left[\mat{\zeta}_{#1}\right]}}

\journalname{Journal of Scientific Computing}
\begin{document}

\title{A finite-difference summation-by-parts conditionally stable partitioned algorithm for conjugate heat transfer problems
}

\titlerunning{A FD-SBP-SATs conditionally stable partitioned algorithm for CHT problems}  
\author{ Sarah Nataj  \and
	     David C. Del Rey Fern\'andez \and
	     David Brown \and
	     Rajeev Jaiman
}


\institute{
	Submitted to Journal of Scientific Computing, May 2025\\
	\\
	Sarah Nataj \at
              Department of Applied  Mathematics, University of Waterloo, Waterloo, ON, Canada  \\
              \email{sarah.nataj@gmail.com}             \\
              \emph{Present address: Institute of Mechanical and Electrical Engineering, University of
              	Southern Denmark, Odense, Denmark} 
           \and
           David C. Del Rey Fern\'andez \at
           Department of Applied  Mathematics, University of Waterloo, Waterloo, ON, Canada \\
           \email{ddelreyfernandez@uwaterloo.ca}
           \and
           David Brown \at
           Ansys Canada Inc., Waterloo, ON, Canada \\
           \email{david.brown@ansys.com}
           \and
           Rajeev Jaiman \at
           Department of Mechanical Engineering, the University of British Columbia, Vancouver, BC, Canada\\
           \email{rjaiman@mech.ubc.ca}
}


\maketitle

\begin{abstract}
In this work, we design and analyze a novel, provably conditionally stable weakly coupled partitioned scheme to solve the conjugate heat transfer (CHT) problem. We consider a model CHT problem consisting of linear advection–diffusion and heat equations, coupled at an interface through the continuity of temperature and heat flux. We employ high-order summation-by-parts finite-difference operators in conjunction with simultaneous-approximation-terms (SATs) in curvilinear coordinates for spatial derivatives, combined with first- and second-order time discretization, and extrapolation in time at the interface. Energy stability is maintained by carefully defining SAT parameters at the interface. A range of coupling parameters are explored to identify those that yield a stable scheme and a step-wise approach for choosing SAT parameters that result in stability is given. The effectiveness of the method is demonstrated through numerical experiments in a two-dimensional model problem on rectangular domain with curvilinear grids. The proposed approach enables the development of high-order conditionally-stable partitioned solvers suitable for general geometries.


\end{abstract}

\section{Introduction}

Conjugate heat transfer (CHT) problems are examples of multi-physics problems that describe the thermal interaction between two media and have wide-ranging application, including aerospace, nuclear engineering, and biology \cite{ELKaramany2003,Lee2003,Zeng2020,Cang2022}.
Developing efficient and accurate solvers for CHT problems depends on careful discretization of the interface conditions between the fluid and solid/fluid, i.e., the continuity of temperature, heat flux, and possibly discontinuities in the thermal properties of the materials at the interface. The main objective of this work is to develop a provably energy-stable partitioned algorithm for a model transient heat conduction problem defined on two general subdomains with distinct thermal properties. We propose a novel high-order in space partitioned scheme that employs separate solvers for each subdomain and enforces the interface conditions using a penalty method to ensure conditional stability.

Numerical methods for solving multi-physics problems can be broadly categorized into monolithic and partitioned approaches \cite{Jaiman2022}. Monolithic methods consider the equations describing the physics in the whole domain as a single system. While this approach can be computationally efficient, it requires development of a specific solver for each coupled problem. In contrast,  a partitioned approach employs separate solvers in each subdomain which are sequentially solved until convergence is achieved, and thereby enable coupling existing single physics solvers to solve multi-physics problems \cite{Banks2014I,Banks2014II}. Partitioned schemes use two main coupling strategies: strongly coupled and weakly (or loosely) coupled approaches. Strongly coupled schemes perform multiple iterations per time step, solving the coupled equations until the interface conditions meet a specified tolerance. On the other hand, weakly coupled schemes rely on only one or a small fixed number of iterations \cite{Felippa2001}. Despite their flexibility, partitioned schemes often encounter stability issues for certain multi-physics problems \cite{Causin2005}.

The numerical stability of the CHT problem is discussed in the literature mostly through the Godunov–Ryabenkii normal-mode analysis 
\cite{Giles1997,Henshaw2009,Joshi2014,Roe2008,Meng2017}.
Giles in \cite{Giles1997} employed a traditional Dirichlet-Neumann coupling approach, while in \cite{Henshaw2009,Joshi2014} authors applied a mixed (Robin) interface condition. Roe et al. \cite{Roe2008} employed combined interface boundary conditions  to solve a coupled system of two thermal diffusion problems with a moving interface. Using a modal analysis, they established a conditional stability criterion. When the time step is chosen to satisfy a condition analogous to the Courant–Friedrichs–Lewy (CFL) condition, the overall stability of the coupled system depends on the ratio of the spatial resolutions (or mesh sizes) between the subdomains.  Meng et. al. in \cite{Meng2017} introduced  the CHT advanced multi-domain partitioned  scheme, where they determined the weights in the Robin condition through an optimization process based on a condition derived from the local stability analysis of the coupling scheme. They showed that the proposed scheme for coupled heat equations is second-order accurate and stable for the vast majority of practical problems. This is achieved by applying a second-order difference operator in space and a second-order backward-difference formula  in time, combined with a third-order extrapolation in time  at the interface.

CHT and fluid structure interaction (FSI) problems are closely related in many engineering applications, as both involve coupled fluid and solid domains at an interface. While CHT primarily involves thermal interactions between fluid and solid, FSI incorporates kinematic and dynamic conditions at interface. Stability of partitioned schemes for FSI problems has been widely  studied \cite{Banks2014I,Banks2014II,Fernandez2009,Banks2011,Hou2012,Banks2016}.

Of particular interest in this work  is the development of schemes based on high-order accurate operators in space that follow the summation-by-parts (SBP) property \cite{Strand1994,DelRey2014,Svard2014}. SBP operators, in conjunction with simultaneous approximation terms (SATs) \cite{Carpenter1994,Carpenter1999,Nordstrom1999,Mattsson2003}, have wide applications in fluid dynamics simulations and wave propagation problems \cite{Mattsson2006,Hicken2008,Kozdon2013,Carpenter2014}.  By replicating integration-by-parts at the discrete level, SBP-SAT schemes enforce boundary conditions weakly, providing a systematic framework for constructing energy-stable numerical methods, see the review papers  \cite{DelRey2014,Svard2014}.
Moreover, the SBP-SAT approach can be applied in the context of complex domains with curved boundaries through careful construction of SBP operators that mimic integration by parts \cite{Nordstrom2001}. This can be accomplished using a skew-symmetric splitting of the derivative operator (i.e., an average of a conservative and non-conservative form) and to maintain conservation, the metric terms are approximated such that they satisfy the discrete metric invariants \cite{Crean2018,Lundquist2018,DelRey2019,Nolasco2020}. The SBP-SAT approach has previously been used successfully to construct provably stable high-order discretizations of the CHT and FSI problems using a monolithic approach  \cite{Lundquist2018,Lindstrom2010,Carpenter2010,Gong2011,Nordstrom2013,Sjogreen2013,Ghasemi2017}. In contrast to these monolithic formulations, in the present work we introduce a partitioned strategy within the SBP-SAT framework and provide an energy-based stability analysis.

In this work, we focus on a model transient CHT problem, which involves solving the heat conduction equation in solids coupled with the convective heat transfer equation for fluids. An example of such a problem is the design of an electronic cooling system \cite{Zeng2020,Sun2024}. 
The SBP-SAT discretization employed in this work is specifically designed for general curvilinear grids, enabling flexibility and applicability to complex geometries typically encountered in practical engineering applications.  We propose a provably conditionally-energy-stable partitioned scheme that utilizes high-order spatial SBP differentiation operators in curvilinear coordinates. The interface conditions are enforced using a combination of SATs, ensuring the conditional stability of the scheme. We show that the proposed scheme, employing first-order backward Euler time discretization, is energy stable for a range of penalty parameters related to the SATs at the interface, along with a CFL-like relationship between the spatial and temporal grid spacings. Additionally, we demonstrate stability when a second-order extrapolation in time is applied at the interface. We further validate the conditional stability using a second-order time discretization. The proposed scheme is weakly coupled, applying a fixed number of iterations at each time level, independent of any convergence criteria. However, to improve the accuracy of the method, we observe from numerical results that sub-iterations are necessary to achieve the same error as a monolithic scheme. The main novelty of this work is the combination of a curvilinear SBP--SAT construction with a provably conditionally stable partitioned coupling for the CHT problem. In particular, the main contributions of this article are:

\begin{itemize}
	\item[$-$] We establish conditional energy stability for a weakly coupled partitioned scheme using high-order finite difference SBP operator for spatial derivatives, first- and second-order method for temporal derivatives and first- and second order extrapolation in time at the interface. 
	
	\item[$-$] SATs are carefully designed at the interface to ensure stability within the partitioned framework. This novel application is distinct from prior works focused on monolithic SBP-SAT schemes in multi-physics application. Moreover, we provide  a step-wise approach for choosing SAT parameters that result in stability.
	
	\item[$-$] SBP operators are constructed in curvilinear coordinate system, extending their applicability to complex geometries and domains in multi-physics applications.
\end{itemize}

An outline of the remainder of this paper is as follows. In Section \ref{sec:CHT}, we present the model CHT problem, discuss its well-posedness and provide a continuous stability estimate. In Section~\ref{sec:1dCart}, we develop the SBP--SAT discretization in a one-dimensional (1D) Cartesian setting. We describe both monolithic and partitioned interface couplings, and we derive discrete energy estimates that yield conditional stability for the corresponding 1D schemes. In Section~\ref{sec:3dCurv}, we extend the SBP--SAT framework to three-dimensional (3D) curvilinear grids and formulate the corresponding partitioned scheme. We derive conditional discrete stability estimates and discuss higher-order extensions in time for temporal discretization and extrapolation at interface. Section~\ref{sec:Numeric} offers some numerical experiments that validate our theoretical results in a two-dimensional (2D) spatial domain. Finally, in Section~\ref{sec:Summary}, we draw conclusions. For readability, several technical proofs are deferred to the appendices.

\section{The continuous system and well-posedness}\label{sec:CHT}

Assume $\Omega_L,\, \Omega_R\subset\Rs^d$,  where $d=1, 2, 3$,  denote the  two domains  under consideration, with boundaries $\partial \Omega_L$ and $\partial\Omega_R$. Let $\n_L$ and $\n_R$ denote the outward unit normal vectors for $ \partial \Omega_L$ and $\partial \Omega_R$, respectively.  Let $\Sigma:=\partial \Omega_L\cap\partial\Omega_R$ denote the common interface between the two domains, where the two PDEs are interacting. 
Figure \eqref{diag:domain} presents the 2D domain that we consider for a model CHT problem.

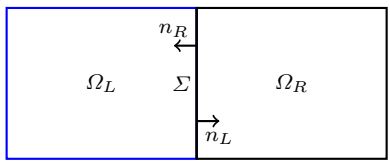
\begin{figure}[ht]
	\centering
	\begin{tikzpicture}[scale=0.5] 
		\draw[blue, thick] (-5, -2) rectangle (0, 2); 
		\draw[black, thick] (0, -2) rectangle (5, 2); 
		\draw[thick, black] (0, -2) -- (0, 2);
		
		\node at (-2.5, 0) { $\Omega_{L}$};
		\node at (2.5, 0) { $\Omega_{R}$};  
		
		\node at (-5.4, 0) {}; 
		\node at (6, 0) {};  
		
		\draw[->, thick] (0, 1) -- (-0.6, 1);
		\node at (-0.6, 1.4) {$n_R$};
		
		\draw[->, thick] (0, -1) -- (0.6, -1); 
		\node at (0.6, -1.45) {$n_L$};
		
		\node at (-0.4, 0) { $\Sigma$}; 
	\end{tikzpicture}
	\caption{Domain representation for the model CHT problem in 2D.}
	\label{diag:domain}
\end{figure}
Our model CHT problem consists of a linear advection-diffusion equation, in the left subdomain, coupled with a heat equation in the right subdomain. The advection term is included only in the left subdomain to model a fluid (convection--diffusion) region, while the right subdomain models a solid (pure conduction) region. This choice allows us to assess the impact of convection on the interface coupling in a simplified CHT setting.  These equations can be used to model transient of heat  between a fluid flow and solid-heated plates in electronic cooling system or air-cooled  heat sinks \cite{Zeng2020,Sun2024}. 

The left problem is given as   
\begin{align}\label{eq03a}
	\begin{aligned}
		\fnc{W}_t+\a\cdot \nabla \fnc{W}&=\epsilon \Delta \fnc{W}, &\quad &\bm{x}:=(x_1,\dots,x_d)\in \Omega _L,\quad 0\le t\le T,\\
		\zeta \fnc{W}+\epsilon \n_L\cdot \nabla \fnc{W}&=g(\bm{x}, t),&\quad &\bm{x}\in\partial \Omega_L\backslash\Sigma,\quad 0\le t\le T,\\\fnc{W}(\bm{x},0)&=q_L(\bm{x}),&\qquad &\bm{x}\in \Omega_L,
	\end{aligned}
\end{align}
where $\fnc{W}(\cdot, t)\in H^2(\Omega_L)$  for each $t\in[0,T]$ is the temperate in the fluid and $\a\in\Rs^d$ is a non-zero given constant advective velocity. Furthermore, $\epsilon>0$ is the thermal conductivity of fluid  and $\zeta$ is a constant. Moreover, $g(\bm{x},t)$ is the given boundary function, and  $q_L(\bm{x})$ is the given initial condition. On  the right subdomain we have that
\begin{align}\label{eq03b}
	\begin{aligned}
		\fnc{V}_t&=\kappa \Delta \fnc{V},&\quad &\bm{x}\in \Omega _R,\quad 0\le t\le T,\\
		\fnc{V}+\varphi \n_R\cdot \nabla \fnc{V}&=h(\bm{x},t),&\qquad &\bm{x}\in\partial \Omega_R\backslash\Sigma,\quad 0\le t\le T,\\
		\fnc{V}(\bm{x},0)&=q_R(\bm{x}),&\qquad &\bm{x}\in \Omega_R,
	\end{aligned}
\end{align}
where $\fnc{V}(\cdot, t)\in H^2(\Omega_L)$  is the temperature of the solid for each $t\in[0,T]$, $\kappa>0$ is the thermal conductivity of solid, and $h(\bm{x})$ is the boundary function, $q_R(\bm{x})$ is the initial condition and $\varphi$ is a given non-zero constant. We assume that these PDEs have been non dimensionalized.  These PDEs are coupled at $\Sigma$ with interface conditions:
\begin{align}\label{eq03c}
	\begin{aligned}
		\fnc{W}&=\fnc{V},\qquad\qquad\quad\;\quad\mbox{on}\;\, \Sigma,\\\epsilon \,\n_L\cdot\nabla \fnc{W}&=-\kappa\, \n_R\cdot\nabla \fnc{V},\qquad\mbox{on}\;\, \Sigma.
	\end{aligned}
\end{align}

To ensure a meaningful physical interpretation of the interface condition,
the analysis is restricted to purely tangential advection, i.e.\ it is assumed that
$\a\cdot\n_{L}=0$ at the interface. This assumption is required for the continuous
energy estimate, since it 
eliminates the interface flux term in the continuous energy identity. 
If instead $\a\cdot\n_{L}\neq 0$, the interface becomes direction-dependent (inflow/outflow), and a stable coupling generally requires a characteristic (upwind) interface treatment.  For further discussion of interface procedures for coupling advection--diffusion and heat equations in one dimension, we refer to \cite{Lindstrom2010,Gong2011}.

For a given $\Omega\in\Rs^d$, define the inner product   $(\fnc{U},\fnc{V})_\Omega:=\int_{\Omega} \fnc{U}\, \fnc{V} \,dx$  for $\fnc{U}$ and $\fnc{V}\in L^2(\Omega)$ and the associated $L^2$-norm $\|\fnc{U}\|_{\Omega}^2:=\int_{\Omega} \fnc{U}^2 dx$.  It can be shown that the PDEs on the left and right subdomains, together with the coupling conditions, yield a well-posed problem satisfying the differential inequality
\[
\frac{d}{dt}E(t)+\epsilon\|\nabla \fnc{W}(t)\|^2_{\Omega_L}+\kappa\|\nabla \fnc{V}(t)\|^2_{\Omega_R}
\le \mathcal{B}(t),
\]
where the energy is defined by
\[
E(t):=\tfrac12\|\fnc{W}(t)\|^2_{\Omega_L}+\tfrac12\|\fnc{V}(t)\|^2_{\Omega_R},
\]
and the (time-dependent) boundary-data term is
\[
\mathcal{B}(t):=\frac{1}{4\alpha}\|\mathcal{G}(t)\|^2_{\partial\Omega_L\setminus\Sigma}
+\frac{\beta}{4}\|\mathcal{H}(t)\|^2_{\partial\Omega_R\setminus\Sigma},
\qquad
\|\mathcal{G}(t)\|^2_{\partial\Omega_L\setminus\Sigma}:=\int_{\partial\Omega_L\setminus\Sigma} g(x,t)^2\,ds,
\]
and similarly for $\|\mathcal{H}(t)\|^2_{\partial\Omega_R\setminus\Sigma}$. Here we assume that $\a\cdot\n_L= 0$ on the interface, $\alpha=\big(\dfrac12\a\cdot \n_L+\zeta\big)>0$ and $\beta=\dfrac{\kappa}{\varphi}>0$. Therefore, for any $\zeta>-\dfrac12\a\cdot \n_L$ and $\varphi>0$, we obtain the energy estimate
\[
E(t)\le E(0)+\int_0^t \mathcal{B}(\tau)\,d\tau,
\qquad
E(0)=\tfrac12\|q_L\|^2_{\Omega_L}+\tfrac12\|q_R\|^2_{\Omega_R}.
\]
For simplicity in exposition, we assume $\zeta=\dfrac12\big(|\a\cdot \n_L|-\a\cdot \n_L\big)$ and $\varphi=\kappa$. See Appendix~\ref{apdA} for more details.

\textit{\emph{\bf Note 1.}} The unweighted choice in the definition of the energy $E(t)$ (equal weights on $\Omega_L$ and $\Omega_R$) is essential for the energy argument, since it yields exact cancellation of the interface contributions. If one instead considers a weighted energy
\[
E_{\theta}(t)=\theta_L\|\fnc{W}\|^2_{\Omega_L}+\theta_R\|\fnc{V}\|^2_{\Omega_R},\qquad 
\theta_L,\theta_R>0,\qquad \theta_L+\theta_R=1,
\]
then the corresponding energy identity contains an additional interface term proportional to $\theta_L-\theta_R$, which does not vanish unless $\theta_L=\theta_R$. Consequently, the weighted energy analysis may fail to yield a coercive stability estimate unless the interface coupling is modified.

\section{1D Cartesian SBP–SAT formulation, monolithic and partitioned interface coupling}\label{sec:1dCart}
The SBP--SAT framework provides a foundation for designing high-order accurate, conservative, and stable numerical schemes for PDEs. In this section, we present a full discretization of the coupled CHT model problem \eqref{eq03a}--\eqref{eq03c} in a 1D Cartesian setting.  A key contribution of this work is the design of interface SAT terms that lead to a  conditionally stable partitioned coupling. Working in 1D Cartesian coordinates enables us to present the scheme and the stability argument in their simplest form, making the interface SAT mechanism and the role of the penalty parameters transparent without additional geometric notation.

We proceed as follows. In Section~\ref{subsec:alg1D} we construct the monolithic and partitioned SBP--SAT discretizations for the model CHT problem and define the corresponding boundary and interface SAT terms. In Section~\ref{subsec:Sta1D} we prove discrete stability estimates for both couplings. The results in this section serve as the template for the higher-dimensional formulation in complex geometry and analysis presented later in the paper.

\subsection{1D discretization and interface coupling using SBP-SATs}\label{subsec:alg1D}
To obtain a fully discrete scheme for the time-dependent PDEs, we employ SBP differentiation operators to discretize the spatial derivatives, and a first-order backward finite-difference scheme for the time discretization.
Assume $x^{(1)},x^{(2)},\ldots, x^{(N)}$ are grid points defined on the 1D domain
$\Omega=(\alpha,\beta)$, where $N$ is the number of nodes and $x^{(1)}=\alpha$ and $x^{(N)}=\beta$. 
The definition of a 1D SBP operator $\DoneD{x}$ approximating the first derivative is defined as follows,
\begin{definition}\label{Def:SBP1D}
A matrix operator $\DoneD{x}\in\Rs^{N\times N}$ is a  SBP operator of degree $p$
approximating the first derivative $\frac{\partial}{\partial x}$ on $x\in[\alpha,\beta]$ with vectors of nodes $\x=[x^{(1)},\dots,x^{(N)}]^T$, if
\begin{enumerate}
  \item  
  $\DoneD{x}\,\x^{\,r}=r\,\x^{\,r-1}$ for $r=1,2,\ldots,p$, where
  $\x^{\,r}:=[(x^{(1)})^r,\dots,(x^{(N)})^r]^T$, and $\DoneD{x}\,\bm{1}=\bm{0}$, where $\one$ is a vector of ones and  $\zero $ is a vector of zeroes of  size of the number of nodes in the domain.
  \item  $\DoneD{x}=\PoneD{x}^{-1}\QoneD{x}$ with $\PoneD{x}$ symmetric positive definite;
  \item  $\QoneD{x}+\QoneD{x}^T=\EoneD{x}$, where
  \[
    \EoneD{x}=\diag(-1,0,\ldots,0,1)
    =\e_{\beta}^T\e_{\beta}-\e_{\alpha}^T\e_{\alpha},
  \]
with $\e_{\alpha},\e_{\beta}\in\Rs^{1\times N}$ given by
$\e_{\alpha}=[1,0,\ldots,0]$ and $\e_{\beta}=[0,\ldots,0,1]$.

\end{enumerate}
\end{definition}
Thus, a degree $p$ SBP operator exactly differentiates monomials up to degree $p$ on the nodal set \cite{DelRey2014}. Given symmetric positive definite matrix $\PoneD{x}$, we define
$\langle \x,\y\rangle_{\PoneD{x}}:=\x\Tr \PoneD{x}\,\y$ and the associated squared $\PoneD{x}$-norm by
$\|\x\|_{\PoneD{x}}^2:=\x\Tr \PoneD{x}\,\x$.

For any $t\in\Rs$, assume $\fnc{U}(\cdot,t),\fnc{V}(\cdot,t)\in H^{1}(\Omega)$. 
The continuous integration-by-parts identity in one dimension reads
\begin{align*}
\begin{aligned}
\int_{\alpha}^{\beta}\fnc{U}(x)\,\frac{\partial \fnc{V}}{\partial x}(x)\,dx
+\int_{\alpha}^{\beta}\fnc{V}(x)\,\frac{\partial \fnc{U}}{\partial x}(x)\,dx
=
\fnc{U}(\beta)\fnc{V}(\beta)-\fnc{U}(\alpha)\fnc{V}(\alpha).
\end{aligned}
\end{align*}
Define the grid function $\u$ on the $N$ nodes by
\[
\u(t) := \big[\fnc{U}(x^{(1)},t),\,\fnc{U}(x^{(2)},t),\,\ldots,\,\fnc{U}(x^{(N)},t)\big]\Tr,
\]
and define $\v(t)$ in the same way from $\fnc{V}(\cdot,t)$. 
SBP operators mimic integration by parts. To see this at the discrete level, note that discretizing the left-hand side using $\PoneD{}$ and $\DoneD{x}$ yields
\begin{align*}
\begin{aligned}
\u\Tr\PoneD{}\DoneD{x}\v+\v\Tr\PoneD{}\DoneD{x}\u
=\u\Tr\QoneD{x}\v+\v\Tr\QoneD{x}\u=\u\Tr\QoneD{x}\v+\u\Tr\QoneD{x}\Tr\v=\u\Tr\big(\QoneD{x}+\QoneD{x}\Tr\big)\v
=\u\Tr\EoneD{x}\v,
\end{aligned}
\end{align*}

For the time discretization in multi-physics applications, implicit methods like the backward Euler method can be advantageous due to their broader stability region compared to explicit methods. This advantage ensures more reliable and stable simulations of the interactions between fluid dynamics and structural responses, see \cite{Meng2017}  and references therein for more detail. Let $\{t_k\}$ denote the grid points for time and assume an equal time step for all 
$0\le k\le N_t$, given by $\delta t=t_{k+1}-t_k$.  
A first-order backward finite-difference scheme (Backward Euler, BE) is given by
\[
\frac{d\u}{dt}(t_{k+1})\approx\partial_{\delta t}\u^{k+1}:=\frac{\u^{k+1}-\u^k}{\delta t},
\]
where $\u^{k+1}:=\u(t_{k+1})$ and $\u^{k}:=\u(t_{k})$. 

Next we introduce the fully discretized scheme for the coupled PDEs. We consider the 1D Cartesian setting of the equations \eqref{eq03a}-\eqref{eq03c} on the interval $[\alpha,\beta]$, split into two subdomains
$\Omega_L:=(\alpha,\sigma)$ and $\Omega_R:=(\sigma,\beta)$, where $\alpha<\sigma<\beta$.
The (constant) advection constant satisfies $a\ge 0$, and we define
\(c:=\zeta+\frac{a}{2},\) with \(\zeta\) chosen so that $c>0$. 
For simplicity, we assume that the left and right subdomains are discretized by 1D grids
$\{x_{L}^{(i)}\}_{i=1}^N \subset \Omega_{L}$ and $\{x_{ R}^{(i)}\}_{i=1}^N \subset \Omega_{ R}$,
using the same differentiation matrix $\DoneD{x}$ and associated weight matrix $\PoneD{x}$. 
With this convention, the boundary selector vectors $\e_{\alpha}=[1,0,\ldots,0]$ and $\e_{\beta}=[0,\ldots,0,1]$ can be used consistently on both subdomains to extract the boundary points. The boundary data in \eqref{eq03a}--\eqref{eq03b} are denoted by $g(x,t)$ at $x=\alpha$ and $h(x,t)$ at $x=\beta$.
In the fully discrete scheme, we store the boundary data as grid vectors
\[
\g^{i}:=\big[g(x_L^{(1)}, t_{i}),\ldots, g(x_L^{(N)}, t_{i})\big]^T,
\qquad
\h^{i}:=\big[h(x_R^{(1)}, t_{i}),\ldots, h(x_R^{(N)}, t_{i})\big]^T,\qquad i=0,\dots, N_t
\]
Hence, $\e_{\alpha}\g^i$ and $\e_{\beta}\h^i$ extract the boundary entries
$g(x_L^{(1)},t_i)$ and $h(x_R^{(N)},t_i)$, respectively.
The initial data are given by $q_L(x)$ on $\Omega_L$ and $q_R(x)$ on $\Omega_R$, and we set the initial grid vectors
\[
\q_{\mat L}:=\big[\fnc{W}(x_L^{(1)},0),\ldots,\fnc{W}(x_L^{(N)},0)\big]^T,\qquad
\q_{\mat R}:=\big[\fnc{V}(x_R^{(1)},0),\ldots,\fnc{V}(x_R^{(N)},0)\big]^T.
\]

Then the coupled SBP-SAT scheme is defined as follow. For given initial guess for $\v^*$ and $\w^*$ and  given SAT parameters $\gamma_1,\gamma_2\ge0$, the following iterative procedure is applied:
\begin{itemize}
\item solve the left sub-problem: given $\v^*$, find $\w^{k+1}$ such that
\begin{align}\label{eq82}
\begin{aligned}
\partial_{\delta t} \w^{k+1}
+a\DoneD{x}\w^{k+1}
+\gamma_1\,{\mat S}^{{\mat L},1}_{\sigma}
+\gamma_2\,{\mat S}^{{\mat L},2}_{\sigma}
&=\epsilon\,\DoneD{x}^2\w^{k+1}+{\mat S}_{\alpha}^{\mat L}.
\end{aligned}
\end{align}
where ${\mat S}_{\alpha}^{\mat L}$ weakly imposes the physical boundary condition in~\eqref{eq03a}:
 \begin{align*}
\begin{aligned}
{\mat S}_{\alpha}^{\mat L}
=-\PoneD{}^{-1}\e_{\alpha}\Tr\e_{\alpha}\big(\zeta\,\w^{k+1}-\epsilon\, \DoneD{x}\w^{k+1}- \g^{k+1}\big),
\end{aligned}
\end{align*}
and ${\mat S}^{{\mat L},1}_{\sigma}$ and ${\mat S}^{{\mat L},2}_{\sigma}$ weakly impose interface conditions~\eqref{eq03c},
\begin{align*}
\begin{aligned}
{\mat S}^{{\mat L},1}_{\sigma}
&=\PoneD{}^{-1}\e_{\beta}\Tr\big(\e_{\beta}\w^{k+1}-\e_{\alpha}\v^{*}\big),\\
{\mat S}^{{\mat L},2}_{\sigma}
&=\epsilon\,\PoneD{}^{-1}\DoneD{x}\Tr\e_{\beta}\Tr\big(\epsilon\, \e_{\beta}\DoneD{x}\w^{k+1}-\kappa\, \e_{\alpha}\DoneD{x}\v^*\big).
\end{aligned}
\end{align*}  
\item update the interface data $\w^*$.
\item solve the right sub-problem: given $\w^*$, find $\v^{k+1}$ such that
\begin{align}\label{eq83}
\begin{aligned}
\partial_{\delta t} \v^{k+1}
+\gamma_1\,{\mat S}^{{\mat R},1}_{\sigma}
+\gamma_2\,{\mat S}^{{\mat R},2}_{\sigma}
+{\mat S}^{{\mat R},3}_{\sigma}
&=\kappa\, \DoneD{x}^2\v^{k+1}+{\mat S}_{\beta}^{\mat R}.
\end{aligned}
\end{align}
where  ${\mat S}_{\beta}^{\mat R}$ weakly imposes the  physical boundary condition in~\eqref{eq03b}:
\begin{align*}
\begin{aligned}
{\mat S}_{\beta}^{\mat R}
=-\PoneD{}^{-1}\e_{\beta}\Tr\e_{\beta}\big(\v^{k+1}+\kappa\, \DoneD{x}\v^{k+1}-\h^{k+1}\big),
\end{aligned}
\end{align*}
and ${\mat S}^{{\mat R},1}_{\sigma}$, ${\mat S}^{{\mat R},2}_{\sigma}$, and ${\mat S}^{{\mat R},3}_{\sigma}$ weakly impose interface conditions~\eqref{eq03c},
\begin{align*}
\begin{aligned}
{\mat S}^{{\mat R},1}_{\sigma}
&=\PoneD{}^{-1}\e_{\alpha}\Tr\big(\e_{\alpha}\v^{k+1}-\e_{\beta}\w^*\big),\\
{\mat S}^{{\mat R},2}_{\sigma}
&=\kappa\,\PoneD{}^{-1}\DoneD{x}^T\e_{\alpha}\Tr\big(\kappa\, \e_{\alpha}\DoneD{x}\v^{k+1}-\epsilon\,\e_{\beta}\DoneD{x}\w^*\big),\\
{\mat S}^{{\mat R},3}_{\sigma}
&=\PoneD{}^{-1}\e_{\alpha}\Tr\big(\epsilon\e_{\beta}\DoneD{x}\w^*-\kappa\e_{\alpha}\DoneD{x}\v^{k+1}\big).
\end{aligned}
\end{align*}
\item update the interface data $\v^*$.
\end{itemize}
Note that the SATs and penalty parameters  at the interface are precisely defined to ensure conditional stability in the partitioned scheme. If we set $\v^* := \v^{k+1}$ and $\w^* := \w^{k+1}$, the scheme corresponds to the monolithic approach; on the other hand, if we set for example $\v^* := \v^{k}$ and $\w^* := \w^{k+1}$, the scheme corresponds to the partitioned approach (without extrapolation in time at interface).

\subsection{Energy stability of the discrete 1D monolithic and partitioned scheme}\label{subsec:Sta1D}
In the subsequent stability proof in 1D the case, we use the following local discrete trace inequality.
\begin{lem}[Discrete trace inequality in 1D]\label{lemTrace1D}
Let $\,\PoneD{}=\diag(\hat p_1,\ldots,\hat p_{N_x})$ with $\hat p_i>0$.
Define\, \[\rho:=\min_{1\le i\le N_x}\{\hat p_i\}. \]
Then for any grid vector $\z\in\mathbb{R}^{N_x}$,
\begin{align}\label{lemTrace1D:eq}
(\e_{\alpha}\z)^2 \le \frac{1}{\rho}\,\|\z\|^2_{\PoneD{}},
\qquad
(\e_{\beta}\z)^2 \le \frac{1}{\rho}\,\|\z\|^2_{\PoneD{}},
\end{align}
where $\|\z\|^2_{\PoneD{}}:=\z\Tr\PoneD{}\z$.
\end{lem}
\begin{proof}
Since $\PoneD{}$ is diagonal and positive,
\[
\|\z\|^2_{\PoneD{}}=\sum_{i=1}^{N_x}\hat p_i z_i^2
\ge \rho \sum_{i=1}^{N_x} z_i^2
\ge \rho\, z_1^2
=\rho\,(\e_{\alpha}\z)^2,
\]
which gives $(\e_{\alpha}\z)^2 \le \dfrac{1}{\rho}\|\z\|^2_{\PoneD{}}$. The right boundary estimate follows analogously.
\end{proof}
Notice  that for diagonal-norm SBP operators on a uniform grid with spacing $h$, the minimum of quadrature weights satisfy
$\rho=\mathcal{O}(h)$. In the following, we establish discrete stability for the monolithic coupling.

\begin{thm}[Stability of the monolithic 1D scheme]\label{thm:1D_mono}
Assume  $2\gamma_1\rho\ge\epsilon$.
Consider the scheme \eqref{eq82}--\eqref{eq83} with the \emph{monolithic} choice
$\v^*=\v^{k+1}$ and $\w^*=\w^{k+1}$.
Then, for all $k\ge 0$,
\begin{align*}
\begin{aligned}
\|\w^{k+1}\|^2_{\PoneD{}}
+\|\v^{k+1}\|^2_{\PoneD{}}
\le\;
\|\q_{\mat L}\|^2_{\PoneD{}}
+\|\q_{\mat R}\|^2_{\PoneD{}}
+\sum_{i=1}^{k+1}\left(
\frac{\delta t}{2c}\,(\e_{\alpha}\g^{i})^2
+\frac{\delta t}{2}\,(\e_{\beta}\h^{i})^2\right).
\end{aligned}
\end{align*}
\end{thm}
\begin{proof} 
Left-multiply both side of equation \eqref{eq82} by ${\w^{k+1}}\Tr \PoneD{}$ to obtain
\begin{align}\label{d1}
\begin{aligned}
{\w^{k+1}}\Tr \PoneD{}\,\partial_{\delta t}\w^{k+1}
+a\,{\w^{k+1}}\Tr \PoneD{}\DoneD{x}\w^{k+1}
+\gamma_1\,{\w^{k+1}}\Tr \PoneD{}{\mat S}^{{\mat L},1}_{\sigma}
+\gamma_2\,{\w^{k+1}}\Tr \PoneD{}{\mat S}^{{\mat L},2}_{\sigma} \\
=\epsilon\,{\w^{k+1}}\Tr \PoneD{}\DoneD{x}^2\w^{k+1}
+{\w^{k+1}}\Tr \PoneD{}{\mat S}_{\alpha}.
\end{aligned}
\end{align}
Notice that
\begin{align*}
\begin{aligned}
{\w^{k+1}}\Tr\PoneD{}\partial_{\delta t} \w^{k+1}
=\frac{1}{\delta t}\,{\w^{k+1}}\Tr\PoneD{}(\w^{k+1}-\w^{k})
=\frac{1}{\delta t}\|\w^{k+1}\|^2_{\PoneD{}}
-\frac{1}{\delta t}\,{\w^{k+1}}\Tr\PoneD{}\w^{k}.
\end{aligned}
\end{align*}
Moreover,
\begin{align*}
\begin{aligned}
{\w^{k+1}}\Tr\PoneD{}\w^{k}
=\frac12\Big(\|\w^{k+1}\|^2_{\PoneD{}}+\|\w^{k}\|^2_{\PoneD{}}
-\|\w^{k+1}-\w^{k}\|^2_{\PoneD{}}\Big)
\le \frac12\|\w^{k+1}\|^2_{\PoneD{}}+\frac12\|\w^{k}\|^2_{\PoneD{}}.
\end{aligned}
\end{align*}
Therefore,
\begin{align*}
\begin{aligned}
{\w^{k+1}}\Tr\PoneD{}\partial_{\delta t} \w^{k+1}
&\ge \frac{1}{2\delta t}\|\w^{k+1}\|^2_{\PoneD{}}-\frac{1}{2\delta t}\|\w^{k}\|^2_{\PoneD{}}.
\end{aligned}
\end{align*}
Apply this in \eqref{d1}
\begin{align}\label{d2}
\begin{aligned}
\frac1{2\delta t}\|\w^{k+1}\|^2_{\PoneD{}}
&+a{\w^{k+1}}\Tr\PoneD{}\DoneD{x} \w^{k+1}
+\gamma_1\,{\w^{k+1}}\Tr\PoneD{}{\mat S}^{{\mat L},1}_{\sigma}
+\gamma_2\,{\w^{k+1}}\Tr\PoneD{}{\mat S}^{{\mat L},2}_{\sigma}
\\&\le \frac1{2\delta t}\|\w^{k}\|^2_{\PoneD{}}+\epsilon{\w^{k+1}}\Tr\PoneD{}\DoneD{x}^2\w^{k+1}
+{\w^{k+1}}\Tr\,\PoneD{}{\mat S}_{\alpha}^{\mat L}.
\end{aligned}
\end{align}
Also
\begin{align}\label{d3}
\begin{aligned}
a{\w^{k+1}}^T\PoneD{}\DoneD{x}\w^{k+1}
&=\dfrac{a}2{\w^{k+1}}^T(\QoneD{x}+\QoneD{x}\Tr)\w^{k+1}=\dfrac{a}2{\w^{k+1}}^T\e_{\beta}\Tr\e_{\beta}\w^{k+1}
-\dfrac{a}2{\w^{k+1}}^T\e_{\alpha}\Tr\e_{\alpha}\w^{k+1}.
\end{aligned}
\end{align}
Additionally, it can be shown that 
\(
\PoneD{}\DoneD{x}^2=\EoneD{x}\DoneD{x}-\DoneD{x}\Tr\PoneD{}\DoneD{x},
\)
therefore
\begin{align}\label{d4}
\begin{aligned}
\epsilon{\w^{k+1}}\Tr\PoneD{}\DoneD{x}^2 \w^{k+1}
&=\epsilon{\w^{k+1}}\Tr\Big(\EoneD{x}\DoneD{x}-\DoneD{x}\Tr\PoneD{}\DoneD{x}\Big)\w^{k+1}\\
&=\epsilon{\w^{k+1}}\Tr\e_{\beta}\Tr\e_{\beta}\DoneD{x}\w^{k+1}
-\epsilon{\w^{k+1}}\Tr\e_{\alpha}\Tr\e_{\alpha}\DoneD{x}\w^{k+1}
-\epsilon\|\DoneD{x}\w^{k+1}\|^2_{\PoneD{}}.
\end{aligned}
\end{align}
Substitute \eqref{d3} and \eqref{d4} in \eqref{d2} to obtain
\begin{align}\label{d5}
\begin{aligned}
\frac1{2\delta t}\|\w^{k+1}\|^2_{\PoneD{}}
&+\epsilon \|\DoneD{x}\w^{k+1}\|^2_{\PoneD{}}
+\gamma_1\,{\w^{k+1}}\Tr\PoneD{}{\mat S}^{{\mat L},1}_{\sigma}
+\gamma_2\,{\w^{k+1}}\Tr\PoneD{}{\mat S}^{{\mat L},2}_{\sigma}\\
&\;\le\;
\frac1{2\delta t}\|\w^{k}\|^2_{\PoneD{}}
-\frac{a}2{\w^{k+1}}\Tr\e_{\beta}\Tr\e_{\beta}\w^{k+1}
+\frac{a}2{\w^{k+1}}\Tr\e_{\alpha}\Tr\e_{\alpha}\w^{k+1}\\
&\quad+\epsilon{\w^{k+1}}\Tr\e_{\beta}\Tr\e_{\beta}\DoneD{x}\w^{k+1}
-\epsilon{\w^{k+1}}\Tr\e_{\alpha}\Tr\e_{\alpha}\DoneD{x}\w^{k+1}
+{\w^{k+1}}\Tr\,\PoneD{}{\mat S}_{\alpha}^{\mat L}.
\end{aligned}
\end{align}
By direct calculation,
\begin{align}\label{d6}
\begin{aligned}
\frac{a}{2}\,{\w^{k+1}}\Tr\e_{\alpha}\Tr\e_{\alpha}\w^{k+1}
&-\epsilon\,{\w^{k+1}}\Tr\e_{\alpha}\Tr\e_{\alpha}\DoneD{x}\w^{k+1}
+{\w^{k+1}}\Tr\,\PoneD{}{\mat S}_{\alpha}^{\mat L} \\
&=
{\w^{k+1}}\Tr\e_{\alpha}\Tr\e_{\alpha}\g^{k+1}
-c\,{\w^{k+1}}\Tr\e_{\alpha}\Tr\e_{\alpha}\w^{k+1}\le \frac{1}{4c}\,(\e_{\alpha}\g^{k+1})^2
\end{aligned}
\end{align}
Substituting the estimate \eqref{d6} into \eqref{d5} and multiplying by $2\delta t$ gives
\begin{align*}
\begin{aligned}
\|\w^{k+1}\|^2_{\PoneD{}}
&+2\delta t\,\epsilon \|\DoneD{x}\w^{k+1}\|^2_{\PoneD{}}
+2\delta t\,\gamma_1\,{\w^{k+1}}\Tr\PoneD{}{\mat S}^{{\mat L},1}_{\sigma}
+2\delta t\,\gamma_2\,{\w^{k+1}}\Tr\PoneD{}{\mat S}^{{\mat L},2}_{\sigma}\\
\le&
\|\w^{k}\|^2_{\PoneD{}}
-\delta t\,a\,{\w^{k+1}}\Tr\e_{\beta}\Tr\e_{\beta}\w^{k+1}
+2\delta t\,\epsilon{\w^{k+1}}\Tr\e_{\beta}\Tr\e_{\beta}\DoneD{x}\w^{k+1}
+\frac{\delta t}{2c}\,(\e_{\alpha}\g^{k+1})^2.
\end{aligned}
\end{align*}
Notice that we could cancel the advection boundary contribution at interface, since it is negative,
\begin{align}\label{d7}
\begin{aligned}
\|\w^{k+1}\|^2_{\PoneD{}}
&+2\delta t\,\epsilon \|\DoneD{x}\w^{k+1}\|^2_{\PoneD{}}
+2\delta t\,\gamma_1\,{\w^{k+1}}\Tr\PoneD{}{\mat S}^{{\mat L},1}_{\sigma}
+2\delta t\,\gamma_2\,{\w^{k+1}}\Tr\PoneD{}{\mat S}^{{\mat L},2}_{\sigma}\\
\le&
\|\w^{k}\|^2_{\PoneD{}}
+2\delta t\,\epsilon{\w^{k+1}}\Tr\e_{\beta}\Tr\e_{\beta}\DoneD{x}\w^{k+1}
+\frac{\delta t}{2c}\,(\e_{\alpha}\g^{k+1})^2.
\end{aligned}
\end{align}
Next multiply both side of  Equation \eqref{eq83} by ${\v^{k+1}}\Tr\PoneD{}$, simplify to obtain
\begin{align*}
\begin{aligned}
\frac1{2\delta t}\|\v^{k+1}\|^2_{\PoneD{}}
&+\kappa\,\|\DoneD{x}\v^{k+1}\|^2_{\PoneD{}}
+\gamma_1\,{\v^{k+1}}\Tr\PoneD{}\,{\mat S}_{\sigma}^{{\mat R},1}
+\gamma_2\,{\v^{k+1}}\Tr\PoneD{}\,{\mat S}_{\sigma}^{{\mat R},2}
+{\v^{k+1}}\Tr\PoneD{}\,{\mat S}_{\sigma}^{{\mat R},3} \\
\le&
\frac1{2\delta t}\|\v^{k}\|^2_{\PoneD{}}
+\kappa\,{\v^{k+1}}\Tr\e_{\beta}\Tr\e_{\beta}\DoneD{x}\v^{k+1}
-\kappa\,{\v^{k+1}}\Tr\e_{\alpha}\Tr\e_{\alpha}\DoneD{x}\v^{k+1}
+{\v^{k+1}}\Tr\PoneD{}\,{\mat S}_{\beta}^{\mat R}.
\end{aligned}
\end{align*}
Doing the same calculations for right boundary terms leads to
\begin{align}\label{d8}
\begin{aligned}
\|\v^{k+1}\|^2_{\PoneD{}}
&+2\delta t\,\kappa\,\|\DoneD{x}\v^{k+1}\|^2_{\PoneD{}}
+2\delta t\,\gamma_1\,{\v^{k+1}}\Tr\PoneD{}\,{\mat S}_{\sigma}^{{\mat R},1}
+2\delta t\,\gamma_2\,{\v^{k+1}}\Tr\PoneD{}\,{\mat S}_{\sigma}^{{\mat R},2}
+2\delta t\,{\v^{k+1}}\Tr\PoneD{}\,{\mat S}_{\sigma}^{{\mat R},3}\\
\le&
\|\v^{k}\|^2_{\PoneD{}}
-2\delta t\,\kappa\,{\v^{k+1}}\Tr\e_{\alpha}\Tr\e_{\alpha}\DoneD{x}\v^{k+1}
+\frac{\delta t}{2}(\e_{\beta}\h^{k+1})^2.
\end{aligned}
\end{align}
Adding \eqref{d7} and \eqref{d8} gives
\begin{align*}
\begin{aligned}
\|\w^{k+1}\|^2_{\PoneD{}}
&+\|\v^{k+1}\|^2_{\PoneD{}}
+2\delta t\,\epsilon \|\DoneD{x}\w^{k+1}\|^2_{\PoneD{}}
+2\delta t\,\kappa\,\|\DoneD{x}\v^{k+1}\|^2_{\PoneD{}}
+2\delta t\,{\v^{k+1}}\Tr\PoneD{}\,{\mat S}_{\sigma}^{{\mat R},3}\\
&+2\delta t\,\gamma_1\Big(
{\w^{k+1}}\Tr\PoneD{}{\mat S}^{{\mat L},1}_{\sigma}
+{\v^{k+1}}\Tr\PoneD{}\,{\mat S}_{\sigma}^{{\mat R},1}\Big)
+2\delta t\,\gamma_2\Big(
{\w^{k+1}}\Tr\PoneD{}{\mat S}^{{\mat L},2}_{\sigma}
+{\v^{k+1}}\Tr\PoneD{}\,{\mat S}_{\sigma}^{{\mat R},2}
\Big)
\\\le&
\|\w^{k}\|^2_{\PoneD{}}
+\|\v^{k}\|^2_{\PoneD{}}+\frac{\delta t}{2c}\,(\e_{\alpha}\g^{k+1})^2
+\frac{\delta t}{2}\,(\e_{\beta}\h^{k+1})^2\\
&
+2\delta t\,\epsilon{\w^{k+1}}\Tr\e_{\beta}\Tr\e_{\beta}\DoneD{x}\w^{k+1}
-2\delta t\,\kappa\,{\v^{k+1}}\Tr\e_{\alpha}\Tr\e_{\alpha}\DoneD{x}\v^{k+1}\\
\end{aligned}
\end{align*}
In the monolithic case ($\v^*=\v^{k+1}$ and $\w^*=\w^{k+1}$),
and this can be simplified to
\begin{align}\label{d10}
\begin{aligned}
\|\w^{k+1}\|^2_{\PoneD{}}
&+\|\v^{k+1}\|^2_{\PoneD{}}
+2\delta t\,\epsilon \|\DoneD{x}\w^{k+1}\|^2_{\PoneD{}}
+2\delta t\,\kappa \|\DoneD{x}\v^{k+1}\|^2_{\PoneD{}}\\
&+2\delta t\,\gamma_1\Big(
{\w^{k+1}}\Tr\PoneD{}{\mat S}^{{\mat L},1}_{\sigma}
+{\v^{k+1}}\Tr\PoneD{}\,{\mat S}_{\sigma}^{{\mat R},1}
\Big)
+2\delta t\,\gamma_2\Big(
{\w^{k+1}}\Tr\PoneD{}{\mat S}^{{\mat L},2}_{\sigma}
+{\v^{k+1}}\Tr\PoneD{}\,{\mat S}_{\sigma}^{{\mat R},2}
\Big)
\\\le&
\|\w^{k}\|^2_{\PoneD{}}
+\|\v^{k}\|^2_{\PoneD{}}
+\frac{\delta t}{2c}\,(\e_{\alpha}\g^{k+1})^2
+\frac{\delta t}{2}\,(\e_{\beta}\h^{k+1})^2\\
&
+2\delta t\,\epsilon(\e_{\beta}\DoneD{x}\w^{k+1})\Big(\e_{\beta}\w^{k+1}-\e_{\alpha}\v^{k+1}\Big).
\end{aligned}
\end{align}
By Cauchy--Schwarz and Young's inequality, for any $\rho>0$,
\begin{align*}
\begin{aligned}
2(\e_{\beta}\DoneD{x}\w^{k+1})\big(\e_{\beta}\w^{k+1}-\e_{\alpha}\v^{k+1}\big)
\le \frac{1}{\rho}\,\big(\e_{\beta}\w^{k+1}-\e_{\alpha}\v^{k+1}\big)^2
+\rho\,(\e_{\beta}\DoneD{x}\w^{k+1})^2.
\end{aligned}
\end{align*}
Applying Lemma~\ref{lemTrace1D} with $\z=\DoneD{x}\w^{k+1}$ yields,
\(
(\e_{\beta}\DoneD{x}\w^{k+1})^2 \le \dfrac{1}{\rho}\,\|\DoneD{x}\w^{k+1}\|^2_{\PoneD{}},
\) therefore,
\begin{align*}
\begin{aligned}
2(\e_{\beta}\DoneD{x}\w^{k+1})\big(\e_{\beta}\w^{k+1}-\e_{\alpha}\v^{k+1}\big)
\le\frac{1}{\rho}\,\big(\e_{\beta}\w^{k+1}-\e_{\alpha}\v^{k+1}\big)^2
+\|\DoneD{x}\w^{k+1}\|_{\PoneD{}}^2.
\end{aligned}
\end{align*}
Apply this in \eqref{d10} 
\begin{align}\label{d11}
\begin{aligned}
\|\w^{k+1}\|^2_{\PoneD{}}
&+\|\v^{k+1}\|^2_{\PoneD{}}
+\delta t\,\epsilon \|\DoneD{x}\w^{k+1}\|^2_{\PoneD{}}
+2\delta t\,\kappa \|\DoneD{x}\v^{k+1}\|^2_{\PoneD{}}\\
&+2\delta t\,\gamma_1\Big(
{\w^{k+1}}\Tr\PoneD{}{\mat S}^{{\mat L},1}_{\sigma}
+{\v^{k+1}}\Tr\PoneD{}\,{\mat S}_{\sigma}^{{\mat R},1}
\Big)
+2\delta t\,\gamma_2\Big(
{\w^{k+1}}\Tr\PoneD{}{\mat S}^{{\mat L},2}_{\sigma}
+{\v^{k+1}}\Tr\PoneD{}\,{\mat S}_{\sigma}^{{\mat R},2}
\Big)
\\\le&
\|\w^{k}\|^2_{\PoneD{}}
+\|\v^{k}\|^2_{\PoneD{}}
+\frac{\delta t}{2c}\,(\e_{\alpha}\g^{k+1})^2
+\frac{\delta t}{2}\,(\e_{\beta}\h^{k+1})^2
+\dfrac{\delta t\,\epsilon}{\rho}\big(\e_{\beta}\w^{k+1}-\e_{\alpha}\v^{k+1}\big)^2.
\end{aligned}
\end{align}
Use $\v^*=\v^{k+1}$ and $\w^*=\w^{k+1}$ in the remaining interface SATs and simplify to obtain:
\begin{align*}
\begin{aligned}
\|\w^{k+1}\|^2_{\PoneD{}}
&+\|\v^{k+1}\|^2_{\PoneD{}}
+\delta t\,\epsilon \|\DoneD{x}\w^{k+1}\|^2_{\PoneD{}}
+2\delta t\,\kappa \|\DoneD{x}\v^{k+1}\|^2_{\PoneD{}}\\
&+\delta t\,(2\gamma_1-\dfrac{\epsilon}{\rho})\Big(\e_{\beta}\w^{k+1}-\e_{\alpha}\v^{k+1}\Big)^2
+2\delta t\,\gamma_2\Big(\epsilon\,\e_{\beta}\DoneD{x}\w^{k+1}-\kappa\,\e_{\alpha}\DoneD{x}\v^{k+1}\Big)^2\\
\le&
\|\w^{k}\|^2_{\PoneD{}}
+\|\v^{k}\|^2_{\PoneD{}}
+\frac{\delta t}{2c}\,(\e_{\alpha}\g^{k+1})^2
+\frac{\delta t}{2}\,(\e_{\beta}\h^{k+1})^2.
\end{aligned}
\end{align*}
By assumption $2\gamma_1\rho\ge\epsilon$. Therefore this simplifies to
\begin{align*}
\begin{aligned}
\|\w^{k+1}\|^2_{\PoneD{}}
+\|\v^{k+1}\|^2_{\PoneD{}}
\le
\|\w^{k}\|^2_{\PoneD{}}
+\|\v^{k}\|^2_{\PoneD{}}
+\frac{\delta t}{2c}\,(\e_{\alpha}\g^{k+1})^2
+\frac{\delta t}{2}\,(\e_{\beta}\h^{k+1})^2.
\end{aligned}
\end{align*}
Applying this estimate recursively yields
\begin{align*}
\begin{aligned}
\|\w^{k+1}\|^2_{\PoneD{}}
+\|\v^{k+1}\|^2_{\PoneD{}}
\le
\|\q_{\mat L}\|^2_{\PoneD{}}+\|\q_{\mat R}\|_{\PoneD{}}^2
+\Sigma_{i=1}^{k+1}\left(\frac{\delta t}{2c}\,(\e_{\alpha}\g^{i})^2
+\frac{\delta t}{2}\,(\e_{\beta}\h^{i})^2\right).
\end{aligned}
\end{align*}
Notice that we used \(\|\w^0\|^2_{\PoneD{}}+\|\v^0\|^2_{\PoneD{}}=\|\q_{\mat L}\|^2_{\PoneD{}}+\|\q_{\mat R}\|_{\PoneD{}}^2.\)
\end{proof}
\textit{\emph{\bf Note 2.}}  Notice that the calculation shows the contributions of ${\mat S}^{{\mat L},2}_{\sigma}$ and ${\mat S}^{{\mat R},2}_{\sigma}$ are not necessary to obtain a stable coupled monolithic scheme. Therefore, we may set $\gamma_2=0$ in the monolithic case without loss of stability. In the next theorem, we derive the conditional stability conditions for the proposed partitioned coupling.
\begin{thm}[Stability of the partitioned 1D scheme]\label{thm:1D_part}
Consider the scheme \eqref{eq82}--\eqref{eq83} with the \emph{partitioned} choice
$\v^*=\v^{k}$ and $\w^*=\w^{k+1}$. 
Then, under the conditions 
	\begin{subequations}
		\begin{align}
			\dfrac{\epsilon}{\rho}\le\gamma_1&\le \dfrac{C_1}{\delta t},\label{da1}
			\\ \gamma_2&\le \dfrac{C_2}{\kappa^2\delta t},\label{da2}
		\end{align}
	\end{subequations}
with $C_1$ and $C_2$ are given positive constants, the following energy estimate holds 
\begin{align*}
\begin{aligned}
\|\w^{k+1}\|^2_{\PoneD{}}
&+\|\v^{k+1}\|^2_{\PoneD{}}
+\delta t\,\gamma_1(\e_{\alpha}\v^{k+1})^2
+\delta t\,\gamma_2(\kappa\,\e_{\alpha}\DoneD{x}\v^{k+1})^2
\\\le&
\|\q_{\mat L}\|^2_{\PoneD{}}
+\|\q_{\mat R}\|^2_{\PoneD{}} 
+\sum_{i=1}^{k+1}\left(
\frac{\delta t}{2c}\,(\e_{\alpha}\g^{i})^2
+\frac{\delta t}{2}\,(\e_{\beta}\h^{i})^2\right)
+\delta t\,\gamma_1(\e_{\alpha}\v^{0})^2
+\delta t\,\gamma_2(\kappa\,\e_{\alpha}\DoneD{x}\v^{0})^2,
\end{aligned}
\end{align*}
for all $k\ge 0$,
\end{thm}
\begin{proof} 
Starting from \eqref{d11}, we consider the partitioned scheme and set
$\v^*=\v^{k}$ and $\w^*=\w^{k+1}$ in the remaining interface SAT terms.
In particular, by treating the interface SAT contributions
${\mat S}_{\sigma}^{{\mat L},1}$ and ${\mat S}_{\sigma}^{{\mat R},1}$ together, we obtain
\begin{align}\label{d12}
\begin{aligned}
I_1
&:=2\delta t\,\gamma_1\,{\w^{k+1}}\Tr\PoneD{}\,{\mat S}_{\sigma}^{{\mat L},1}
  +2\delta t\,\gamma_1\,{\v^{k+1}}\Tr\PoneD{}\,{\mat S}_{\sigma}^{{\mat R},1}\\
&\;=2\delta t\,\gamma_1\,(\e_{\beta}\w^{k+1})\big(\e_{\beta}\w^{k+1}-\e_{\alpha}\v^{k}\big)
  +2\delta t\,\gamma_1\,(\e_{\alpha}\v^{k+1})\big(\e_{\alpha}\v^{k+1}-\e_{\beta}\w^{k+1}\big)\\
&\;=2\delta t\,\gamma_1\Big(\e_{\beta}\w^{k+1}-\e_{\alpha}\v^{k+1}\Big)^2
  +2\delta t\,\gamma_1\,(\e_{\beta}\w^{k+1})\Big(\e_{\alpha}\v^{k+1}-\e_{\alpha}\v^{k}\Big).
\end{aligned}
\end{align}
Next, consider the interface SAT terms ${\mat S}_{\Sigma}^{{\mat L},2}$ and ${\mat S}_{\Sigma}^{{\mat R},2}$ together, in same way
\begin{align}\label{d13}
\begin{aligned}
I_2
&:=2\delta t\,\gamma_2\,{\w^{k+1}}\PoneD{}\,{\mat S}_{\sigma}^{{\mat L},2}
  +2\delta t\,\gamma_2\,{\v^{k+1}}\PoneD{}\,{\mat S}_{\sigma}^{{\mat R},2}\\
&\;=2\delta t\,\gamma_2\Big(\epsilon\,\e_{\beta}\DoneD{x}\w^{k+1}-\kappa\,\e_{\alpha}\DoneD{x}\v^{k+1}\Big)^2
 +2\delta t\,\gamma_2\,\big(\epsilon\,\e_{\beta}\DoneD{x}\w^{k+1}\big)
   \Big(\kappa\,\e_{\alpha}\DoneD{x}\v^{k+1}-\kappa\,\e_{\alpha}\DoneD{x}\v^{k}\Big).
\end{aligned}
\end{align}
Substituting \eqref{d12}--\eqref{d13} into \eqref{d11}, we obtain
\begin{align}\label{d14}
\begin{aligned}
\|\w^{k+1}\|^2_{\PoneD{}}
&+\|\v^{k+1}\|^2_{\PoneD{}}
+\delta t\,\epsilon \|\DoneD{x}\w^{k+1}\|^2_{\PoneD{}}
+2\delta t\,\kappa \|\DoneD{x}\v^{k+1}\|^2_{\PoneD{}}\\
&+\delta t\,\big(2\gamma_1-\dfrac{\epsilon}{\rho}\big)\Big(\e_{\beta}\w^{k+1}-\e_{\alpha}\v^{k+1}\Big)^2
+2\delta t\,\gamma_2\Big(\epsilon\,\e_{\beta}\DoneD{x}\w^{k+1}-\kappa\,\e_{\alpha}\DoneD{x}\v^{k+1}\Big)^2\\
\le&
\|\w^{k}\|^2_{\PoneD{}}
+\|\v^{k}\|^2_{\PoneD{}}
+\frac{\delta t}{2c}\,(\e_{\alpha}\g^{k+1})^2
+\frac{\delta t}{2}\,(\e_{\beta}\h^{k+1})^2\\&
-2\delta t\,\gamma_1\,(\e_{\beta}\w^{k+1})\Big(\e_{\alpha}\v^{k+1}-\e_{\alpha}\v^{k}\Big)\\&
-2\delta t\,\gamma_2\,\big(\epsilon\,\e_{\beta}\DoneD{x}\w^{k+1}\big)
\Big(\kappa\,\e_{\alpha}\DoneD{x}\v^{k+1}-\kappa\,\e_{\alpha}\DoneD{x}\v^{k}\Big).
\end{aligned}
\end{align}
Then consider the last term on the right-hand side of \eqref{d14}. We have
{\setlength{\abovedisplayskip}{7pt}
\setlength{\belowdisplayskip}{7pt}
\begin{align}\label{d15}
\begin{aligned}
-(\e_{\beta}\w^{k+1})&\Big(\e_{\alpha}\v^{k+1}-\e_{\alpha}\v^{k}\Big)
=(\e_{\beta}\w^{k+1})\Big(\e_{\alpha}\v^{k}-\e_{\alpha}\v^{k+1}\Big)\\
&=\Big(\e_{\beta}\w^{k+1}-\e_{\alpha}\v^{k+1}+\e_{\alpha}\v^{k+1}\Big)
   \Big(\e_{\alpha}\v^{k}-\e_{\alpha}\v^{k+1}\Big)\\
&=(\e_{\alpha}\v^{k+1})\Big(\e_{\alpha}\v^{k}-\e_{\alpha}\v^{k+1}\Big)
 +\Big(\e_{\beta}\w^{k+1}-\e_{\alpha}\v^{k+1}\Big)
  \Big(\e_{\alpha}\v^{k}-\e_{\alpha}\v^{k+1}\Big)\\
&\le (\e_{\alpha}\v^{k+1})(\e_{\alpha}\v^{k})-(\e_{\alpha}\v^{k+1})^2
+\frac12\Big(\e_{\alpha}\v^{k}-\e_{\alpha}\v^{k+1}\Big)^2
+\frac12\Big(\e_{\beta}\w^{k+1}-\e_{\alpha}\v^{k+1}\Big)^2\\
&=\frac12(\e_{\alpha}\v^{k})^2-\frac12(\e_{\alpha}\v^{k+1})^2
+\frac12\Big(\e_{\beta}\w^{k+1}-\e_{\alpha}\v^{k+1}\Big)^2.
\end{aligned}
\end{align}}
Doing the same for the flux term gives
{\setlength{\abovedisplayskip}{7pt}
\setlength{\belowdisplayskip}{7pt}
\begin{align}\label{d16}
\begin{aligned}
-\big(\epsilon\,\e_{\beta}\DoneD{x}\w^{k+1}\big)&\Big(\kappa\,\e_{\alpha}\DoneD{x}\v^{k+1}-\kappa\,\e_{\alpha}\DoneD{x}\v^{k}\Big)
=\big(\epsilon\,\e_{\beta}\DoneD{x}\w^{k+1}\big)\Big(\kappa\,\e_{\alpha}\DoneD{x}\v^{k}-\kappa\,\e_{\alpha}\DoneD{x}\v^{k+1}\Big)\\
&\le \frac12(\kappa\,\e_{\alpha}\DoneD{x}\v^{k})^2-\frac12(\kappa\,\e_{\alpha}\DoneD{x}\v^{k+1})^2
+\frac12(\epsilon\,\e_{\beta}\DoneD{x}\w^{k+1}-\kappa\,\e_{\alpha}\DoneD{x}\v^{k+1})^2.
\end{aligned}
\end{align}}
Substituting \eqref{d15}--\eqref{d16} into \eqref{d14} yields
\begin{align*}
\begin{aligned}
\|\w^{k+1}\|^2_{\PoneD{}}
&+\|\v^{k+1}\|^2_{\PoneD{}}
+\delta t\,\epsilon \|\DoneD{x}\w^{k+1}\|^2_{\PoneD{}}
+2\delta t\,\kappa \|\DoneD{x}\v^{k+1}\|^2_{\PoneD{}}\\
&\quad+\delta t\,\big(\gamma_1-\dfrac{\epsilon}{\rho}\big)\Big(\e_{\beta}\w^{k+1}-\e_{\alpha}\v^{k+1}\Big)^2
+\delta t\,\gamma_2\Big(\epsilon\,\e_{\beta}\DoneD{x}\w^{k+1}-\kappa\,\e_{\alpha}\DoneD{x}\v^{k+1}\Big)^2\\
&\le
\|\w^{k}\|^2_{\PoneD{}}
+\|\v^{k}\|^2_{\PoneD{}}
+\frac{\delta t}{2c}\,(\e_{\alpha}\g^{k+1})^2
+\frac{\delta t}{2}\,(\e_{\beta}\h^{k+1})^2\\
&\quad+\delta t\,\gamma_1\Big((\e_{\alpha}\v^{k})^2-(\e_{\alpha}\v^{k+1})^2\Big)
+\delta t\,\gamma_2\Big((\kappa\,\e_{\alpha}\DoneD{x}\v^{k})^2-(\kappa\,\e_{\alpha}\DoneD{x}\v^{k+1})^2\Big).
\end{aligned}
\end{align*}
By assumption \eqref{da2}, $\gamma_1\rho\ge\epsilon$ and therefore this simplifies to
\begin{align*}
\begin{aligned}
\|\w^{k+1}\|^2_{\PoneD{}}
+\|\v^{k+1}\|^2_{\PoneD{}}
\le&
\|\w^{k}\|^2_{\PoneD{}}
+\|\v^{k}\|^2_{\PoneD{}}
+\frac{\delta t}{2c}\,(\e_{\alpha}\g^{k+1})^2
+\frac{\delta t}{2}\,(\e_{\beta}\h^{k+1})^2\\
&+\delta t\,\gamma_1\Big((\e_{\alpha}\v^{k})^2-(\e_{\alpha}\v^{k+1})^2\Big)
+\delta t\,\gamma_2\Big((\kappa\,\e_{\alpha}\DoneD{x}\v^{k})^2-(\kappa\,\e_{\alpha}\DoneD{x}\v^{k+1})^2\Big).
\end{aligned}
\end{align*}
Applying this estimate recursively yields
\begin{align*}
\begin{aligned}
\|\w^{k+1}\|^2_{\PoneD{}}
+\|\v^{k+1}\|^2_{\PoneD{}}
\le&
\|\w^{0}\|^2_{\PoneD{}}
+\|\v^{0}\|^2_{\PoneD{}}
+\Sigma_{i=1}^{k+1}\left(\frac{\delta t}{2c}\,(\e_{\alpha}\g^{i})^2
+\frac{\delta t}{2}\,(\e_{\beta}\h^{i})^2\right)\\
&+\delta t\,\gamma_1\Sigma_{i=0}^{k}\Big((\e_{\alpha}\v^{i})^2-(\e_{\alpha}\v^{i+1})^2\Big)
+\delta t\,\gamma_2\Sigma_{i=0}^{k}\Big((\kappa\,\e_{\alpha}\DoneD{x}\v^{i})^2-(\kappa\,\e_{\alpha}\DoneD{x}\v^{i+1})^2\Big).
\end{aligned}
\end{align*}
Equivalently,
\begin{align*}
\begin{aligned}
\|\w^{k+1}\|^2_{\PoneD{}}
&+\|\v^{k+1}\|^2_{\PoneD{}}
+\delta t\,\gamma_1(\e_{\alpha}\v^{k+1})^2
+\delta t\,\gamma_2(\kappa\,\e_{\alpha}\DoneD{x}\v^{k+1})^2\\
\le&
\|\q_{\mat L}\|^2_{\PoneD{}}
+\|\q_{\mat R}\|^2_{\PoneD{}}
+\Sigma_{i=1}^{k+1}\left(\frac{\delta t}{2c}\,(\e_{\alpha}\g^{i})^2
+\frac{\delta t}{2}\,(\e_{\beta}\h^{i})^2\right)+\delta t\,\gamma_1(\e_{\alpha}\v^{0})^2
+\delta t\,\gamma_2(\kappa\,\e_{\alpha}\DoneD{x}\v^{0})^2.
\end{aligned}
\end{align*}
Assuming that \eqref{da1}--\eqref{da2} hold, this provide a discrete stability estimate for the proposed partitioned scheme.
\end{proof}
\textit{\emph{\bf Note 3.}}  The conditions on stability in Theorem \ref{thm:1D_part} establish a CFL-like stability condition of the form $\delta t\le \rho^2$ to ensure both the compatibility of the continuous and discrete energy estimate and energy stability.\\

\textit{\emph{\bf Note 4.}} We use an energy-based analysis to design a conditionally stable partitioned coupling at the discrete level. This yields a global stability bound, but it does not provide a mode-by-mode characterization of transverse/tangential behavior, such as mode amplification, dispersion, or transient effects that may be important in practice. A detailed transverse analysis (e.g., normal-mode analysis) is complementary and provide sharper mode-by-mode insight.

\section{Partitioned SBP-SAT scheme in 3D curvilinear setting}\label{sec:3dCurv}
Having established the monolithic and partitioned coupling strategies in 1D Cartesian coordinates, we now extend the construction to 3D curvilinear grids in order to handle general geometries. We first introduce the mapping between the physical and computational domains in Section~\ref{subsec:coor}. To illustrate the coordinate-transformation procedure, in Section~\ref{subsec:tran} we rewrite the advection--diffusion equation in curvilinear coordinates. Also in Section~\ref{subsec:SbpCurv} we define multi-dimensional SBP operators together with boundary and interface SAT terms, and we present a stable semi-discrete SBP--SAT scheme for the advection--diffusion equation in curvilinear coordinate; the same construction is later applied to the model CHT problem \eqref{eq03a}--\eqref{eq03c}. Finally, in Section~\ref{sebsec:3DCurvSBP} we present the fully discrete 3D scheme for the coupled PDEs. Discrete stability estimates are derived in Section~\ref{subsec:alg3D}. Extensions based on time extrapolation at the interface and a second-order-in-time discretization are treated in Sections~\ref{SubSec:BEEXT} and \ref{SubSec:BEFE}, respectively.

\subsection{Coordinate transformation}\label{subsec:coor}
We now extend the formulation and stability analysis from the 1D Cartesian setting to 3D curvilinear grids. We assume the physical domain $\Omega\subset\Rs^3$ with boundary $\Gamma:=\partial \Omega$, with Cartesian coordinates  
$\x:=(x_1,\,x_2,\,x_3)\in\Omega$ is mapped to a rectangular cuboid  computational (reference) domain, 
\[\hat \Omega=[\alpha_1,\beta_1]\times[\alpha_2,\beta_2]\times[\alpha_3,\beta_3],\]
with coordinates  $\bxi:=(\xi_1,\,\xi_2,\,\xi_3)$ with the boundary $\hat{\Gamma}:=\partial \hat\Omega$. More precisely,  the 
reference domain maps onto the physical domain by a fixed (time-independent) transformation,
\[(x_1,\,x_2,\,x_3)=X(\xi_1,\,\xi_2,\,\xi_3).\]
The Jacobian matrix and its determinant (called the metric Jacobian) are denoted by 
\begin{eqnarray}\nonumber\label{eq03}
	\dfrac{\partial(x_1,x_2,x_3)}{\partial (\xi_1,\xi_2,\xi_3)}=\begin{bmatrix}
		\dfrac{\partial x_1}{\partial \xi_1}&\dfrac{\partial x_1}{\partial \xi_2}&\dfrac{\partial x_1}{\partial \xi_3}\\
		\dfrac{\partial x_2}{\partial \xi_1}&\dfrac{\partial x_2}{\partial \xi_2}&\dfrac{\partial x_2}{\partial \xi_3}\\
		\dfrac{\partial x_3}{\partial \xi_1}&\dfrac{\partial x_3}{\partial \xi_2}&\dfrac{\partial x_3}{\partial \xi_3}\\
	\end{bmatrix},\qquad J:=\bigg|\dfrac{\partial(x_1,x_2,x_3)}{\partial (\xi_1,\xi_2,\xi_3)}\bigg|,
\end{eqnarray} 
where we assume that $X$ is such that $J>0$. The vectors $\n=[n_{x_1},n_{x_2},n_{x_3}]\Tr$ and $\hat \n=[\hat n_{\xi_1},\hat n_{\xi_2},\hat n_{\xi_3}]\Tr$ are the 
outward facing unit normals in computational and physical coordinates, respectively. For a given surface, for example,  the surface $\Gamma_l$ and associated  reference surface $\hat \Gamma_l$ given by $\xi_l=\beta_l$ (surface perpendicular to the $l$ computational coordinate), they are related as follows: 
\begin{eqnarray*}
	n_{x_m}=  \sum_{l=1}^3 \dfrac{J}{\hat J_{\beta_l}}\, \dfrac{\partial \xi_l}{\partial x_m}\,\hat n_{\xi_l},\qquad m=1,2,3,
\end{eqnarray*}
where $\hat J_{\beta_l}$ is  the surface Jacobian and is given by 
\begin{eqnarray*}
	\hat J_{\beta_l}=  \bigg(\sum_{m=1}^3 \left(J\dfrac{\partial \xi_l}{\partial x_m}\right)^2  \bigg)^{\frac12},\qquad l=1,2,3.
\end{eqnarray*}

\subsection{Transforming the PDE to curvilinear coordinates}\label{subsec:tran}
In this section we explain how to transform the PDEs from physical coordinates to curvilinear coordinates. To facilitate this, we begin by considering a single linear advection-diffusion equation defined in physical coordinates. This approach allows us to clearly illustrate construction of stable SBP-SAT scheme before extending it to the CHT problem. Consider the linear advection-diffusion equation in  physical coordinates $\Omega\in\Rs^3$
\begin{subequations}
	\begin{align}     
		\dfrac{\partial \fnc{W}}{\partial t}+\sum_{m=1}^3 \dfrac{\partial(a_m \fnc{W})}{\partial x_m}&=\epsilon\sum_{m=1}^3\dfrac{\partial^2 \fnc{W}}{\partial x_m^2},&\qquad& \x\in\Omega,\;0\le t\le T,\label{eq07a}
		\\\zeta \fnc{W}+\epsilon\sum_{m=1}^3\dfrac{\partial(n_{x_m}\fnc{W})}{\partial x_m}&=g(\x,t),&\qquad& \x\in\partial \Omega,\;0\le t\le T,\label{eq07b}
		\\\fnc{W}(\x,0)&=q(\x),&\qquad&\x\in\Omega \label{eq07c},
	\end{align}
\end{subequations}
where \[\zeta=\dfrac12\left(\bigg|\sum_{m=1}^3 a_mn_{x_m}\bigg|-\sum_{m=1}^3 a_mn_{x_m}\right).\] Expanding the derivatives with the chain rule results in
\begin{align*}
	\begin{aligned}
		&\dfrac{\partial}{\partial x_m}=\sum_{l=1}^3\dfrac{\partial\xi_l}{\partial x_m}\dfrac{\partial}{\partial \xi_l},\quad\quad\dfrac{\partial^2}{\partial x_m^2}=\sum_{l,a=1}^3\dfrac{\partial\xi_l}{\partial x_m}\dfrac{\partial}{\partial \xi_l}\big(\dfrac{\partial\xi_a}{\partial x_m}\dfrac{\partial}{\partial \xi_a}\big),\qquad m=1,2,3.
	\end{aligned}
\end{align*}
Utilizing these in \eqref{eq07a} and multiplying by the metric Jacobian, $J$, we obtain
\begin{align}\label{eq08}
	\begin{aligned}
		J\dfrac{\partial \fnc{W}}{\partial t}+\sum_{l,m=1}^3 J \dfrac{\partial\xi_l}{\partial x_m}\dfrac{\partial (a_m\fnc{W})}{\partial \xi_l}=\epsilon\sum_{l,a,m=1}^3J\dfrac{\partial\xi_l}{\partial x_m}\dfrac{\partial}{\partial \xi_l}\big(\dfrac{\partial\xi_a}{\partial x_m}\dfrac{\partial \fnc{W}}{\partial \xi_a}\big).
	\end{aligned}
\end{align}
Bringing the metric term, $J \dfrac{\partial\xi_l}{\partial x_m}$, inside the derivative and using again the chain rule and metric identities:
\begin{align}\label{eqMIs}
	\begin{aligned}
		\sum_{l=1}^3 \dfrac{\partial }{\partial\xi_l }\bigg(J \dfrac{\partial \xi_l}{\partial x_m}\bigg)=0,\qquad m=1,2,3,
	\end{aligned}
\end{align}
leads to the strong conservation form of the advection-diffusion equation in curvilinear coordinates,
\begin{align}\label{eq09}
	\begin{aligned}
		J\dfrac{\partial \fnc{W} }{\partial t}+\sum_{l,m=1}^3 \dfrac{\partial }{\partial\xi_l }\bigg(J \dfrac{\partial \xi_l }{\partial x_m}a_m \fnc{W}\bigg)
		= \epsilon\sum_{l,a,m=1}^3 \dfrac{\partial }{\partial \xi_l}\bigg(J\dfrac{\partial\xi_l }{\partial x_m}\dfrac{\partial \xi_a}{\partial x_m }\dfrac{\partial \fnc{W} }{\partial\xi_a }\bigg).
	\end{aligned}
\end{align}
Notice that the mapping is fixed in time and hence $J\dfrac{\partial \fnc{W} }{\partial t}=\dfrac{\partial J \fnc{W} }{\partial t}$.

\subsection{Multi-dimensional SBP operators, governing equations and boundary and interface SATs in 3D curvilinear coordinates}\label{subsec:SbpCurv}
Assume $\bxi^{(1)}, \bxi^{(2)},\ldots,\bxi^{(N)}$ are grid points defined on $\hat\Omega$ where $N$ is the total 
number of nodes $N=N_1\,N_2\,N_3$, where $N_l$, for $l=1,2,3$, are positive integers denoting the number of nodes 
in each direction. For any real $t$ assume $\fnc{U}(\cdot, t)\in H^2(\Omega)$, then we define the grid function $\u$ on the $N$ nodes as
\[\u :=\big[\fnc{U}(\bxi^{(1)},t), \fnc{U}(\bxi^{(2)},t),\ldots,\fnc{U}(\bxi^{(N)},t)\big]\Tr.\]
The definition of a 1D SBP operator in the $\xi_l$ direction, $\DoneD{\xi_l}\in\Rs^{N_l\times N_l}$, for $ l=1,2,3$  are given in same manner as Definition \eqref{Def:SBP1D}. 
\begin{defnition}\label{Def:SBP1DCurv}
A matrix operator $\DoneD{\xi_l}\in\Rs^{N_l\times N_l}$ is an SBP operator of degree $p$
approximating the first derivative $\frac{\partial}{\partial \xi_l}$ on $\xi_l\in[\alpha_l,\beta_l]$
with nodal vector $\bxi_l=[\xi_l^{(1)},\dots,\xi_l^{(N_l)}]^T$, if
\begin{enumerate}
  \item $\DoneD{\xi_l}\,(\bxi_l)^{\,r}=r\,(\bxi_l)^{\,r-1}$ for $r=1,2,\ldots,p$, where
  $(\bxi_l)^{\,r}:=[(\xi_l^{(1)})^r,\dots,(\xi_l^{(N_l)})^r]^T$, and $\DoneD{\xi_l}\,\bm{1}=\bm{0}$.
  \item $\DoneD{\xi_l}=\PoneD{\xi_l}^{-1}\QoneD{\xi_l}$ with $\PoneD{\xi_l}$ symmetric positive definite.
  \item $\QoneD{\xi_l}+\QoneD{\xi_l}^T=\EoneD{\xi_l}$, where
  \[
    \EoneD{\xi_l}=\diag(-1,0,\ldots,0,1)
    =\e_{\beta_l}^T\e_{\beta_l}-\e_{\alpha_l}^T\e_{\alpha_l},
  \]
  with $\e_{\alpha_l},\e_{\beta_l}\in\Rs^{1\times N_l}$ given by
  $\e_{\alpha_l}=[1,0,\ldots,0]$ and $\e_{\beta_l}=[0,\ldots,0,1]$.
\end{enumerate}
\end{defnition}
We extend the 1D SBP operator to multiple dimensions using tensor products \cite{Nordstrom2001}, \cite{DelRey2018}, \cite{Hicken2016}. The tensor product between the matrices $\mr A$ and $\mr B$ is given as $\mr A\otimes\mr B$.
We construct our multi-dimensional SBP operators as 
\[\Pnorm{}:=\PoneD{\xi_1}\otimes \PoneD{\xi_2}\otimes\PoneD{\xi_3},\qquad
\qquad\Q{\xi_1}:=\QoneD{\xi_1}\otimes\PoneD{\xi_2}\otimes\PoneD{\xi_3}\qquad\qquad
\D{\xi_1}:=\Pnorm{}^{-1}\Q{\xi_1}.\]
Equivalently we have 
\(\D{\xi_1}=\DoneD{\xi_1}\otimes \I{\xi_2}\otimes\I{\xi_3},\) 
where $\I{\xi_l}$ are $N_l\times N_l$ identity matrices for $l=1,2,3$. $\D{\xi_2}$ and $\D{\xi_3}$ are defined similarly. The surface matrices are given as
\begin{align*} 
	\begin{aligned}
	&\E{\xi_1}:=\E{\beta_1}-\E{\alpha_1},&\qquad& \E{\alpha_1}:=\R{\alpha_1}\Tr\Pnorm{\perp\xi_1}\R{\alpha_1},&\qquad&		\E{\beta_1}:=\R{\beta_1}\Tr\Pnorm{\perp\xi_1}\R{\beta_1},		\\&\R{\alpha_1}:=\e_{\alpha_1}\otimes\I{\xi_2}\otimes\I{\xi_3},&\qquad&		\R{\beta_1}:=\e_{\beta_1}\otimes\I{\xi_2}\otimes\I{\xi_3},&\qquad&
	\Pnorm{\perp\xi_1}:=\PoneD{\xi_2}\otimes\PoneD{\xi_3}.
	\end{aligned}
\end{align*}
The operators in the other two computational directions are defined similarly.

In order to construct a semi-discrete form of \eqref{eq09}, we define the diagonal matrices containing the metric Jacobian and metric terms along their diagonals, respectively, as follows:
\begin{align*} 
	\begin{aligned}
		\MJack{}&:=\diag\big(J(\bxi^{(1)}),\ldots,J(\bxi^{(N)})\big),\\
		\MJdxildxmk{l}{m}{}&:=\diag\bigg(\Jdxildxm{l}{m}{}(\bxi^{(1)}),\ldots,\Jdxildxm{l}{m}{}(\bxi^{(N)})\bigg),\\
		\Cmat{l,a}{}&:=\diag\big({\mat C}_{l,a}(\bxi^{(1)}),\ldots, {\mat C}_{l,a}(\bxi^{(N)})\big),
	\end{aligned}
\end{align*}
where ${\mat C}_{l,a}=\sum_{m=1}^3 J\dfrac{\partial \xi_l}{\partial x_m}\dfrac{\partial \xi_a}{\partial x_m},$
for $l=1,2,3$. To define a stable scheme, we split the inviscid terms into one half of the inviscid terms in \eqref{eq08} and one half of  the inviscid terms in \eqref{eq09}, while keeping the viscous terms in strong conservation form, see \cite{DelRey2020p} for more details. This leads to the following semi-discrete form for \eqref{eq07a}
\begin{align*}
	\begin{aligned}
		\MJack{} \dfrac{d \w}{d t}+\dfrac12\sum_{l,m=1}^3 a_m\bigg\{\D{\xi_l}\MJdxildxmk{l}{m}{}+\MJdxildxmk{l}{m}{}\D{\xi_l}\bigg\}\w
		=\epsilon\sum_{l,a=1}^3\D{\xi_l}\Cmat{l,a}{}\D{\xi_a}\w.
	\end{aligned}
\end{align*}
We assume that the metric terms used to discretize the PDE satisfy the discrete metric identities~\eqref{eqMIs} (see \cite{Crean2018,Thomas_1979,Vinokur_2001,Nolasco_2020}):
\begin{align*}
	\begin{aligned}
		\sum_{l=1}^3\D{\xi_l}\MJdxildxmk{l}{m}{}\one=\zero,\qquad m=1,2,3,
	\end{aligned}
\end{align*}
In a similar manner to the definition of the diagonal volume Jacobian matrix, we define the diagonal surface Jacobian matrices as follows:
let $\hat{\Gamma}_{\alpha_l}$ and $\hat{\Gamma}_{\beta_l}$ 
be surfaces on computational domain with $\xi_l=\alpha_l$ and $\xi_l=\beta_l$, respectively. Define
\begin{align*} 
	\begin{aligned}
		\MSJack{\alpha_l}{}:=\diag\big(\hat J_{\alpha_l}(\hat\bxi^{(1)}_{\alpha_l}),\ldots,\hat J_{\alpha_l}(\hat\bxi^{(N_{\alpha_l}}_{\alpha_l})\big),\qquad l=1,2,3,
	\end{aligned}
\end{align*}
where $\hat\bxi^{(1)}_{\alpha_l},\dots,\hat\bxi^{(N_{\alpha_l})}_{\alpha_l}$ are the nodes on surface $\xi_l=\alpha_l$  and $N_{\alpha_l}=N/N_l$ is the number of nodes on this surface. 
$\MSJack{\beta_l}{}$ is defined identically. Without loss of generality, we assume that all diagonal entries of  volume Jacobian and surface matrices are positive. 

Then SATs (in $\alpha_l$ direction) for the physical boundary conditions \eqref{eq07b} are given as follows
\begin{align}\label{eqS3}
	\begin{aligned}
		{\mat S}_{\alpha_l}&=-{\Pnorm{}}^{-1}\bigg(\dfrac12\R{\alpha_l}\Tr(|\Normal{\alpha_l}|-\Normal{\alpha_l})\MSJack{\alpha_l}{}\Pnorm{\perp\xi_l}\R{\alpha_l}\w
		\\&\qquad\qquad-\epsilon\sum_{a=1}^3\R{\alpha_l}\Tr\Pnorm{\perp\xi_l}\R{\alpha_l}\Cmat{l,a}{}\D{\xi_a}\w-\R{\alpha_l}\Tr\MSJack{\alpha_l}{}\Pnorm{\perp\xi_l}\g_{\alpha_l}\bigg),\qquad l=1,2,3,
	\end{aligned}
\end{align}
where
\(\zeta_{\alpha_l}:=-\sum_{m=1}^3a_m \dfrac{J}{\hat J_{\alpha_l}}\dfrac{\partial\xi_l}{\partial x_m}\),
and
\[\Normal{\alpha_l}:=\diag\big( \zeta_{\alpha_l}(\hat\bxi_{\alpha_l}^{(1)}),\ldots, \zeta_{\alpha_l}(\hat\bxi_{\alpha_l}^{(N_{\alpha_l})}) ),\qquad \g_{\alpha_l}:=\left[g(\hat\bxi^{(1)}_{\alpha_l}, t),\ldots,g(\hat\bxi^{(N_{\alpha_l})}_{\alpha_l}, t)\right]^T.\]
In same manner ${\mat S}_{\beta_l}$ for $l=1,2,3$ are constructed. Then a semi-discretized SBP-SAT scheme for \eqref{eq07a}-\eqref{eq07b} is given by
\begin{align*}
\begin{aligned}
\MJack{}\,\partial_{\delta t}\w
&=-\dfrac12\sum_{l,m=1}^3 a_m\bigg\{\D{\xi_l}\MJdxildxmk{}{m}{}+\MJdxildxmk{}{m}{}\D{\xi_l}\bigg\}\w
+\epsilon\sum_{l,a=1}^3\D{\xi_l}\Cmat{l,a}{}\D{\xi_a}\w
+\sum_{l=1}^3 {\mat S}_{\alpha_l}
+\sum_{l=1}^3 {\mat S}_{\beta_l}.
\end{aligned}
\end{align*}

\subsection{SBP--SAT partitioned scheme in 3D curvilinear coordinates}\label{sebsec:3DCurvSBP}
To impose the interface conditions \eqref{eq03c}, we construct multi-dimensional interface SAT terms in analogy with the coupling strategy in classical SBP multi-block formulations (see \cite{Lindstrom2010} and \cite{Lundquist2018}).  Without loss of generality, we assume that interface surface $\xi_1=\beta_1$ is the surface associated with  left physical  domain and the surface $\xi_1=\alpha_1$ is associated with the right physical domain. Note that at the interface, $\n_L=(1,0,0)$ and  $\n_R=-\n_L$, and the metric Jacobians are the same. We set \[\MSJack{\Sigma}{}:=\MSJack{\beta_1}{}=\MSJack{\alpha_1}{}.\] 
Given $\v^*$, define the SATs at interface for left PDE in following way.
\begin{align}\label{eqS1}
	\begin{aligned}
		{\mat S}^{{\mat L},1}_{\Sigma}&=\Pnorm{}^{-1}\R{\beta_1}\Tr\MSJack{\Sigma}{}\Pnorm{\perp\xi_1}\big(\R{\beta_1}\w^{k+1}-\R{\alpha_1}\v^*\big),
		\\{\mat S}^{{\mat L},2}_{\Sigma}&=\Pnorm{}^{-1}\left(\epsilon\sum_{a=1}^3\D{\xi_a}^T\Cmat{1,a}{\mat L}\R{\beta_1}\Tr\right)\MSJack{\Sigma}{}^{-1}\Pnorm{\perp\xi_1}
		\\&\qquad\qquad\qquad\qquad
        \cdot\left(\epsilon\sum_{a=1}^3\R{\beta_l}\Cmat{1,a}{\mat L}\D{\xi_a}\w^{k+1}-\kappa\sum_{a=1}^3\R{\alpha_1}\Cmat{1,a}{\mat R}\D{\xi_a}\v^{*}\right).
	\end{aligned}
\end{align}
In similar way, given $\w^*$, define the SATs at interface for right PDE:
\begin{align}\label{eqS2}
	\begin{aligned}
		{\mat S}^{{\mat R},1}_{\Sigma}&=\Pnorm{}^{-1}\R{\alpha_1}\Tr\MSJack{\Sigma}{}\Pnorm{\perp\xi_1}\big(\R{\alpha_1}\v^{k+1}-\R{\beta_1}\w^{*}\big),	
		\\{\mat S}^{{\mat R},2}_{\Sigma}&=\Pnorm{}^{-1}\left(\kappa\sum_{a=1}^3\D{\xi_a}^T\Cmat{1,a}{\mat R}\R{\alpha_1}\Tr\right){\MSJack{\Sigma}{}}^{-1}\Pnorm{\perp\xi_1}
		\\&\qquad\qquad\qquad\qquad\cdot
        \left(\kappa\sum_{a=1}^3\R{\alpha_1}\Cmat{1,a}{\mat R}\D{\xi_a}\v^{k+1}-\epsilon\sum_{a=1}^3\R{\beta_1}\Cmat{1,a}{\mat L}\D{\xi_a}\w^{*}\right),
		\\{\mat S}^{{\mat R},3}_{\Sigma}&=\Pnorm{}^{-1}\R{\alpha_1}\Tr\Pnorm{\perp\xi_1}\bigg(\epsilon\sum_{a=1}^3\R{\beta_1}\Cmat{1,a}{\mat L}\D{\xi_a}\w^{*}-\kappa\sum_{a=1}^3\R{\alpha_1}\Cmat{1,a}{\mat R}\D{\xi_a}\v^{k+1}\bigg).
	\end{aligned}
\end{align}
A fully discretize partitioned scheme using SBP-SAT operators for the CHT problem consisting of linear advection-diffusion equation \eqref{eq03a} and heat equation \eqref{eq03b} coupled at $\Sigma=\partial \Omega_L\cap \partial \Omega_R$ with interface conditions \eqref{eq03c} is  given as follows:  with given initial guess for $\v^*$ and $\w^*$ and given SAT parameters  $\gamma_1^{\mat L,\,\mat R},\gamma_2^{\mat L,\, \mat R}\ge 0$, solve at each time step

\begin{itemize}
	\item solve the left sub-problem: given $\v^*$, find $\w^{k+1}$ such that
	\begin{align}\label{eq12a}
		\begin{aligned}
			&\MJack{\mat L}\partial_{\delta t} \w^{k+1}+\dfrac12\sum_{l,m=1}^3 a_m\bigg\{\D{\xi_l}\MJdxildxmk{l}{m}{\mat L}+\MJdxildxmk{l}{m}{\mat L}\D{\xi_l}\bigg\}\w^{k+1}
			+\gamma_1^{\mat L}{\mat S}^{{\mat L},1}_{\Sigma}+\gamma_2^{\mat L} {\mat S}^{{\mat L},2}_{\Sigma}
			\\&\qquad\qquad\qquad=\epsilon\sum_{l,a=1}^3\D{\xi_l}\Cmat{l,a}{\mat L}\D{\xi_a}\w^{k+1}+\sum_{l=1}^3 {\mat S}_{\alpha_l}^{\mat L}+\sum_{l=2}^3 {\mat S}_{\beta_l}^{\mat L},
		\end{aligned}
	\end{align}
	\item update the interface data $\w^*$.
	\item solve the right sub-problem: given $\w^*$, find $\v^{k+1}$ such that
	\begin{align}\label{eq12b}
		\begin{aligned}
			\MJack{\mat R}\partial_{\delta t} \v^{k+1}+\gamma_1^{\mat R}{\mat S}^{{\mat R},1}_{\Sigma}+\gamma_2^{\mat R}{\mat S}^{{\mat R},2}_{\Sigma}
			+{\mat S}^{{\mat R},3}_{\Sigma}
			\;=\kappa\sum_{l,a=1}^3\D{\xi_l}\Cmat{l,a}{\mat R}\D{\xi_a}\v^{k+1}+\sum_{l=1}^3 {\mat S}_{\beta_l}^{\mat R}+\sum_{l=2}^3 {\mat S}_{\alpha_l}^{\mat R},
		\end{aligned}
	\end{align}
	\item update the interface data $\v^*$.
\end{itemize}
where the SATs for the boundaries $\partial \Omega_{\mat L}\backslash\Sigma$ and $\partial\Omega_{\mat R}\backslash\Sigma$ are given by ${\mat S}_{\alpha_l}^{\mat L},\; {\mat S}_{\beta_l}^{\mat L},\;{\mat S}_{\beta_l}^{\mat R}$ and ${\mat S}_{\alpha_l}^{\mat R}$ for $l=1,2, 3$ and they are constructed in same manner as \eqref{eqS3} and the boundary vectors $\g_{\alpha_l}^{k+1},\, \h_{\alpha_l}^{k+1},\, \g_{\beta_l}^{k+1}$ and $\h_{\beta_l}^{k+1}$ are defined similarly.
The SATs for interface are given by \eqref{eqS1} and \eqref{eqS2}.

The proposed partitioning scheme is fully defined in Algorithm \eqref{Alg:part}. In the given algorithm, the non-zero source vector $\f_{\mat L}^{k+1}$ is added to the scheme for the propose of using manufactured solution in numerical experiment and it is given by
\[\f_{\mat L}^{k+1}=[f_{ L}(\bm{\xi}^{(1)},t_{k+1}),\dots f_{ L}(\bm{\xi}^{(N)},t_{k+1})]^T,\] where $f_L(\x,t)$  is non zero given source function for left sub-problem, $\f_{\mat R}^{k+1}$ is defined similarly. 

Our approach follows a weakly-coupled method, where only a small fixed number of iterations are performed in each time slab. In contrast, a strongly-coupled scheme involves iterative solutions of two sub-problems until both interface conditions meet a specified tolerance. 
Furthermore, we assume that the initial guess for the interface value at each time step is calculated by a time extrapolation of order $m=1$ or $m=2$ (EXT$m$), described by the following formula,
\begin{align*}
	\v^*=\,\mbox{Extrapolate}(\v^k,\v^{k-1})=\begin{cases}
		\v^k,&\quad m=1,\\
		2\v^k-\v^{k-1}, &\quad m=2.
	\end{cases}
\end{align*}
\begin{algorithm}[ht!]\caption{Weakly coupled partitioned SBP-SAT for CHT using BE-EXT$s$}\label{Alg:part}
	\begin{algorithmic}[1]
		\State Start with given SAT parameters $\gamma_1^{\mat L,\, \mat R},\,\gamma_2^{\mat L,\, \mat R}\ge0$ and  $\v^0=\q_L$ and $\w^0=\q_R$
		\For {$k=1,\ldots, N_t$}
		\State Let $c=1/\delta t,\, \G_{\mat L}^{k+1}=c\MJack{\mat L}\w^k,\; \G_{\mat R}^{k+1}=c\MJack{\mat R}\v^k$
		\State Let $\v^*=\mbox{Extrapolate}(\v^{k},\, \v^{k-1})$ if $k\ge 2$ else $\v^*=\v^k$.
		\For{$\ell= 1, \ldots, N_{\ell oop}$} 
		\State  Solve  left sub--problem for $\w^{(\ell)}$
		\begin{align*}
			\begin{aligned}		
				&c\MJack{\mat L}\w^{(\ell)}+\dfrac12\sum_{l,m=1}^3 a_m\bigg\{\D{\xi_l}\MJdxildxmk{l}{m}{\mat L}+\MJdxildxmk{l}{m}{\mat L}\D{\xi_l}\bigg\}\w^{(\ell)}
				+\gamma_1^{\mat L}{\mat S}^{{\mat L},1}_{\Sigma}+\gamma_2^{\mat L}{\mat S}^{{\mat L},2}_{\Sigma}
				\\&\qquad\qquad\;=\epsilon\sum_{l,a=1}^3\D{\xi_l}\Cmat{l,a}{\mat L}\D{\xi_a}\w^{(\ell)}
				+\sum_{l=1}^3{\mat S}_{\alpha_l}^{\mat L}+\sum_{l=2}^3{\mat S}_{\beta_l}^{\mat L}
				+\G_{\mat L}^{k+1}+\MJack{\mat L}\f_{\mat L}^{k+1}
			\end{aligned}
		\end{align*}
		\State Update the interface $\w^*=\w^{(\ell)}$
		\State Solve right sub--problem for $\v^{(\ell)}$
		\begin{align*}
			\begin{aligned}
				&c\MJack{\mat R}\v^{(\ell)}+\gamma_1^{\mat R}{\mat S}^{{\mat R},1}_{\Sigma}+\gamma_2^{\mat R}{\mat S}^{{\mat R},2}_{\Sigma}
				+{\mat S}^{{\mat R},3}_{\Sigma}
				\\&\qquad\qquad=\kappa\sum_{l,a=1}^3\Pnorm{}\D{\xi_l}\Cmat{l,a}{\mat R}\D{\xi_a}\v^{(\ell)}
				+\sum_{l=1}^3{\mat S}_{\beta_l}^{\mat R}+\sum_{l=2}^3{\mat S}_{\alpha_l}^{\mat R}	+\G_{\mat R}^{k+1}+\MJack{\mat R}\f_{\mat R}^{k+1}.
			\end{aligned}
		\end{align*}
		\State Update the interface $\v^*=\v^{(\ell)}$
		\EndFor
		\State Set $\v^{k+1}=\v^{(N_{\ell oop})}$ and $\w^{k+1}=\w^{(N_{\ell oop})}$
		\EndFor
	\end{algorithmic}
\end{algorithm}

\subsection{Discrete stability analysis of partitioned scheme in the 3D curvilinear setting}\label{subsec:alg3D}
We now extend the  1D stability analysis in Section~\ref{subsec:Sta1D} to the 3D case on curvilinear grids.
In the subsequent stability proof, we use the following local discrete trace inequality, which extends Lemma~\ref{lemTrace1D} to the curvilinear setting.
Let ${\mathcal{K}}_{\beta_l}$ denote the set of all indices corresponding to the points along the surface $\hat{\Gamma}_{\beta_l}$ and $\mathcal{K}$ denote the set of all indices corresponding to all grid points on the reference domain including its boundary.
\begin{lem}\label{lemTrace}
	Let $\fnc{Z}$ be a real-valued function in the reference domain $\hat \Omega\in\Rs^d$ and $\{\bxi^{(j)}\}$ for $j=0,\ldots, N$ be given distinct grid points in $\hat \Omega$.  
	Let $\z=[\fnc{Z}(\bxi^{(1)}),\ldots,\fnc{Z}(\bxi^{(N)})]^T.$
	Fix $l\in \mathcal L=\{1,\dots, 2d\}$. Let $\tilde\rho_l=\max_{j\in {\mathcal K}_{\beta_l}}\{(\MSJack{\beta_l}{}\Pnorm{\perp\xi_l})_{jj}\}$ and $\tilde \rho= \min_{j\in {\mathcal K}}\{(\MJack{}\Pnorm{})_{jj}\}$.
	Then
	\begin{align}\label{lem}
		\begin{aligned}
			\big\|\R{\beta_l}\z\big\|^2_{\MSJack{\beta_l}{}\Pnorm{\perp\xi_l}}\le \dfrac{1}{ \rho}\big\|\z\big\|^2_{\MJack{}\Pnorm{}},
		\end{aligned}	 
	\end{align}
	where \begin{equation}\label{rho} 
		\rho= \dfrac{\tilde \rho}{\max_{l\in \mathcal L}\{\tilde \rho_l \}}.
	\end{equation}
\end{lem}
\begin{proof}
	\begin{align*}
		\begin{aligned}
			\big\|\R{\beta_l}\z\big\|^2_{\MSJack{\beta_l}{}\Pnorm{\perp\xi_l}}&=\z^T\R{\beta_l}\Tr\MSJack{\beta_l}{}\Pnorm{\perp\xi_l}\R{\beta_l}\z=\sum_{j\in\mathcal{K}_{\beta_l}}\left(\MSJack{\beta_l}{}\Pnorm{\perp\xi_l}\right)_{jj}\bm{z}_{j}^{2}
			\\&\le \tilde\rho_l \sum_{j\in {\mathcal{K}}_{\beta_l}} \bm{z}_{j}^{2}
			\le \dfrac{\max_{l\in \mathcal L}\{\tilde \rho_l \}}{\tilde \rho} \sum_{j\in {\mathcal{K}}} (\MJack{}\Pnorm{})_{jj}\bm{z}_{j}^{2}= \dfrac{1}{ \rho}\|\z\|^2_{\MJack{}\Pnorm{}}.
		\end{aligned}
	\end{align*}	
\end{proof}
The scaling factor \( \rho \) in \eqref{lem} depends on the size of domain and the mesh distribution. It connects norms over the boundaries (surface norms) to norms over the interior of the domain (volume norms) via the discrete local trace inequality given by~\eqref{lem}. The order of accuracy of the SBP operators appearing in \( \Pnorm{} \), the weight matrix, is reflected in \( \rho \). The magnitude of \(\rho \) is of order of  mesh resolution. 

Furthermore, to simplify the notations, we introduce the following interface stabilization energies. At time step $t_{j}$, $1\le j\le k+1 $ we define 
\begin{align}\label{E1}
	\begin{aligned}\calE_1^{j}=\left\|\R{\alpha_1}\v^{j}-\R{\beta_1}\w^{j}\right\|^2_{\Sigma},
	\end{aligned}
\end{align}
and the energy related to dissipation at interface by
\begin{align}\label{E2}
	\begin{aligned}\calE_2^{j}=\bigg\|\kappa\sum_{a=1}^3\R{\alpha_1}\Cmat{1,a}{\mat R}\D{\xi_a}\v^{j}-\epsilon\sum_{a=1}^3\R{\beta_1}\Cmat{1,a}{\mat L}\D{\xi_a}\w^{j}\bigg\|^2_{\bar \Sigma},
	\end{aligned}
\end{align}
where
$\|\cdot\|_{\Sigma}:=\|\cdot\|_{\MSJack{\Sigma}{}\Pnorm{\perp\xi_1}}$ and $
\|\cdot\|_{\bar \Sigma}:=\|\cdot\|_{\MSJack{\Sigma}{}^{-1}\Pnorm{\perp\xi_1}}$. Also we define 
\begin{align}\label{F}
	\begin{aligned}
		\calF_1^{j}=\left\|\R{\alpha_1}\v^{j}\right\|^2_{\Sigma},\qquad\qquad
		\calF_2^{j}=\bigg\|\sum_{a=1}^3\R{\alpha_1}\Cmat{1,a}{\mat R}\D{\xi_a}\v^{j}\bigg\|^2_{\bar \Sigma}.
	\end{aligned}
\end{align}
Also denote by $\calG^{j}$ the energy accumulated from the information from the physical boundaries at left and right subdomains 
{\setlength{\abovedisplayskip}{7pt}
 \setlength{\belowdisplayskip}{7pt} 
\begin{align}\label{data}
	\begin{aligned}
		&\calG^{0}=\left\|\q_{\mat L}\right\|^2_{\mat L} +\left\|\q_{\mat R}\right\|^2_{\mat R}+\delta t\gamma_1\left\|\R{\alpha_1}\q_{\mat R}\right\|^2_{\Sigma}+\kappa^2\delta t\gamma_2
		\bigg\|\sum_{a=1}^3\R{\alpha_1}\Cmat{1,a}{\mat R}\D{\xi_a}\q_{\mat R}\bigg\|^2_{\bar \Sigma},\\
		&\calG^{j}=\delta t\sum_{l=1}^3\dfrac1{\zeta_{\alpha_l}}\left\|\g_{\alpha_l}^{j}\right\|^2_{\MSJack{\alpha_l}{\mat L}\Pnorm{\perp\xi_l}}+\delta t\sum_{l=2}^3\dfrac1{\zeta_{\beta_l}}\left\|\g_{\beta_l}^{j}\right\|^2_{\MSJack{\beta_l}{\mat L}\Pnorm{\perp\xi_l}}
		\\&\qquad\quad+\delta t\sum_{l=1}^3\dfrac1{2}\left\|\h_{\beta_l}^{j}\right\|^2_{\MSJack{\beta_l}{\mat R}\Pnorm{\perp\xi_l}}+\delta t\sum_{l=2}^3\dfrac1{2}\left\|\h_{\alpha_l}^{j}\right\|^2_{\MSJack{\alpha_l}{\mat R}\Pnorm{\perp\xi_l}}, \qquad  1\le j\le k+1,
	\end{aligned}
\end{align}}
where 
\begin{align}\label{zeta}
	\begin{aligned}
		\zeta_{\alpha_l}=\max_{j\in {\cal J}_{\alpha_l}}\{{|(\Normal{\alpha_l})_{jj}|}\},\qquad l=1,2,3,\quad
		\mbox{and}\;\qquad\zeta_{\beta_l}=\max_{j\in {\cal J}_{\beta_l}}\{{|(\Normal{\beta_l})_{jj}|}\},\qquad l=2,3.
	\end{aligned}
\end{align}

Throughout the remainder of this article, we assume that the constant velocity vector $\a$ is defined such that $\a\cdot \n_{\mat L} =0$  at interface. Let $\rho_{\mat L}$ and $\rho_{\mat R}$ denote the positive constants provided by Lemma \eqref{lemTrace} for the left and right subdomains, respectively.  We denote
\begin{align}\label{gamma}
	\begin{aligned}
		\gamma_1=\gamma_1^{\mat L}=\gamma_1^{\mat R},\qquad 
        \gamma_2=\min\{\gamma_2^{\mat L},\,\gamma_2^{\mat R}\},\qquad
        \bar\gamma_2=\max\{\gamma_2^{\mat L}, \, \gamma_2^{\mat R}\},\qquad
        \hat \gamma_2=|\gamma_2^{\mat L}-\gamma_2^{\mat R}|.	
\end{aligned}
\end{align} In the analysis that follows, we consider a weakly coupled scheme without iteration, i.e., \( N_{\ell oop} = 1 \). For simplifying the notation, also let $\|\cdot\|_{\mat L}:=\|\cdot\|_{\MJack{\mat L}\Pnorm{}}$ and $\|\cdot\|_{\mat R}:=\|\cdot\|_{\MJack{\mat R}\Pnorm{}}$.

\begin{thm}[Stability of the partitioned BE-EXT1 scheme in 3D curvilinear]\label{thm:BE}
Given the partitioned scheme \eqref{eq12a} and \eqref{eq12b}.
Let $\v^*=\v^{k}$, $\w^*=\w^{k+1}$ for $k\ge 0$. 
Then, under the conditions 
	\begin{subequations}
		\begin{align}
			\dfrac{\epsilon}{\rho_{\mat L}}\le\gamma_1&\le \dfrac{C_1}{\delta t},\label{a1}
			\\ \gamma_2&\le \dfrac{C_2}{\kappa^2\delta t},\label{a2}
			\\\hat\gamma_2&\le\dfrac{\min\{\rho_{\mat L},\, \rho_{\mat R}\}}{\max\{\epsilon,\,\kappa\}} \label{a3},
		\end{align}
	\end{subequations}
with $C_1$ and $C_2$ are given positive constants, the following energy estimate holds
	\begin{align}\label{eq100}
		\begin{aligned}
			\|\w^{k+1}\|^2_{\mat L}+ \|\v^{k+1}\|^2_{\mat R}+\delta t\gamma_2\sum_{j=1}^{k+1}\calE_2^{j}+\delta t\,\gamma_1\calF_1^{k+1}+\kappa^2\delta t\,\gamma_2\calF_2^{k+1}
			\le \mathcal{S}^{k+1},
		\end{aligned}
	\end{align}for $0\le k\le N_t-1$, where $\mathcal{S}^{k+1}=\sum_{j=0}^{k+1}\calG^{j}$.
\end{thm}
\begin{proof}
A proof of this theorem is given in Appendix \ref{apdB}.
\end{proof}

\textit{\emph{\bf Note 5.}} In Theorem \ref{thm:BE} we assume that the SAT parameters on both sides of the interface for imposing the Nuemann interface condition are chosen independently. This allows us to analyze the benefits of adjusting mesh resolution in normal direction to the  interface when the diffusion constants of the two media differ significantly. The conditional stability of the partitioned scheme, under the  constraint that $\gamma_1 := \gamma_1^{\mat L} = \gamma_1^{\mat R}$ and $\gamma_2^{\mat L} = \gamma_2^{\mat R} = 0$, can then be expressed as following.

\begin{corollary}\label{thm:BE0}
	Assume the partitioned scheme \eqref{eq12a} and \eqref{eq12b} with  $\gamma_2^{\mat L} = \gamma_2^{\mat R} = 0$. Let $\v^*=\v^{k}$, $\w^*=\w^{k+1}$. Then, under the assumptions
	\begin{subequations}
		\begin{align}
			\dfrac{\epsilon}{2\rho_{\mat L}}\le \gamma_1\le \dfrac{C_1}{\delta t},\label{con1}
		\end{align}
	\end{subequations}
	with $C_1$ is given positive constant, the following energy estimate holds
	\begin{align}\label{eq101}
		\begin{aligned}
			& \|\w^{k+1}\|^2_{\mat L}+ \|\v^{k+1}\|^2_{\mat R}+\delta t\gamma_1\calF_1^{k+1}
			\le \mathcal{S}^{k+1},
		\end{aligned}
	\end{align}for $0\le k\le N_t-1$, where $\mathcal{S}^{k+1}=\sum_{j=1}^{k+1}\calG^{j}+\left\|\q_{\mat L}\right\|^2_{\mat L} +\left\|\q_{\mat R}\right\|^2_{\mat R}+\delta t\gamma_1\left\|\R{\alpha_1}\q_{\mat R}\right\|^2_{\Sigma}$.
\end{corollary}

By choosing matching Dirichlet conditions at the interface for both subdomains while imposing the Neumann interface condition only on the right subdomain and not on the left, we adopt a Dirichlet--Robin coupling approach. Although this coupling can lead to a formally conditionally stable partitioned scheme, it may propagate errors at the interface due to the use of inaccurate interface data. In particular, when only a fixed number of coupling iterations is performed per time step (weak coupling), the resulting coupling error can limit the achievable temporal order. The primary issue arises because each subdomain solves its respective equation using interface data from the other subdomain that are not fully consistent, which can result in a loss of accuracy. This inaccuracy propagates across time steps, effectively degrading the designed accuracy of the method \cite{Kazemi2013}.

To mitigate these issues, strategies such as applying a strongly coupled scheme, using higher-order extrapolation in time for the interface conditions, or incorporating additional SAT terms to enforce the Neumann condition at the interface in both subdomains while preserving stability as established in Theorem \ref{thm:BE} can improve order of consistency and restore accuracy while maintaining the efficiency of the partitioned scheme.\\

\textit{\emph{\bf Note 6.}} To provide a constructive approach for selecting parameters based on the given conditions \eqref{a1}-\eqref{a3}, we proceed as follows. Given the number of nodes used for spatial discretization in each direction and the transformation information, we compute $\rho_{\mat L}$  and $\rho_{\mat R}$  using Lemma \ref{lemTrace}. Notice that we assume a conforming mesh at the interface, and $\rho_{\mat L}$ and $\rho_{\mat R}$ can be adjusted by modifying the number of nodes in the direction normal to the interface, taking into account the geometric properties and grid distribution in each subdomain.

Then by applying condition \eqref{a1}, we select $\gamma_1$ such that $\gamma_1 \,\rho_L \geq \epsilon.$
Using condition \eqref{a3}, we choose $\gamma_2^{\mat L}$  and $\gamma_2^{\mat R}$  such that 
\[|\gamma_2^{\mat L}-\gamma_2^{\mat R}|\le\dfrac{\min\{\rho_{\mat L},\, \rho_{\mat R}\}}{\max\{\epsilon,\,\kappa\}}, \]
ensuring $\gamma_2^{\mat L}\ge 0$ and $\gamma_2^{\mat R}\ge 0$.

To derive an upper bound for $\delta t$, we make specific choices for the arbitrary constants \( C_1 \) and \( C_2 \) in conditions \eqref{a1} and \eqref{a2}. In particular, we set  $C_1 := C^*\rho_{\mat L} $  and $C_2 := C^*\min\{\rho_{\mat L},\rho_{\mat R}\}$.  These choices are convenient because they approach zero with mesh refinement. With this selection, it is sufficient to choose the time step such that  
\[
\delta t \leq \min\bigg\{\dfrac{C^*\rho_{\mat L}}{\gamma_1},\, \dfrac{C^*\min\{\rho_{\mat L},\rho_{\mat R}\}}{\kappa^2\min\{\gamma_2^{\mat L},\gamma_2^{\mat R}\}}\bigg\},
\]
for a given positive constant $C^*$. 

Notice that while choosing arbitrarily large \(C_1\) and \(C_2\) allowing a larger time step for stability, it reduces the coherence between the discrete stability estimates \eqref{eq100}  and the continuous stability estimate. The discrete stability estimate is expected to be an equivalent form of the continuous energy estimate in the limit of mesh and time refinement. In other words, the last three terms on the left hand side of \eqref{eq100} vanish.   This was also reported in~\cite{Fernandez2009} in the context of coupling FSI problem using a penalty method at the interface.

\subsection{Using extrapolation  in time of order two at interface as initial guess}\label{SubSec:BEEXT}
In \cite{Meng2017}, it is demonstrated that achieving a second-order accuracy with the CHAMP interface conditions and a second-order time discretization method requires third-order extrapolation in time  at the interface. In this section, we show that by employing backward Euler method for time discretization and second-order extrapolation in time along with the specified given SAT terms at the interface, conditional stability can still be achieved.  In Figure \eqref{diag:BE}, the sequencing of the coupling procedure with BE-EXT2 is graphically presented at the beginning of the loop at each time step. To solve the advection-diffusion problem, the solution at the interface is updated using data from the two previous time levels from the diffusion problem. Notice that in the next loop at current time level, the current calculated solution from the advection-diffusion problem will be used at the interface to solve the diffusion problem. 
\begin{figure}[ht]
\centering
\begin{tikzpicture}
		\node at (2, -0.5) {$t_{n-1}$};
		\node at (6, -0.5) {$t_{n}$};
		\node at (10, -0.5) {$t_{n+1}$};
		
		\node[black] at (-0.5, 0) {Time};
		\node[black] at (-0.5, 1) {Advection-diffusion};
		\node[black] at (-0.5, 2) {Diffusion};
		
		\draw[dashed, thick, gray] (0.5, 0) -- (11, 0); 
		\draw[dashed, thick, gray, ->] (11, 0) -- (11.5, 0);  	
		\foreach \x in { 2, 6, 10 } {
			\draw[black, thick] (\x, 0.1) -- (\x, -0.1);  
		}		
		\foreach \x in { 2,  6, 10} {
			\fill[black] (\x, 1) circle (2pt);      
			\fill[black] (\x, 2) circle (2pt);       
		}		
		\draw[dashed, blue, ->,>=latex] (2, 2) -- (5.9, 1);
		\draw[dashed, blue, ->,>=latex] (6, 2) -- (6, 1.1);
		\draw[thick, blue, ->, >=latex] (6, 1) -- (10, 1);
		\node[draw, blue, circle, fill=white, minimum size=12pt, inner sep=0pt] at (8, 0.7) {1};
		\draw[dashed, red, ->, >=latex] (10, 1) -- (6.05, 1.9);
		\draw[thick, red, ->, >=latex ](6, 2) -- (10, 2);
		\node[draw, red, circle, fill=white, minimum size=12pt, inner sep=0pt] at (8, 2.3) {2};
\end{tikzpicture}
\caption{Sequencing of coupling scheme for BE-EXT2 at the beginning of the loop.}
\label{diag:BE}
\end{figure}

\begin{thm}[Stability of the partitioned BE-EXT2 scheme in 3D curvilinear]\label{thm:BE-EXT2}
Given the partitioned scheme \eqref{eq12a} and \eqref{eq12b}. Assume that the initial guess for interface SAT terms \eqref{eqS1} at each step uses extrapolation in time of order two,  given by the formula
\begin{align}
\begin{aligned}
\v^*=2\v^{k}-\v^{k-1},
\end{aligned}
\end{align}
in \eqref{eqS2} and  $\w^*=\w^{k+1}$ in \eqref{eqS1} for $k\ge 1$ . For $k=0$, let \(\v^*=\v^0\) and $\w^*=\w^{0}$. Then, under the conditions $\rho_{\mat R}<1$ and
\begin{subequations}
\begin{align}
\gamma_1 \rho_{\mat L}(1-\rho_{\mat R})&\ge \epsilon , \label{b1} \\
\delta t\gamma_1(1+\dfrac{4}{\rho_{\mat R}^2})&\le 1 \label{b2},
\\\bar\gamma_2&\le\dfrac{2\rho_{\mat R}}{5\kappa},\label{b3}
\\\hat\gamma_2&\le\dfrac{\min\{\rho_{\mat L},\, \rho_{\mat R}\}}{\max\{\epsilon,\,\kappa\}} \label{b4},
\end{align}
\end{subequations}
the following energy estimate holds
\begin{align}\label{eqBE-EXT2}
\begin{aligned}
&\|\w^{k+1}\|^2_{\mat L}+ (1-\delta t\gamma_1)\|\v^{k+1}\|^2_{\mat R}
\le \mathcal{H}^{k+1},
\end{aligned}
\end{align}
for $1\le k\le N_t-1$, where
\[\mathcal{H}^{k+1}=\mathcal{S}^{k+1}+(k+1)\mathcal{S}^1+\sum_{j=1}^{k}(k+1-j)\mathcal{S}^j,\] and $\mathcal{S}^{j}=\sum_{i=0}^j\calG^{i}$.
\end{thm}
\begin{proof}
	A proof follows that of Theorem \eqref{thm:BE} and is given in Appendix \ref{apdC}.
\end{proof}

Note that, generally, for higher-order partitioned scheme combined with extrapolation in time of order $m$ (EXT$m$) at interface,  unconditional stability cannot be assured. For results on the stability properties of the advection-diffusion equation solved on overlapping grids using the backward finite difference formula of order $l$ (BDF$l$) and EXT$m$, see \cite{Peet2012}. Moreover, for weakly coupled schemes with a fixed number of sub-iterations per time step, the coupling error may limit the achievable temporal order, so higher-order extrapolation alone does not, in general, guarantee high-order accuracy without additional coupling iterations.
As noted by Peet \cite{Peet2012}, it is generally more practical to use BDF$m$/EXT$m$ schemes. This is because increasing the accuracy of the interface conditions beyond the accuracy of the time integration leads to additional computational costs without significantly improving the overall scheme's accuracy. However, as suggested in \cite{Meng2017}, using one order higher extrapolation  in time at interface is a viable choice to preserve the order of accuracy of the time marching scheme which aligns with the findings in our numerical experiments.

\

\subsection{Applying second-order time discretization}\label{SubSec:BEFE}
The natural question that arises is how to adapt to a second-order scheme for the time derivative. For coupled complex systems  transitioning from the first-order backward Euler method to a second-order energy-stable method to improve numerical accuracy can be challenging. The most commonly used second-order accurate time discretization method in multi-physics applications is BFD2. However, it is not energy stable when applied with variable step sizes \cite{Dahlquist1983}. Recently an alternative formulation of the midpoint method and a theta-like generalization of it has been introduced and they have been shown to be unconditionally stable \cite{Burkardt2020} and has been use in FSI application \cite{Bukac2022}. Recall that the classical implicit  midpoint rule is given by
\[\dfrac{\u^{k+1}-\u^k}{\delta t}=\calA(\u^{k+1/2},t_{k+1/2}),\]
where $\calA(\u,t)$ is the matrix results from the spatial discretization of the underlying PDE at  time $t$ (semi-discretized operator) and $\u^{k+1/2}=(\u^{k+1}+\u^{k})/2$.

We implement the midpoint rule by solving a backward finite-difference Euler step at the half time step $t_{k+1/2}$, followed by a forward finite-difference Euler step to $t_{k+1}$,
\begin{align}\label{BEFE}
	\begin{aligned}
		\dfrac{\u^{k+1/2}-\u^{k}}{\delta t/2}&=\calA(\u^{k+1/2},t_{k+1/2}),
		\\\dfrac{\u^{k+1}-\u^{k+1/2}}{\delta t/2}&=\calA(\u^{k+1/2},t_{k+1/2}).
	\end{aligned}
\end{align}
Note that \eqref{BEFE} is equivalent to  the one-stage Gauss--Legendre implicit Runge--Kutta method of order two. It has been shown that this method  is a second order accurate, unconditionally A-stable and energy stable method \cite{Burkardt2020}.

\begin{prop}\label{propStab}
	The midpoint method \eqref{BEFE} is unconditionally-stable, and the following equality holds:
	\[\|\u^{k+1}\|^2_{\Pnorm{}}-\|\u^k\|^2_{\Pnorm{}}=2\delta t \langle \u^{k+1/2},\calA(\u^{k+1/2},t_{k+1/2})\rangle_{\Pnorm{}}.\]
\end{prop}
\begin{proof}
	See proof of Proposition 2.1. in ~\cite{Burkardt2020}.
\end{proof}
Equations \eqref{BEFE} can be  written as
\begin{align}\label{BEFE01}
	\begin{aligned}
		\dfrac{\u^{k+1/2}-\u^{k}}{\delta t/2}&=\calA(\u^{k+1/2},t_{k+1/2}),\\
		\u^{k+1}&=2\u^{k+1/2}-\u^{k}.
	\end{aligned}
\end{align}
We refer to this method as backward Euler, forward Euler (BEFE). In the following, we modify Algorithm~\ref{Alg:part} by applying the BEFE scheme for time discretization.

We define the SAT terms at interface the time level $t_{k+1/2}$.
Given $\v^*$, define the new SAT term at interface for left PDE in following way:
\begin{align*}
	\begin{aligned}
		\hat{\mat S}^{{\mat L},1}_{\Sigma}=\Pnorm{}^{-1}\R{\beta_1}\Tr\MSJack{\Sigma}{}\Pnorm{\perp\xi_1}\big(\R{\beta_1}\w^{k+1/2}-\R{\alpha_1}\v^*\big),
	\end{aligned}
\end{align*}
and similarly we define the other interface SAT terms and boundary SAT terms. Next is to choose initial guess for $\v^*$ in SATs at interface. For $k=0$,  let $\v^*=\v^{0}$ and for $k\ge 1$, set $\v^*=2\v^{k}-\v^{k-1/2}$.  Then for given SAT parameters $\gamma_1,\gamma_2>0$, we solve iteratively

\begin{itemize}
\item solve the left sub-problem: given $\v^*$, find $\w^{k+1/2}$ such that
\begin{align}\label{eq20a}
\begin{aligned}
\MJack{\mat L} \partial_{\delta t} \w^{k+1/2}&+\dfrac12\sum_{l,m=1}^3 a_m\bigg\{\D{\xi_l}\MJdxildxmk{l}{m}{\mat L}+\MJdxildxmk{l}{m}{\mat L}\D{\xi_l}\bigg\}\w^{k+1/2}
+\gamma_1\hat {\mat S}^{{\mat L},1}_{\Sigma}+\gamma_2^{\mat L} \hat{\mat S}^{{\mat L},2}_{\Sigma}
\\&=\epsilon\sum_{l,a=1}^3\D{\xi_l}\Cmat{l,a}{\mat L}\D{\xi_a}\w^{k+1/2}+\sum_{l=1}^3\hat {\mat S}_{\alpha_l}^{\mat L}+\sum_{l=2}^3 \hat{\mat S}_{\beta_l}^{\mat L}.
\end{aligned}
\end{align}
\item update the interface data $\w^*$.
\item solve the right sub-problem: given $\w^*$, find $\v^{k+1/2}$ such that
\begin{align}\label{eq20b}
\begin{aligned}
\MJack{\mat R} \partial_{\delta t} \v^{k+1/2}+\gamma_1\hat{\mat S}^{{\mat R},1}_{\Sigma}+\gamma_2^{\mat R}\hat {\mat S}^{{\mat R},2}_{\Sigma}
+\hat{\mat S}^{{\mat R},3}_{\Sigma}
=\kappa\sum_{l,a=1}^3\D{\xi_l}\Cmat{l,a}{\mat R}\D{\xi_a}\v^{k+1/2}+\sum_{l=1}^3 \hat{\mat S}_{\beta_l}^{\mat R}+\sum_{l=2}^3 \hat{\mat S}_{\alpha_l}^{\mat R}.
\end{aligned}
\end{align}
\item update the interface data $\v^*$.
\item update the solution, set
\begin{align}\label{eq20c}
\begin{aligned}
\w^{k+1}&=2\w^{k+1/2}-\w^{k},
\\\v^{k+1}&=2\v^{k+1/2}-\v^{k}.
\end{aligned}
\end{align}
\end{itemize}
Establishing a stability estimate for this partitioned scheme  is a result of applying Proposition \eqref{propStab} and follows exactly as Theorem \eqref{thm:BE-EXT2}.
\begin{thm}[Stability of the partitioned BEFE-EXT2 scheme in 3D curvilinear]\label{thm:BEFE}Given the partitioned scheme \eqref{eq20a}-\eqref{eq20c}. Then, under the conditions of theorem \eqref{thm:BE-EXT2} the following energy estimate holds
	\begin{align*}
		\begin{aligned}
			\|\w^{k+1}\|^2_{\mat L}+ (1-\delta t\gamma_1)\|\v^{k+1}\|^2_{\mat R}
			\le \mathcal{H}^{k+1}+2\mathcal{H}^{k+1/2}
		\end{aligned}
	\end{align*}
for $1\le k\le N_t-1$.
\end{thm}
\begin{proof}
    The proof follows from the argument of Theorem~\ref{thm:BE-EXT2} with only minor modifications and is omitted for brevity.
\end{proof}

\section{Numerical verification}\label{sec:Numeric}
In this section, we present numerical results to validate the stability and accuracy of the proposed weakly coupled partitioned scheme. We examine the spatial accuracy and demonstrate the order of spatial convergence. 
We analyze the spectral properties of the iteration matrix associated with the scheme. Finally, we investigate the relationship between the SAT parameters and the spatial and temporal discretization step sizes.

We consider a domain consisting of two adjacent squares  $\Omega=\Omega_{\mat L}\cup \Omega_{\mat R}$, where $\Omega_{\mat L}=[-1,0]\times[-1,1]$ and $\Omega_R=[0,1.2]\times[-1,1]$, so the interface is the line $x=0, -1\le y\le1$. The time interval is given as $[0, T]$ and $\delta t=T/N_t$, where $N_t$ is given. 

The curvilinear mesh is generated by first defining an uniform computational grid $\mathbb{G}=(\xi_i, \eta_j)$ for $i,j=1,\ldots n$, for a given $n$, on reference domain $[0,1]^2$ with  spacing $\delta \xi=\delta \eta=1/(n-1)$. For simplicity we assumed that we use the same number of nodes in $\xi$ and $\eta$ directions. We use following transformation to obtain the positions of the curvilinear grid $\mathbb{\hat{G}}=(x_i, y_j)$ for $i,j=1,\ldots n$ on the physical domain
\begin{align}
	\begin{aligned}
		x&=1-\xi-\dfrac{1}{32}\cos \big(\pi(\xi-\dfrac12)\big)\cos\big(3\pi(\eta-\dfrac12)\big),\\\nonumber
		y&=1-\eta-\dfrac{1}{32}\sin\big(4\pi (x-\dfrac12)\big)\cos\big(\pi(\eta-\dfrac12)\big).
	\end{aligned}
\end{align}
\begin{figure}[htp]
	\centering
	\begin{subfigure}{6.25cm}
		\centering\includegraphics[width=5.25cm]{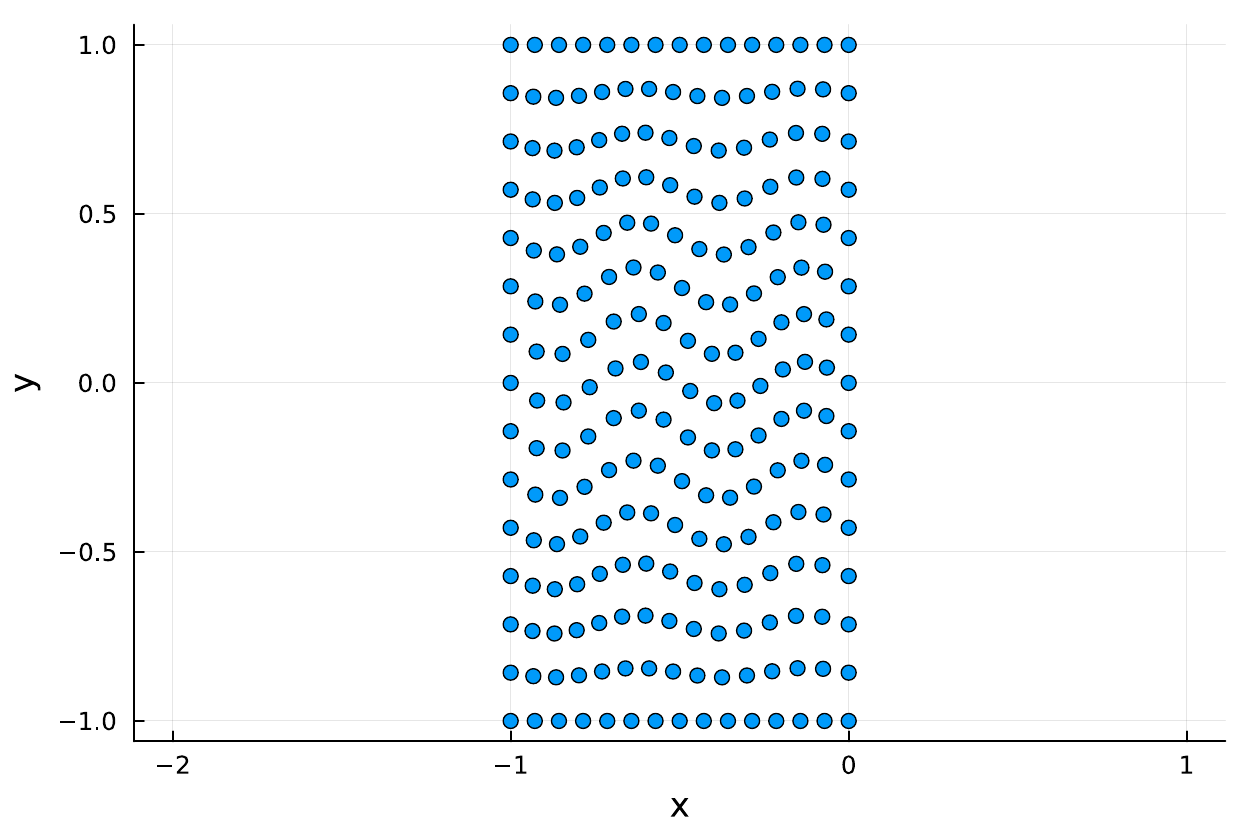}
	\end{subfigure}
	\hspace{1.5cm}
	\begin{subfigure}{6.25cm}
		\centering\includegraphics[width=5.25cm]{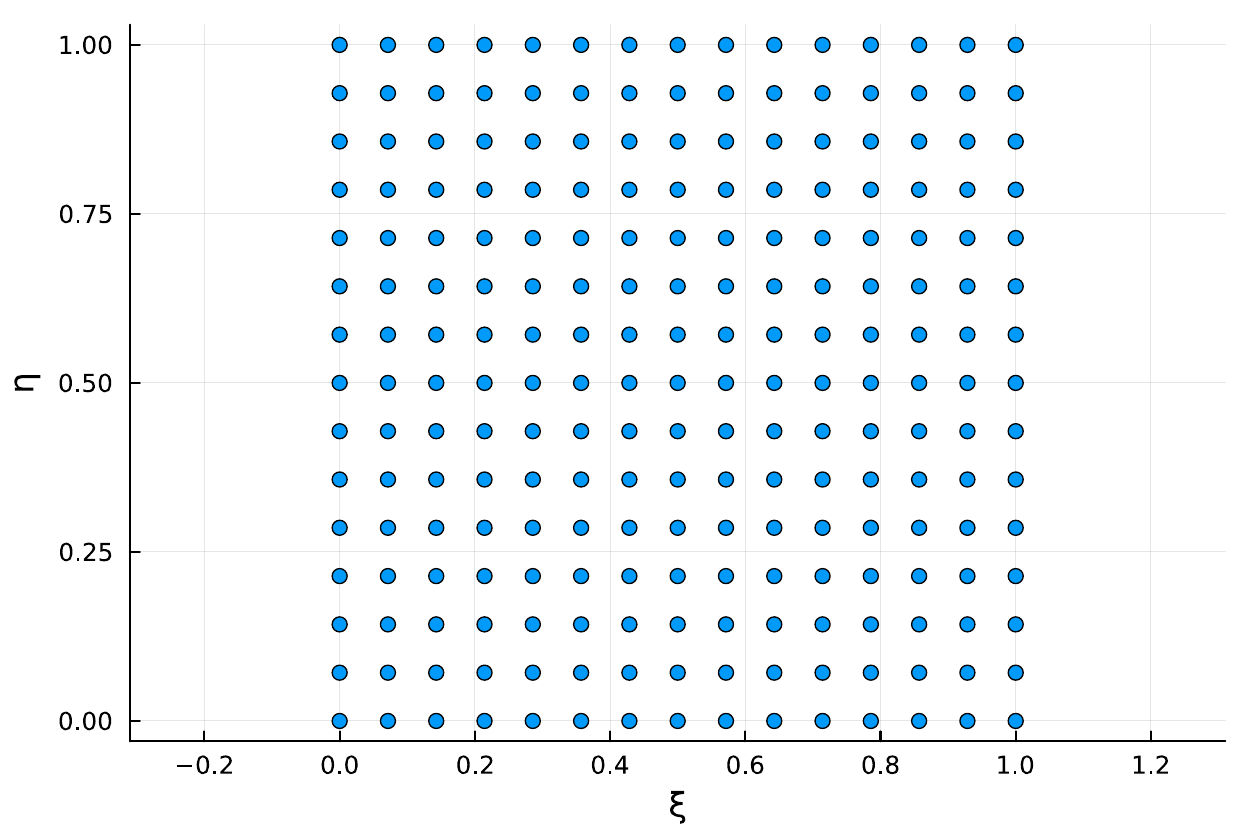}
	\end{subfigure}
	\caption{The distribution of nodes on physical (left) and reference (right) domains } 
	\label{fig:nodes}
\end{figure}

In this work, we focus on SBP operators constructed using finite difference discretization, though the theory extends naturally to finite volume and discontinuous Galerkin methods \cite{DelRey2014}. With a given  number of nodes for spatial discretization in each direction and the transformation information, we compute $\rho_{\mat L}$ and $\rho_{\mat R}$ by using Lemma \ref{lemTrace} for left subdomain. By using conditions \eqref{b1}-\eqref{b4}, we select $\gamma_1, \gamma_2= \gamma_2^{\mat L}= \gamma_2^{\mat R}$ and $\delta t$ accordingly. 

In all of the following experiments the accuracy of the proposed partitioned scheme is evaluated using the method of manufactured solution. We construct solutions for right and left problems such that they satisfy the corresponding PDEs and the interface conditions given by \eqref{eq03c}. We set 
\begin{eqnarray}\nonumber\label{eq46}
	w(x,y,t)=\dfrac{1}{\epsilon}\sin(x^3+x^2y)\,e^{0.1(x+y)t},\qquad v(x,y,t)=\dfrac{1}{\kappa}\sin(x^3+x^2y)\,e^{0.1(x+y)t}.
\end{eqnarray}
With these exact solutions, we calculate the boundary functions $g_L(t)$ and $g_R(t)$, and initial function $q_L(x)$ and $q_R(x)$. We assume $ \a=[a_1,a_2],\, \epsilon$ and $\kappa$  such that the manufactured solutions satisfy the PDEs, the interface conditions hold and  $f_R(0,t)=f_L(0,t)$, where $f_R$ and $f_L$ are the  source terms of right and left PDEs. Also for the wellposedness of PDEs on interface,  $\a\cdot\n_L=0,  (a_1=0)$.

Let  $\w^{N_t+1}$  and $\v^{N_t+1}$ be vectors containing the numerical solutions at the final time $t=T$. Assume $\w_{ex}=w(x_i,y_j, t=T)$ and $\v_{ex}=v(x_i,y_j,t=T)$  are  the exact solutions for left and right problems evaluated at the physical grid nodes at the final time $T$. The error is determined using $\Pnorm{}$ norm at the final time $T$, 
\[Error=\sqrt{\|\w^{N_t+1}-\w_{ex}\|^2_{\mat L}+\|\v^{N_t+1}-\v_{ex}\|^2_{\mat R}}.\]

Table \eqref{tab:SpaceBE-EXT2} presents the results of a spatial convergence study for the partitioned BE-EXT2 scheme in 2D curvilinear coordinates using 1D SBP  operators of order $p=1, 2, 3$.  The ratio in the table is computed by dividing the error on a coarser mesh by the error on the next finer mesh.

Figure \eqref{fig:SpaceOrder} illustrates the convergence rate for the second-order time-marching method, BEFE-EXT2. An estimated rate of the convergence (slope) is determined by fitting a linear polynomial of data.
The results demonstrate that high-order SBP operators provide better accuracy compared to lower-order methods.  By carefully defining the SATs  at the interface, the proposed approach  allows independent solvers for the solid and fluid to interact without introducing instability  while leveraging the advantage of high order accuracy.

\begin{table}[ht!]
\caption{Results of spatial convergence study of iterated partitioned scheme  ($N_{\ell oop}=2$)  in comparison with monolithic scheme for 2D problem with time discretization using BE-EXT2. The results are given at final time $T=1$ with time step $\delta t= 10^{-4}$ for 1D SBP operator of order $p=1,2,3$. }
\centering\scalebox{1}{
\begin{tabular}{c|cll|ll} 
			\hline
			&&Partitioned&&Monolithic&\\
			&  $n$ & error & ratio&error & ratio   \\
			\hline 
			&$5$&  0.175380 &   -  &  0.175380 & -   \\
			$p=1$   &$10$& 0.061504 &1.51  &  0.061504 &1.51  \\
			&$20$& 0.015366 &2.00  &  0.015366 &2.00   \\
			&$40$& 0.003723 &2.04  &  0.003723 &2.04  \\
			\hline 
			&$9$ &  0.070053 &  -  &0.070053&  \\
			$p=2$       &$18$&  0.009911 &2.82 &0.009911&2.82 \\
			&$36$&  0.001193 &3.05 &0.001193&3.05 \\
			&$72$&  0.000143 &3.06 &0.000143&3.06\\
			\hline 
			&$13$ &  0.016970 &-    &0.016970  &- \\
			$p=3$     &$26$ &  0.001230 &3.78 &0.001230  &3.78\\
			&$52$ &  8.0546e-5&3.93 &8.0537e-5 & 3.93\\
			&$104$&  5.7811e-6&3.80 &5.7448e-6&3.80\\
			\hline
\end{tabular}}
\label{tab:SpaceBE-EXT2}
\end{table}

\begin{figure}[ht!]
\centering\includegraphics[width=9.25cm]{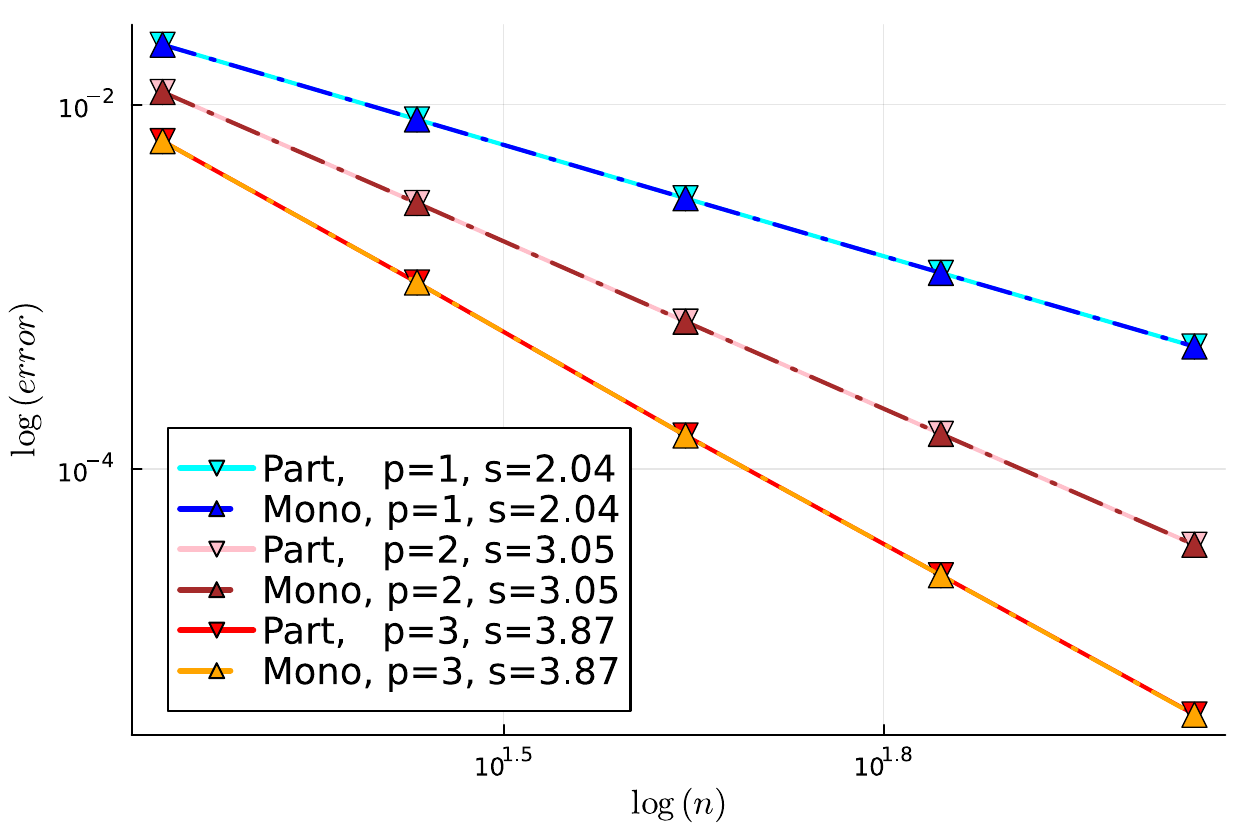}
\caption{Results of spatial convergence study of iterated partitioned scheme  ($N_{\ell oop}=2$) in comparison with monolithic scheme with time discretization BEFE where extrapolation in time of order 2 (EXT2) is used at interface. The results are given at final time $T=1$ with time step $\delta t= 10^{-4}$.} 
	\label{fig:SpaceOrder}
\end{figure}

Strongly coupled partitioned schemes typically require multiple iterations per time step to satisfy interface conditions to a specified tolerance. In this work, we adopt weakly coupled schemes that require significantly fewer iterations. Figure \eqref{fig:EXT} compares the convergence of a monolithic scheme, the iterated BEFE-EXT1, and the non-iterated BEFE-EXT2. The results illustrate that the iterated scheme converges to the monolithic solution with four sub-iterations, while by leveraging extrapolation in time of order two, the partitioned solution convergence to monolithic solution in one sub-iteration. The numerical experiments suggest that sub-iterations and time extrapolation at the interface can improve the accuracy of partitioned schemes. Nevertheless, we emphasize that higher-order time accuracy cannot, in general, be guaranteed with a fixed number of sub-iterations, and achieving the monolithic accuracy may require additional coupling iterations (strong coupling) depending on the time step and problem parameters.

\begin{figure}[h!]
\centering
\centering\includegraphics[width=9.25cm]{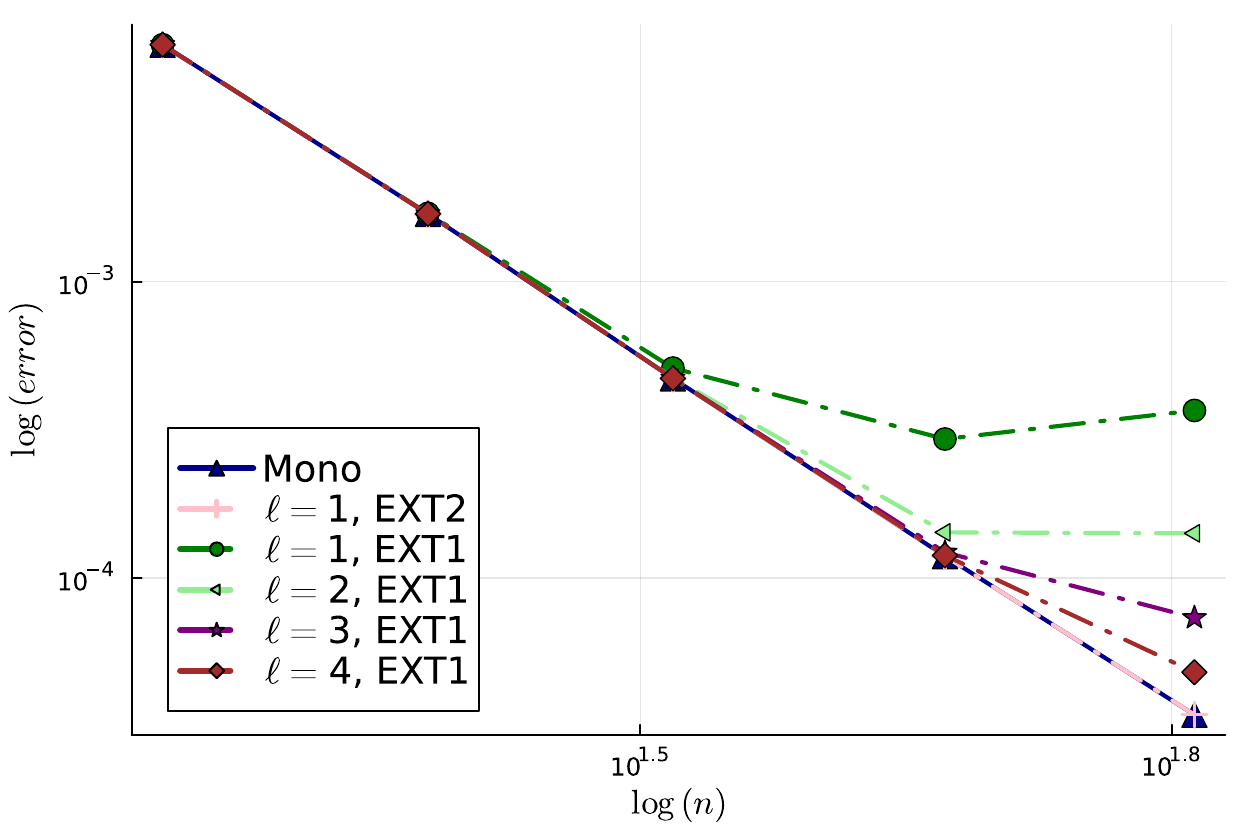}
\caption{Comparison of the convergence of  monolithic scheme, iterated BEFE-EXT1 and non-iterated BEFE-EXT2, $\delta t=10^{-4}$ and $p=3$.} 
\label{fig:EXT}
\end{figure}

In Figure \eqref{fig:DisErr1} we show the error distribution in both subdomains using the iterated partitioned scheme with  BE-EXT2. We notice that the error is generally larger near the interface and the physical boundaries.  Notice that currently we use a second derivative SBP operator of the form  
\[\DoneD{}\DoneD{}=\Pnorm{}^{-1}(-\DoneD{}\Tr \Pnorm{}\DoneD{}+\EoneD{}\DoneD{}).\] 
One way to improve the errors at interface and boundaries is to employ narrow-diagonal second-derivative SBP operators \cite{Mattsson2004b,Mattsson2012}.  

\begin{figure}[h!]
\centering
\centering\includegraphics[width=7.25cm]{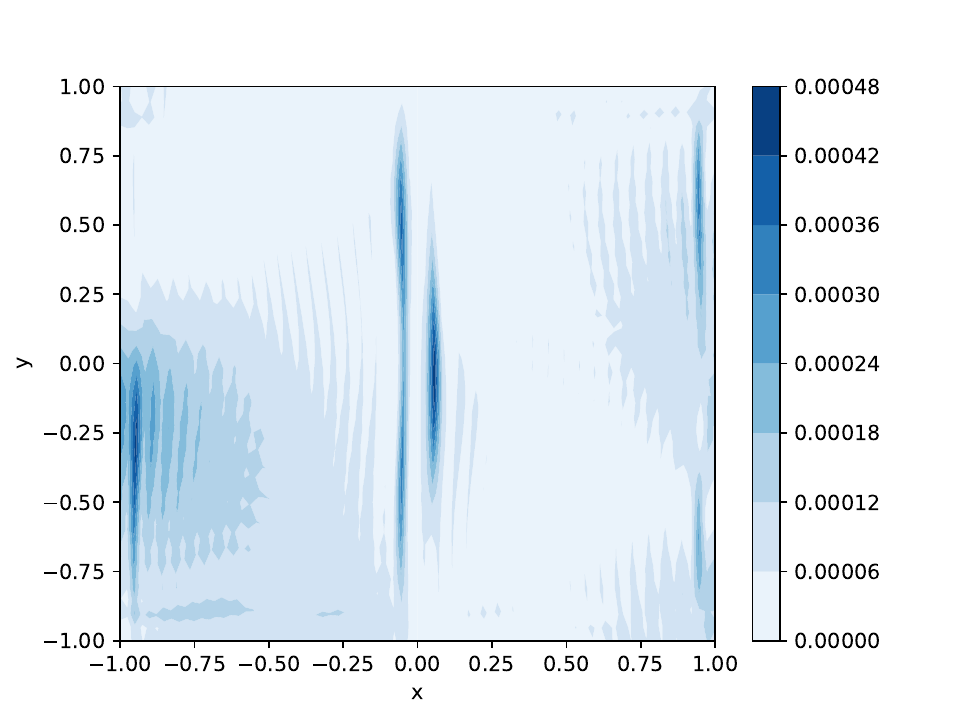}
\caption{Error distribution in using iterated partitioned scheme with  BE-EXT2, $\delta t=10^{-3}$ and $p=3$ with $n=40$ nodes in both $x$ and $y$ directions.} 
\label{fig:DisErr1}
\end{figure}

To analyze the stability of the partitioned scheme using matrix analysis, we numerically investigate the spectral properties of the corresponding time iteration matrix, as described in \cite{Peet2012}. Consider solving a semi-discrete coupled system in a sequential manner, and for the stability analysis assume a homogeneous setting (i.e., zero source terms),
\begin{eqnarray}\nonumber
&\w_t = A_{11} \w + A_{12} \v^* ,\qquad \w(0) = \q_{\mat L},\\\nonumber
&\v_t = A_{22} \v + A_{21} \w^* ,\qquad \v(0) = \q_{\mat R}.
\end{eqnarray}
Here, $A_{11}$ and $A_{22}$ represent the discretizations of the spatial differential operators in the interior of the subdomains, while $A_{12}$ and $A_{21}$ correspond to the discretizations of the interface conditions. 
We construct the time iteration matrix for the partitioned approach in following way. For simplicity, we apply BE for time discretization and a second-order time extrapolation at the interface, using a non-iterated scheme 
\begin{align*}
c\w^{k+1}-c\w^k&=A_{11} \w^{k+1}+	A_{12}(2\v^k-\v^{k-1}),\\
c\v^{k+1}-c\v^k&=A_{22} \v^{k+1}+	A_{21}\w^{k+1},
\end{align*}
where $c=1/\delta t $ and $1\le k\le N_t-1$.  Note that $k$ denotes the time level. Then 
\begin{align*}
(c I-A_{11})\w^{k+1}&=c \w^k+A_{12} (2\v^k-\v^{k-1}),\\
(c I-A_{22})\v^{k+1}&=c \v^k+A_{21} \w^{k+1}.
\end{align*}
Let 
\begin{align*}
M_{\mat L}= c (c I-A_{11})^{-1},&\qquad N_{\mat L}=  (c I-A_{11})^{-1} A_{12},\\
M_{\mat R}= c (c I-A_{22})^{-1},&\qquad N_{\mat R}=  (c I-A_{22})^{-1} A_{21}.
\end{align*}
Then
\begin{align*}
\w^{k+1}&=M_{\mat L} \w^k+N_{\mat L} (2\v^k-\v^{k-1}),\\
\v^{k+1}&=M_{\mat R}\v^k+N_{\mat R} \w^{k+1}.
\end{align*}
After doing some simplifications, the above system in a matrix form is given by
\begin{eqnarray}\label{330}
\begin{bmatrix}\w^{k+1}\\\v^{k+1}\\\w^{k}\\\v^{k}\end{bmatrix}=
\begin{bmatrix}
	M_{\mat L}&2N_{\mat L}& 0&-N_L\\
	N_{\mat R}M_{\mat L}&2N_{\mat R}N_{\mat L}+M_{\mat R}&0&-N_RN_L\\
	I&0&0&0\\
	0&I& 0&0
\end{bmatrix}\begin{bmatrix}\w^{k}\\\v^{k}\\\w^{k-1}\\\v^{k-1}\end{bmatrix}.
\end{eqnarray}
Thus, the partitioned discrete operator can be viewed as the monolithic operator plus a lower-order interface perturbation induced by the partitioning, i.e., by evaluating the interface coupling with extrapolated interface data. Stability of the partitioned scheme can be assessed by studying the spectral radius of its time-iteration matrix,
\begin{align*}\label{BmatExt2}
B=\begin{bmatrix}
	M_{\mat L}&2N_{\mat L}& 0&-N_L\\
	N_{\mat R}M_{\mat L}&2N_{\mat R}N_{\mat L}+M_{\mat R}&0&-N_RN_L\\
	I&0&0&0\\
	0&I& 0&0
\end{bmatrix}.
\end{align*}
Define  the spectral radius of $B$ as
\[\varrho(B)=\{\max |\lambda | : \lambda \text{ is an eigenvalue of} \, B\}. \]
Let $\u^{k}=[\w^{k},\,\v^{k},\,\w^{k-1},\,\v^{k-1}]^T$. Consider two solutions
$\u^{k}$ and $\tilde\u^{k}$ of the same homogeneous recurrence $\u^{k+1}=B\,\u^{k}$ with different initial histories, and define the error (perturbation)
\(
\e^{k}=\u^{k}-\tilde\u^{k}.
\)
Then $\e^{k+1}=B\,\e^{k}$, and hence
\[
\|\e^{k}\|\le \|B^{k}\|\,\|\e^{0}\|.
\]
A necessary and sufficient condition for $\lim_{k\to\infty}\|\e^k\|=0$ for arbitrary $\e^0$ is $\varrho(B)<1$ \cite[Theorem~3.2.5.2]{Horn1985}. While $\varrho(B)<1$ guarantees asymptotic decay, the matrix $B$ may be non-normal, in which case the error may exhibit transient growth before eventually decaying \cite{Trefethen2005}.
\begin{table}[ht!]
\centering
\caption{The spectral radius of the iteration matrices of time iteration matrix for BE-EXT2 for different values of \( \gamma_1, \gamma_2, \delta t \), and \( \delta y\). In the middle column, all parameters are assumed to be the same. }
\begin{tabular}{c|ccccc} 
	\multicolumn{6}{c}{} \\
	\(\gamma_1\) &  2.0 & 1.0 & 0.5& 0.25  &0.125 \\
	\hline 
	$p=1$&0.794745& 0.795275& 0.796328& 0.798413& 0.802494\\
	$p=2$&0.796160 & 0.796373 &0.796800 &0.797650 & 0.799334\\
	$p=3$&0.796695 & 0.796845 &0.797145 &0.797742&0.798929\\
	\multicolumn{6}{c}{} \\
	\(\gamma_2\) & 0.4 & 0.2 & 0.1 & 0.05 &0.01 \\
	\hline
	$p=1$&0.796325 &0.796326 &0.796328 &0.796332 &0.796368\\
	$p=2$&0.796798 &0.796799&0.796800 &0.796803 &0.7968257\\
	$p=3$&0.797143&0.797144 &0.797145&0.797147& 0.797166\\
	\multicolumn{6}{c}{} \\
	\(\delta t\) & 0.004 & 0.002 & 0.001 & 0.0005 & 0.00025 \\
	\hline
	$p=1$&0.494302& 0.661582& 0.796328 &0.886617& 0.939901\\
	$p=2$&0.495030 & 0.662234& 0.796800 &0.886910&0.940066\\
	$p=3$&0.495563 &0.662711 &0.797145&0.887123 &0.940186\\
	\multicolumn{6}{c}{} \\
	\(\delta y\)& 0.04 &  0.035  &  0.03 & 0.025 & 0.02\\
	\hline
	$p=1$ &0.794526 & 0.795557 &0.796328&0.801353&0.801425\\
	$p=2$&0.795107 &0.796078 & 0.796800&0.801210 & 0.801331\\
	$p=3$&0.795697&0.796522 &0.797145 &0.801216& 0.801372
\end{tabular}
\label{tab:var}
\end{table}

Table \ref{tab:var} demonstrates spectral radius of iteration matrices for various values of \( \gamma_1 \), \( \gamma_2=\gamma_2^{\mat L}=\gamma_2^{\mat R}\), \( \delta t \) and \( \delta y \), where we vary one parameter at a time. Note that \( \delta y \) is the mesh resolution at the interface, and \( \rho_{\mat{L}} \) and \( \rho_{\mat{R}} \) are proportional to it. This verifies that when the assumptions of Theorem \ref{thm:BE} hold, the iteration converges across  a wide range of parameters. We note that using time extrapolation at the interface with one order higher than the design accuracy, as well as performing multiple iterations at each time step in the coupled solver, is necessary to recover the design accuracy in partitioned scheme.

In Table \ref{tab:Stab}, we present accuracy of the scheme at interface for the weakly partitioned scheme BE-EXT2 for various values of $\kappa/\epsilon$ and $\delta t/\delta y ^2$, where $\delta y$ is spatial step size at the interface which is assumed to be a vertical line. The test is designed to check the accuracy of the method at interface without adequately changing the mesh resolution. 
We set $p=3$ and $n = 30$, while the time step $\delta t$ varies. With $\epsilon = 1$ held constant, $\kappa$ takes on  the set  $\{0.01, 0.1, 1, 10, 100\}$. $\gamma_1$ and $\gamma_2=\gamma_2^{\mat L}=\gamma_2^{\mat R}$ are chosen such that assumptions \eqref{b1}-\eqref{b4} holds. We report the norm $\Pnorm{}$ of the error at interface at the time $T = 2\delta t$, where a second-order extrapolation is used at the interface, and the maximum number of sub-iteration is set to $\ell = 20$. A result marked with $\times$ indicates that the error failed to satisfy a given tolerance after a given number of sub-iterations. 

The results demonstrate that when the diffusion coefficient in the left subdomain is  larger than that in the right subdomain (i.e., $\kappa/\epsilon < 1$), the scheme struggles with slow convergence and warrants sub-iterations. However, in most CHT applications, the fluid typically has a smaller thermal diffusion coefficient than the solid, making such cases ($\kappa/\epsilon < 1$) less frequent. 

\begin{table}[ht]
\caption{Accuracy of weakly partitioned scheme at interface for various values of $\kappa/\epsilon$ and $\delta t/\delta y^2$.}
\centering\scalebox{1}{
	\begin{tabular}{c l|ccccc} 
		& &&  & $\kappa/\epsilon$ &\\
		& &0.01 & 0.1  & 1 & 10 & 100  \\
		\hline 
		&0.09& $\times$  & 0.000223 &  6.3973e-14 & 0.000185 &0.000323\\
		&0.225& $\times$  & 1.3656e-5 &  1.4872e-12 & 0.000143  &0.000298\\
		$\dfrac{\delta t }{\delta y ^2}$ 
		&0.45&0.043817  &$\times$ &  1.6764e-11 & 7.9223e-5  &  0.000202\\
		&0.9&0.015051  &9.8161e-6 &  1.9046e-10  &3.9750e-6   &0.000114\\
		&1.8&$\times$ & 5.2628e-6  &2.2663e-9  & 2.0774e-6  & 5.5049e-5\\
\end{tabular}}
\label{tab:Stab}
\end{table}

Notice that in the case where \( k/\epsilon < 1 \), we can modify the scheme to stay in the stability region. Consider a coupling approach where the right sub-problem is solved first using the modified SAT terms defined in \eqref{eqS1}, with the assumption $\w^* = \w^{k}$. Subsequently, the left sub-problem is solved with modified SAT terms defined in \eqref{eqS2}, assuming $\v^* = \v^{k+1}$.  Similar stability analyses can be carried out to derive analogous conditional constraints for the specified given SAT parameters.

\section{Conclusion}\label{sec:Summary}
We proposed a weakly coupled partitioned scheme for a model CHT problem involving linear advection–diffusion and heat equations. The scheme integrates first- and second-order time discretization methods with high-order SBP-SATs operators in curvilinear coordinates system for spatial derivatives. We  applied temporal extrapolation at the interface to achieve stability and accuracy across the interface and within the subdomains. Stability is ensured by carefully designing SAT parameters at the interface, enabling the solid and fluid solvers to operate independently while maintaining stable and accurate interactions. 

The analysis presented in this work could potentially be extended to support non-conforming meshes at the interface, allowing more efficient and flexible mesh design in complex multi-physics applications. Another potential direction for future study is to extend the proposed partitioning framework to more practical problems involving turbulent heat transfer problem. This would require incorporating appropriate turbulence models and addressing additional challenges such as increased nonlinearity commonly encountered in realistic engineering applications.


\section*{Acknowledgments}
The authors thank the referees for their careful reading and valuable suggestions.

\section*{Declarations}

\begin{itemize}
	\item {\bf{Funding}} This work was partially funded by Natural Sciences and Engineering Research Council of Canada (NSERC).
	\item {\bf{Conflict of interest/Competing interests}} The authors declare no competing interests.
	\item {\bf{Ethics approval}} Not Applicable 
	\item {\bf{Consent to participate}} Not Applicable
	\item {\bf{Consent for publication}}  Not Applicable
	\item {\bf{Availability of data and materials}} Not Applicable
	\item {\bf{Authors' contributions}} The contribution of each author listed is given below.
    
\emph{Sarah Nataj:} Conceptualization, Methodology, Formal Analysis, Investigation, Software, Validation, Visualization, Writing -Original Draft, Writing -review \& editing.
    
 \emph{David C. Del Rey Fern\'andez:} Conceptualization, Methodology, Formal Analysis, Supervision, Funding Acquisition, Writing -review \& editing.

\emph{David Brown:} Conceptualization, Methodology, Writing -review.

\emph{Rajeev Jaiman:} Conceptualization, Methodology, Writing -review.  
\end{itemize}

%
%


	%
	%
 
\bibliographystyle{elsarticle-num} 
\bibliography{references}

@book {Jaiman2022,
    AUTHOR = {Jaiman, Rajeev Kumar and Joshi, Vaibhav},
     TITLE = {Computational mechanics of fluid-structure
              interaction, computational methods for coupled
              fluid-structure analysis},
 PUBLISHER = {Springer, Singapore},
      YEAR = {2022},
     PAGES = {xv+329},
      ISBN = {978-981-16-5354-4; 978-981-16-5355-1},
   MRCLASS = {74F10 (76-02 76Mxx)},
  MRNUMBER = {4390189},
}

@article {Banks2011,
	AUTHOR = {Banks, J. W. and Sj\"{o}green, B.},
	TITLE = {A normal node stability analysis of numerical interface
	conditions for fluid/structure interaction},
	JOURNAL = {Commun. Comput. Phys.},
	FJOURNAL = {Communications in Computational Physics},
	VOLUME = {10},
	YEAR = {2011},
	NUMBER = {2},
	PAGES = {279--304},
	ISSN = {1815-2406},
	MRCLASS = {65M06 (65M12)},
	MRNUMBER = {2799643},}

@article {Banks2014I,
	AUTHOR = {Banks, J. W. and Henshaw, W. D. and Schwendeman, D. W.},
	TITLE = {An analysis of a new stable partitioned algorithm for {FSI}
	problems. {P}art {I}: {I}ncompressible flow and elastic
	solids},
	JOURNAL = {J. Comput. Phys.},
	FJOURNAL = {Journal of Computational Physics},
	VOLUME = {269},
	YEAR = {2014},
	PAGES = {108--137},
	ISSN = {0021-9991},
	MRCLASS = {65M99 (74F10 76D07 76N99)},
	MRNUMBER = {3197683},}

@article {Banks2014II,
	AUTHOR = {Banks, J. W. and Henshaw, W. D. and Schwendeman, D. W.},
	TITLE = {An analysis of a new stable partitioned algorithm for {FSI}
	problems. {P}art {II}: {I}ncompressible flow and structural
	shells},
	JOURNAL = {J. Comput. Phys.},
	FJOURNAL = {Journal of Computational Physics},
	VOLUME = {268},
	YEAR = {2014},
	PAGES = {399--416},
	ISSN = {0021-9991},
	MRCLASS = {65M70 (74F10 76Bxx 76Dxx)},
	MRNUMBER = {3192450},}

@article {Banks2016,
	AUTHOR = {Banks, J. W. and Henshaw, W. D. and Kapila, A. K. and
	Schwendeman, D. W.},
	TITLE = {An added-mass partition algorithm for fluid-structure
	interactions of compressible fluids and nonlinear solids},
	JOURNAL = {J. Comput. Phys.},
	FJOURNAL = {Journal of Computational Physics},
	VOLUME = {305},
	YEAR = {2016},
	PAGES = {1037--1064},
	ISSN = {0021-9991},
	MRCLASS = {65M06 (74F10 76N99)},
	MRNUMBER = {3429618},}

@article {Causin2005,
	AUTHOR = {Causin, P. and Gerbeau, J. F. and Nobile, F.},
	TITLE = {Added-mass effect in the design of partitioned algorithms for
	fluid-structure problems},
	JOURNAL = {Comput. Methods Appl. Mech. Engrg.},
	FJOURNAL = {Computer Methods in Applied Mechanics and Engineering},
	VOLUME = {194},
	YEAR = {2005},
	NUMBER = {42-44},
	PAGES = {4506--4527},
	ISSN = {0045-7825},
	MRCLASS = {74F10 (74S30 76M25 76Z05 92C35)},
	MRNUMBER = {2157973},}

@article {Hou2012,
	AUTHOR = {Hou, Gene and Wang, Jin and Layton, Anita},
	TITLE = {Numerical methods for fluid-structure interaction, a review},
	JOURNAL = {Commun. Comput. Phys.},
	FJOURNAL = {Communications in Computational Physics},
	VOLUME = {12},
	YEAR = {2012},
	NUMBER = {2},
	PAGES = {337--377},
	ISSN = {1815-2406},
	MRCLASS = {76-02},
	MRNUMBER = {2897143},}

@article {Fernandez2009,
	AUTHOR = {Burman, Erik and Fern\'{a}ndez, Miguel A.},
	TITLE = {Stabilization of explicit coupling in fluid-structure
	interaction involving fluid incompressibility},
	JOURNAL = {Comput. Methods Appl. Mech. Engrg.},
	FJOURNAL = {Computer Methods in Applied Mechanics and Engineering},
	VOLUME = {198},
	YEAR = {2009},
	NUMBER = {5-8},
	PAGES = {766--784},
	ISSN = {0045-7825},
	MRCLASS = {74F10 (76D10)},
	MRNUMBER = {2498525},

}

@article {Svard2014,
	AUTHOR = {Sv\"{a}rd, Magnus and Nordstr\"{o}m, Jan},
	TITLE = {Review of summation-by-parts schemes for
	initial-boundary-value problems},
	JOURNAL = {J. Comput. Phys.},
	FJOURNAL = {Journal of Computational Physics},
	VOLUME = {268},
	YEAR = {2014},
	PAGES = {17--38},
	ISSN = {0021-9991,1090-2716},
	MRCLASS = {65M06 (65M12)},
	MRNUMBER = {3192433},
}

@article {DelRey2014,
	AUTHOR = {Del Rey Fern\'{a}ndez, David C. and Hicken, Jason E. and
	Zingg, David W.},
	TITLE = {Review of summation-by-parts operators with simultaneous approximation terms for the numerical solution of partial differential equations},
	JOURNAL = {Comput. \& Fluids},
	FJOURNAL = {Computers \& Fluids. An International Journal},
	VOLUME = {95},
	YEAR = {2014},
	PAGES = {171--196},
	ISSN = {0045-7930,1879-0747},
	MRCLASS = {65M06},
	MRNUMBER = {3189003},
	MRREVIEWER = {Sergio\ Amat},
}

@article {Mattsson2003,
    AUTHOR = {Mattsson, Ken},
     TITLE = {Boundary procedures for summation-by-parts operators},
   JOURNAL = {J. Sci. Comput.},
  FJOURNAL = {Journal of Scientific Computing},
    VOLUME = {18},
      YEAR = {2003},
    NUMBER = {1},
     PAGES = {133--153},
      ISSN = {0885-7474,1573-7691},
   MRCLASS = {65M06 (76M20)},
  MRNUMBER = {1958938},
}

@article {Mattsson2004b,
	AUTHOR = {Mattsson, Ken and Nordstr\"{o}m, Jan},
	TITLE = {Summation by parts operators for finite difference
	approximations of second derivatives},
	JOURNAL = {J. Comput. Phys.},
	FJOURNAL = {Journal of Computational Physics},
	VOLUME = {199},
	YEAR = {2004},
	NUMBER = {2},
	PAGES = {503--540},
	ISSN = {0021-9991,1090-2716},
	MRCLASS = {65M06},
	MRNUMBER = {2091906},
}

@article {Mattsson2012,
    AUTHOR = {Mattsson, Ken},
     TITLE = {Summation by parts operators for finite difference
              approximations of second-derivatives with variable
              coefficients},
   JOURNAL = {J. Sci. Comput.},
  FJOURNAL = {Journal of Scientific Computing},
    VOLUME = {51},
      YEAR = {2012},
    NUMBER = {3},
     PAGES = {650--682},
      ISSN = {0885-7474,1573-7691},
   MRCLASS = {65M06 (65M12)},
  MRNUMBER = {2914426},
MRREVIEWER = {Tore\ Fl\aa tten},
   
}

@article {Roe2008,
	AUTHOR = {Roe, B. and Jaiman, R. and Haselbacher, A. and Geubelle, P.
	H.},
	TITLE = {Combined interface boundary condition method for coupled
	thermal simulations},
	JOURNAL = {Internat. J. Numer. Methods Fluids},
	FJOURNAL = {International Journal for Numerical Methods in Fluids},
	VOLUME = {57},
	YEAR = {2008},
	NUMBER = {3},
	PAGES = {329--354},
	ISSN = {0271-2091,1097-0363},
	MRCLASS = {80A20 (65M06)},
	MRNUMBER = {2410154},
}

@article {Meng2017,
	AUTHOR = {Meng, F. and Banks, J. W. and Henshaw, W. D. and Schwendeman,
	D. W.},
	TITLE = {A stable and accurate partitioned algorithm for conjugate heat
	transfer},
	JOURNAL = {J. Comput. Phys.},
	FJOURNAL = {Journal of Computational Physics},
	VOLUME = {344},
	YEAR = {2017},
	PAGES = {51--85},
	ISSN = {0021-9991,1090-2716},
	MRCLASS = {65M06 (65M55 74F10 76M20 80A20)},
	MRNUMBER = {3656731},
	MRREVIEWER = {Elena\ Zampieri},
}

@article {DelRey2020p,
	AUTHOR =  {Del Rey Fern\'{a}ndez, David C.  and Carpenter, Mark H. and
	Dalcin, Lisandro and Fredrich, Lucas and Winters, Andrew R.
	and Gassner, Gregor J. and Parsani, Matteo},
	TITLE = {Entropy-stable {$p$}-nonconforming discretizations with the
	summation-by-parts property for the compressible
	{N}avier-{S}tokes equations},
	JOURNAL = {Comput. \& Fluids},
	FJOURNAL = {Computers \& Fluids. An International Journal},
	VOLUME = {210},
	YEAR = {2020},
	PAGES = {104631, 14},
	ISSN = {0045-7930,1879-0747},
	MRCLASS = {65M60 (76N06)},
	MRNUMBER = {4127206},
}

@article {DelRey2019,
	AUTHOR = {Del Rey Fern\'{a}ndez, David C. and Boom, Pieter D. and
	Carpenter, Mark H. and Zingg, David W.},
	TITLE = {Extension of tensor-product generalized and dense-norm
	summation-by-parts operators to curvilinear coordinates},
	JOURNAL = {J. Sci. Comput.},
	FJOURNAL = {Journal of Scientific Computing},
	VOLUME = {80},
	YEAR = {2019},
	NUMBER = {3},
	PAGES = {1957--1996},
	ISSN = {0885-7474,1573-7691},
	MRCLASS = {65M06 (65M60 65M70)},
	MRNUMBER = {3995995},
}

@article {Kazemi2013,
	AUTHOR = {Kazemi-Kamyab, V. and van Zuijlen, A. H. and Bijl, H.},
	TITLE = {A high order time-accurate loosely-coupled solution algorithm
	for unsteady conjugate heat transfer problems},
	JOURNAL = {Comput. Methods Appl. Mech. Engrg.},
	FJOURNAL = {Computer Methods in Applied Mechanics and Engineering},
	VOLUME = {264},
	YEAR = {2013},
	PAGES = {205--217},
	ISSN = {0045-7825,1879-2138},
	MRCLASS = {80A20 (65M06 80M20)},
	MRNUMBER = {3093597},
}

@article {Giles1997,
	AUTHOR = {Giles, M. B.},
	TITLE = {Stability analysis of numerical interface conditions in
	fluid--structure thermal analysis},
	JOURNAL = {Internat. J. Numer. Methods Fluids},
	FJOURNAL = {International Journal for Numerical Methods in Fluids},
	VOLUME = {25},
	YEAR = {1997},
	NUMBER = {4},
	PAGES = {421--436},
	ISSN = {0271-2091,1097-0363},
	MRCLASS = {73K70 (73V20 76M25)},
	MRNUMBER = {1463580},
	
}

@article {Henshaw2009,
	AUTHOR = {Henshaw, William D. and Chand, Kyle K.},
	TITLE = {A composite grid solver for conjugate heat transfer in
	fluid-structure systems},
	JOURNAL = {J. Comput. Phys.},
	FJOURNAL = {Journal of Computational Physics},
	VOLUME = {228},
	YEAR = {2009},
	NUMBER = {10},
	PAGES = {3708--3741},
	ISSN = {0021-9991,1090-2716},
	MRCLASS = {80M25 (65M06 65M08 65M12 76M25)},
	MRNUMBER = {2511073},
	MRREVIEWER = {Reinhard\ Redlinger},
}

@article{Felippa2001,
	title = {Partitioned analysis of coupled mechanical systems},
	journal = {Computer Methods in Applied Mechanics and Engineering},
	volume = {190},
	number = {24},
	pages = {3247-3270},
	year = {2001},
	note = {Advances in Computational Methods for Fluid-Structure Interaction},
	issn = {0045-7825},
	author = {Carlos A. Felippa and K.C. Park and Charbel Farhat},
}

@article{Joshi2014,
	author = {Joshi, Ojas and Leyland, P\'{e}n\'{e}lope},
	title = {Stability Analysis of a Partitioned Fluid–-Structure Thermal Coupling Algorithm},
	journal = {Journal of Thermophysics and Heat Transfer},
	volume = {28},
	number = {1},
	pages = {59-67},
	year = {2014},
}

@article{Peet2012,
	author = {Peet, Yulia T. and Fischer, Paul F.},
	title = {Stability Analysis of Interface Temporal Discretization in Grid Overlapping Methods},
	journal = {SIAM Journal on Numerical Analysis},
	volume = {50},
	number = {6},
	pages = {3375-3401},
	year = {2012},
}

@article{Lee2003,
	author = {In Lee, Jin-Ho Roh and Il-Kwon Oh},
	title = {Aerothermoelastic phenomena of aerospace and composite structures},
	journal = {Journal of Thermal Stresses},
	volume = {26},
	number = {6},
	pages = {525-546},
	year = {2003},
	publisher = {Taylor & Francis},
}

@article {ElKaramany2003,
	AUTHOR = {El-Karamany, Ahmed S.},
	TITLE = {Boundary integral equation formulation in generalized
	thermoviscoelasticity with rheological volume and density in
	material having temperature-dependent properties},
	JOURNAL = {J. Thermal Stresses},
	FJOURNAL = {Journal of Thermal Stresses},
	VOLUME = {26},
	YEAR = {2003},
	NUMBER = {2},
	PAGES = {123--147},
	ISSN = {0149-5739,1521-074X},
	MRCLASS = {74F05 (74D05 74S15)},
	MRNUMBER = {1955447},
}

@article {Lindstrom2010,
	AUTHOR = {Lindstr\"{o}m, Jens and Nordstr\"{o}m, Jan},
	TITLE = {A stable and high-order accurate conjugate heat transfer
	problem},
	JOURNAL = {J. Comput. Phys.},
	FJOURNAL = {Journal of Computational Physics},
	VOLUME = {229},
	YEAR = {2010},
	NUMBER = {14},
	PAGES = {5440--5456},
	ISSN = {0021-9991,1090-2716},
	MRCLASS = {80A20 (65M06)},
	MRNUMBER = {2646148},
}

@article {Carpenter2010,
	AUTHOR = {Carpenter, Mark H. and Nordstr\"{o}m, Jan and Gottlieb, David},
	TITLE = {Revisiting and extending interface penalties for multi-domain
	summation-by-parts operators},
	JOURNAL = {J. Sci. Comput.},
	FJOURNAL = {Journal of Scientific Computing},
	VOLUME = {45},
	YEAR = {2010},
	NUMBER = {1-3},
	PAGES = {118--150},
	ISSN = {0885-7474,1573-7691},
	MRCLASS = {65M60 (65M06 65M12)},
	MRNUMBER = {2679793},
	MRREVIEWER = {Istv\'{a}n\ Farag\'{o}},	
}

@article {Carpenter1999,
    AUTHOR = {Carpenter, Mark H. and Nordstr\"om, Jan and Gottlieb, David},
     TITLE = {A stable and conservative interface treatment of arbitrary
              spatial accuracy},
   JOURNAL = {J. Comput. Phys.},
  FJOURNAL = {Journal of Computational Physics},
    VOLUME = {148},
      YEAR = {1999},
    NUMBER = {2},
     PAGES = {341--365},
      ISSN = {0021-9991,1090-2716},
   MRCLASS = {76M20},
  MRNUMBER = {1669703},
}

@article {Gong2011,
	AUTHOR = {Gong, Jing and Nordstr\"{o}m, Jan},
	TITLE = {Interface procedures for finite difference approximations of
	the advection-diffusion equation},
	JOURNAL = {J. Comput. Appl. Math.},
	FJOURNAL = {Journal of Computational and Applied Mathematics},
	VOLUME = {236},
	YEAR = {2011},
	NUMBER = {5},
	PAGES = {602--620},
	ISSN = {0377-0427,1879-1778},
	MRCLASS = {65M06 (65M12)},
	MRNUMBER = {2853486},
	MRREVIEWER = {M.\ K.\ Kadalbajoo},
	
}

@article {Nordstrom2013,
	AUTHOR = {Nordstr\"{o}m, Jan and Berg, Jens},
	TITLE = {Conjugate heat transfer for the unsteady compressible
	{N}avier-{S}tokes equations using a multi-block coupling},
	JOURNAL = {Comput. \& Fluids},
	FJOURNAL = {Computers \& Fluids. An International Journal},
	VOLUME = {72},
	YEAR = {2013},
	PAGES = {20--29},
	ISSN = {0045-7930,1879-0747},
	MRCLASS = {76N10 (65M06 65M12 76M20 80A20)},
	MRNUMBER = {3035902},
}

@article {Ghasemi2017,
	AUTHOR = {Ghasemi, Fatemeh and Nordstr\"{o}m, Jan},
	TITLE = {Coupling requirements for multiphysics problems posed on two
	domains},
	JOURNAL = {SIAM J. Numer. Anal.},
	FJOURNAL = {SIAM Journal on Numerical Analysis},
	VOLUME = {55},
	YEAR = {2017},
	NUMBER = {6},
	PAGES = {2885--2904},
	ISSN = {0036-1429,1095-7170},
	MRCLASS = {65M06 (35L50 65M12 74F99)},
	MRNUMBER = {3725281},

}

@article {Lundquist2018,
	AUTHOR = {Lundquist, Tomas and Malan, Arnaud and Nordstr\"{o}m, Jan},
	TITLE = {A hybrid framework for coupling arbitrary summation-by-parts
	schemes on general meshes},
	JOURNAL = {J. Comput. Phys.},
	FJOURNAL = {Journal of Computational Physics},
	VOLUME = {362},
	YEAR = {2018},
	PAGES = {49--68},
	ISSN = {0021-9991,1090-2716},
	MRCLASS = {65M08 (65M12 65M50)},
	MRNUMBER = {3774923},
	MRREVIEWER = {Ali\ R.\ Soheili},

}

@article {Burkardt2020,
	AUTHOR = {Burkardt, John and Trenchea, Catalin},
	TITLE = {Refactorization of the midpoint rule},
	JOURNAL = {Appl. Math. Lett.},
	FJOURNAL = {Applied Mathematics Letters. An International Journal of Rapid
	Publication},
	VOLUME = {107},
	YEAR = {2020},
	PAGES = {106438},
	ISSN = {0893-9659,1873-5452},
	MRCLASS = {65L05},
	MRNUMBER = {4092601},
}

@article {Bukac2022,
	AUTHOR = {Buka\v{c}, Martina and Trenchea, Catalin},
	TITLE = {Adaptive, second-order, unconditionally stable partitioned
	method for fluid-structure interaction},
	JOURNAL = {Comput. Methods Appl. Mech. Engrg.},
	FJOURNAL = {Computer Methods in Applied Mechanics and Engineering},
	VOLUME = {393},
	YEAR = {2022},
	PAGES = {Paper No. 114847, 24},
	ISSN = {0045-7825,1879-2138},
	MRCLASS = {74F10},
	MRNUMBER = {4398357},
	MRREVIEWER = {John\ M.\ Stockie},
}

@article {Dahlquist1983,
    AUTHOR = {Dahlquist, Germund G. and Liniger, Werner and Nevanlinna,
              Olavi},
     TITLE = {Stability of two-step methods for variable integration steps},
   JOURNAL = {SIAM J. Numer. Anal.},
  FJOURNAL = {SIAM Journal on Numerical Analysis},
    VOLUME = {20},
      YEAR = {1983},
    NUMBER = {5},
     PAGES = {1071--1085},
      ISSN = {0036-1429},
   MRCLASS = {65L20},
  MRNUMBER = {714701},
MRREVIEWER = {Kevin\ Burrage},     
}

@article {Sun2024,
    AUTHOR = {Sun, Yupeng and Yao, Song and Alexandersen, Joe},
     TITLE = {Topography optimisation using a reduced-dimensional model for
              transient conjugate heat transfer between fluid channels and
              solid plates with volumetric heat source},
   JOURNAL = {Struct. Multidiscip. Optim.},
  FJOURNAL = {Structural and Multidisciplinary Optimization},
    VOLUME = {67},
      YEAR = {2024},
    NUMBER = {4},
     PAGES = {Paper No. 45, 27},
      ISSN = {1615-147X,1615-1488},
   MRCLASS = {74P15 (80A19)},
  MRNUMBER = {4717338},
       
}

@article{Zeng2020,
title = {Topology optimization of heat sinks for instantaneous chip cooling using a transient pseudo-3D thermofluid model},
journal = {International Journal of Heat and Mass Transfer},
volume = {154},
pages = {119681},
year = {2020},
issn = {0017-9310},
author = {Tao Zeng and Hu Wang and Mengzhu Yang and Joe Alexandersen},
}

@article {Cang2022,
    AUTHOR = {Cang, Yu and Hu, Yue and Wang, Lipo and Shen, Yongxing},
     TITLE = {Numerical scheme solving the temperature of the interface
              between gas and heterogeneous solid with phase change},
   JOURNAL = {Internat. J. Numer. Methods Engrg.},
  FJOURNAL = {International Journal for Numerical Methods in Engineering},
    VOLUME = {123},
      YEAR = {2022},
    NUMBER = {4},
     PAGES = {1036--1056},
      ISSN = {0029-5981,1097-0207},
   MRCLASS = {74N99},
  MRNUMBER = {4375326},
  
}

@article {Nordstrom2001,
    AUTHOR = {Nordstr\"om, Jan and Carpenter, Mark H.},
     TITLE = {High-order finite difference methods, multidimensional linear problems, and curvilinear coordinates},
   JOURNAL = {J. Comput. Phys.},
  FJOURNAL = {Journal of Computational Physics},
    VOLUME = {173},
      YEAR = {2001},
    NUMBER = {1},
     PAGES = {149--174},
      ISSN = {0021-9991,1090-2716},
   MRCLASS = {65M06 (65M12)},
  MRNUMBER = {1857624},
    
}

@article {Crean2018,
    AUTHOR = {Crean, Jared and Hicken, Jason E. and Del Rey Fern\'andez,
              David C. and Zingg, David W. and Carpenter, Mark H.},
     TITLE = {Entropy-stable summation-by-parts discretization of the
              {E}uler equations on general curved elements},
   JOURNAL = {J. Comput. Phys.},
  FJOURNAL = {Journal of Computational Physics},
    VOLUME = {356},
      YEAR = {2018},
     PAGES = {410--438},
      ISSN = {0021-9991,1090-2716},
   MRCLASS = {65M06 (65M12 76Nxx)},
  MRNUMBER = {3743647},
MRREVIEWER = {Chenglie\ Hu},
}

@article {Nolasco2020,
    AUTHOR = {Nolasco, Irving Reyna and Dalcin, Lisandro and Del Rey
              Fern\'andez, David C. and Zampini, Stefano and Parsani,
              Matteo},
     TITLE = {Optimized geometrical metrics satisfying free-stream
              preservation},
   JOURNAL = {Comput. \& Fluids},
  FJOURNAL = {Computers \& Fluids. An International Journal},
    VOLUME = {207},
      YEAR = {2020},
     PAGES = {104555, 16},
      ISSN = {0045-7930,1879-0747},
   MRCLASS = {65M99 (65Y20 76-04 76Nxx)},
  MRNUMBER = {4109353},
}

@article{Hicken2008,
author = {Hicken, Jason E. and Zingg, David W.},
title = {Parallel Newton-Krylov Solver for the {E}uler equations Discretized Using Simultaneous Approximation Terms},
journal = {AIAA Journal},
volume = {46},
number = {11},
pages = {2773-2786},
year = {2008},
}

@article {Kozdon2013,
    AUTHOR = {Kozdon, Jeremy E. and Dunham, Eric M. and Nordstr\"om, Jan},
     TITLE = {Simulation of dynamic earthquake ruptures in complex
              geometries using high-order finite difference methods},
   JOURNAL = {J. Sci. Comput.},
  FJOURNAL = {Journal of Scientific Computing},
    VOLUME = {55},
      YEAR = {2013},
    NUMBER = {1},
     PAGES = {92--124},
      ISSN = {0885-7474,1573-7691},
   MRCLASS = {86A17 (65M06 74S20 86A15)},
  MRNUMBER = {3030705},
MRREVIEWER = {T.\ C.\ Mohan},
}

@article {Mattsson2006,
    AUTHOR = {Mattsson, Ken and Nordstr\"om, Jan},
     TITLE = {High order finite difference methods for wave propagation in
              discontinuous media},
   JOURNAL = {J. Comput. Phys.},
  FJOURNAL = {Journal of Computational Physics},
    VOLUME = {220},
      YEAR = {2006},
    NUMBER = {1},
     PAGES = {249--269},
      ISSN = {0021-9991,1090-2716},
   MRCLASS = {65M06},
  MRNUMBER = {2281629},
MRREVIEWER = {Konstantin\ Lipnikov},
}

@article {Carpenter2014,
    AUTHOR = {Carpenter, Mark H. and Fisher, Travis C. and Nielsen, Eric J.
              and Frankel, Steven H.},
     TITLE = {Entropy stable spectral collocation schemes for the
              {N}avier-{S}tokes equations: discontinuous interfaces},
   JOURNAL = {SIAM J. Sci. Comput.},
  FJOURNAL = {SIAM Journal on Scientific Computing},
    VOLUME = {36},
      YEAR = {2014},
    NUMBER = {5},
     PAGES = {B835--B867},
      ISSN = {1064-8275,1095-7197},
   MRCLASS = {65M70 (65M12 65M60)},
  MRNUMBER = {3268619},
}

@article {Strand1994,
    AUTHOR = {Strand, Bo},
     TITLE = {Summation by parts for finite difference approximations for
              {$d/dx$}},
   JOURNAL = {J. Comput. Phys.},
  FJOURNAL = {Journal of Computational Physics},
    VOLUME = {110},
      YEAR = {1994},
    NUMBER = {1},
     PAGES = {47--67},
      ISSN = {0021-9991,1090-2716},
   MRCLASS = {65M06},
  MRNUMBER = {1259900},
}

@article {Carpenter1994,
    AUTHOR = {Carpenter, Mark H. and Gottlieb, David and Abarbanel, Saul},
     TITLE = {Time-stable boundary conditions for finite-difference schemes
              solving hyperbolic systems: methodology and application to
              high-order compact schemes},
   JOURNAL = {J. Comput. Phys.},
  FJOURNAL = {Journal of Computational Physics},
    VOLUME = {111},
      YEAR = {1994},
    NUMBER = {2},
     PAGES = {220--236},
      ISSN = {0021-9991,1090-2716},
   MRCLASS = {65M06 (65N06)},
  MRNUMBER = {1275021},
}

@article {Nordstrom1999,
    AUTHOR = {Nordstr\"om, Jan and Carpenter, Mark H.},
     TITLE = {Boundary and interface conditions for high-order
              finite-difference methods applied to the {E}uler and
              {N}avier-{S}tokes equations},
   JOURNAL = {J. Comput. Phys.},
  FJOURNAL = {Journal of Computational Physics},
    VOLUME = {148},
      YEAR = {1999},
    NUMBER = {2},
     PAGES = {621--645},
      ISSN = {0021-9991,1090-2716},
   MRCLASS = {76M20},
  MRNUMBER = {1669660},
}

@article {DelRey2018,
    AUTHOR = {Del Rey Fern\'andez, David C. and Hicken, Jason E. and Zingg,
              David W.},
     TITLE = {Simultaneous approximation terms for multi-dimensional
              summation-by-parts operators},
   JOURNAL = {J. Sci. Comput.},
  FJOURNAL = {Journal of Scientific Computing},
    VOLUME = {75},
      YEAR = {2018},
    NUMBER = {1},
     PAGES = {83--110},
      ISSN = {0885-7474,1573-7691},
   MRCLASS = {65M06 (65M12)},
  MRNUMBER = {3770313},
}

@article {Hicken2016,
    AUTHOR = {Hicken, Jason E. and Del Rey Fern\'andez, David C. and Zingg,
              David W.},
     TITLE = {Multidimensional summation-by-parts operators: general theory
              and application to simplex elements},
   JOURNAL = {SIAM J. Sci. Comput.},
  FJOURNAL = {SIAM Journal on Scientific Computing},
    VOLUME = {38},
      YEAR = {2016},
    NUMBER = {4},
     PAGES = {A1935--A1958},
      ISSN = {1064-8275,1095-7197},
   MRCLASS = {65N06 (65N12)},
  MRNUMBER = {3519138},
MRREVIEWER = {Juan\ R.\ Torregrosa},
}

@book {Horn1985,
    AUTHOR = {Horn, Roger A. and Johnson, Charles R.},
     TITLE = {Matrix analysis},
 PUBLISHER = {Cambridge University Press, Cambridge},
      YEAR = {1985},
     PAGES = {xiii+561},
      ISBN = {0-521-30586-1},
   MRCLASS = {15-01},
  MRNUMBER = {832183},
MRREVIEWER = {Shao\ Kuan\ Li},
}

@article {Sjogreen2013,
    AUTHOR = {Sj\"ogreen, Bj\"orn and Banks, Jeffrey W.},
     TITLE = {Stability of finite difference discretizations of
              multi-physics interface conditions},
   JOURNAL = {Commun. Comput. Phys.},
  FJOURNAL = {Communications in Computational Physics},
    VOLUME = {13},
      YEAR = {2013},
    NUMBER = {2},
     PAGES = {386--410},
      ISSN = {1815-2406,1991-7120},
   MRCLASS = {65M06 (65M12 76M20)},
  MRNUMBER = {2948022},
}

@Article{Thomas_1979,
  author    = {Thomas, P. D. and Lombard, C. K.},
  title     = {Geometric Conservation Law and Its Application to Flow Computations on Moving Grids},
  journal   = {AIAA Journal},
  year      = {1979},
  volume    = {17},
  number    = {10},
  month     = oct,
  pages     = {1030--1037},
  issn      = {1533-385X},
  publisher = {American Institute of Aeronautics and Astronautics (AIAA)},
}

@Article{Vinokur_2001,
  author    = {Vinokur, M. and Yee, H.C.},
  title     = {Extension of Efficient Low Dissipation High Order Schemes for 3-D Curvilinear Moving Grids},
  year      = {2001},
  month     = dec,
  pages     = {129--164},
  booktitle = {Frontiers of Computational Fluid Dynamics 2002},
  isbn      = {9789812810793},
  publisher = {WORLD SCIENTIFIC},
}

@Article{Nolasco_2020,
  author    = {Nolasco, Irving Reyna and Dalcin, Lisandro and Del Rey Fernández, David C. and Zampini, Stefano and Parsani, Matteo},
  title     = {Optimized geometrical metrics satisfying free-stream preservation},
  journal   = {Computers \&amp; Fluids},
  year      = {2020},
  volume    = {207},
  month     = jul,
  pages     = {104555},
  issn      = {0045-7930},
  publisher = {Elsevier BV},
}

@book {Trefethen2005,
    AUTHOR = {Trefethen, Lloyd N. and Embree, Mark},
     TITLE = {Spectra and pseudospectra},
      NOTE = {The behavior of nonnormal matrices and operators},
 PUBLISHER = {Princeton University Press, Princeton, NJ},
      YEAR = {2005},
     PAGES = {xviii+606},
      ISBN = {978-0-691-11946-5; 0-691-11946-5},
   MRCLASS = {15-02 (15A18 47A10 47A50)},
  MRNUMBER = {2155029},
MRREVIEWER = {David\ Scott\ Watkins},
}

\appendix

\section{Stability for coupled PDEs}\label{apdA}
To show that PDEs on left and right subdomains together with coupling conditions are result in a well-posed problem. For a given $\Omega\subset\Rs^d$, define the inner product
$(\fnc{U},\fnc{V})_\Omega:=\int_{\Omega} \fnc{U}\,\fnc{V}\,dx$
for $\fnc{U},\fnc{V}\in L^2(\Omega)$ and the associated norm
$\|\fnc{U}\|_\Omega^2 := (\fnc{U},\fnc{U})_\Omega=\int_{\Omega} \fnc{U}^2\,dx$.  We calculate the unweighted $L^2$ energy associated with the nondimensionalized variables; with the interface conditions \eqref{eq03c} this choice yields exact cancellation of the interface contributions in the continuous energy estimate.
 We begin with multiplying \eqref{eq03a} with $\fnc{W}$, integrate over the associated domain to get
\begin{align}\label{eq02a}
\begin{aligned}
	\int_{\Omega_L}\fnc{W}\fnc{W}_t\,dx+\int_{\Omega_L}\fnc{W}(\a\cdot \nabla \fnc{W})\,dx=\epsilon \int_{\Omega_L}\fnc{W}\Delta \fnc{W} \,dx.
\end{aligned}
\end{align}
Notice that
\begin{align}\label{eq02b}
\begin{aligned}
	\int_{\Omega_L}\fnc{W}(\a\cdot \nabla \fnc{W})\,dx&=\dfrac12\int_{\partial\Omega_L}(\a\cdot\n_L) \fnc{W}^2 \,ds,
\end{aligned}
\end{align}
assuming $\a$ is constant vector. Also
\begin{align}\label{eq02c}
\begin{aligned}
	\int_{\Omega_L}\fnc{W}\Delta \fnc{W} \,dx= \int_{\partial\Omega_L} \fnc{W}(\n_L\cdot \nabla \fnc{W})\,ds-\int_{\Omega_L}\nabla \fnc{W}\cdot\nabla \fnc{W}\,dx.
\end{aligned}
\end{align}
Substituting \eqref{eq02b} and \eqref{eq02c} in \eqref{eq02a} we obtain
\begin{align*}
\begin{aligned}
	\dfrac12\dfrac{\partial}{\partial t}\int_{\Omega_L}\fnc{W}^2\,dx+\epsilon \int_{\Omega_L}|\nabla \fnc{W}|^2\,dx
	&=-\dfrac12\int_{\partial \Omega_L}(\a\cdot\n_L)\fnc{W}^2 \,ds+\epsilon\int_{\partial \Omega_L}\fnc{W}(\n_L\cdot\nabla \fnc{W})\,ds\\
	&=-\dfrac12\int_{\partial \Omega_L\backslash\Sigma}(\a\cdot\n_L)\fnc{W}^2 \,ds+\epsilon\int_{\partial \Omega_L\backslash\Sigma}\fnc{W}(\n_L\cdot\nabla \fnc{W})\,ds\\
	&\quad-\dfrac12\int_{\Sigma}(\a\cdot\n_L)\fnc{W}^2 \,ds+\epsilon\int_{\Sigma}\fnc{W}(\n_L\cdot\nabla \fnc{W})\,ds.
\end{aligned}
\end{align*}
Using the boundary condition $\zeta \fnc{W}+\epsilon\, \n_L\cdot \nabla \fnc{W}=g(x, t)$, we obtain
\begin{align}\label{eq02d}
\begin{aligned}
	\dfrac12\dfrac{\partial}{\partial t}\int_{\Omega_L}\fnc{W}^2\,dx+\epsilon \int_{\Omega_L}|\nabla \fnc{W}|^2\,dx
	&=-\int_{\partial \Omega_L\backslash\Sigma}\big(\dfrac12\a\cdot \n_L+\zeta\big)\,\fnc{W}^2 \,ds+\int_{\partial \Omega_L\backslash\Sigma} \fnc{W}g\,ds\\
	&\quad-\dfrac12\int_{\Sigma}(\a\cdot\n_L)\fnc{W}^2 \,ds+\epsilon\int_{\Sigma}\fnc{W}(\n_L\cdot\nabla \fnc{W})\,ds.
\end{aligned}
\end{align}
Next multiply \eqref{eq03b} with $\fnc{V}$, integrate over the associated domain and use integration by part to obtain
\begin{align*}
\begin{aligned}
	\int_{\Omega_R} \fnc{V}\fnc{V}_t\,dx&=\kappa \int_{\Omega_R}\fnc{V}\Delta \fnc{V}\,dx
	=\kappa\int_{\partial\Omega_R} \fnc{V}(\n_R\cdot \nabla \fnc{V}) \,ds-\kappa\int_{\Omega_R}\nabla \fnc{V}\cdot\nabla \fnc{V}\,dx\\
	&=\kappa\int_{\partial\Omega_R\backslash\Sigma} \fnc{V}(\n_R\cdot \nabla \fnc{V}) \,ds+\kappa\int_{\Sigma} \fnc{V}(\n_R\cdot \nabla \fnc{V}) \,ds-\kappa\int_{\Omega_R}\nabla \fnc{V}\cdot\nabla \fnc{V}\,dx
\end{aligned}
\end{align*}
Using the boundary condition  $\fnc{V}+\varphi\, \n_R\cdot \nabla \fnc{V}=h(x, t)$, we get
\begin{align}\label{eq02e}
\begin{aligned}
	\dfrac12\dfrac{\partial}{\partial t}\int_{\Omega_R} \fnc{V}^2\,dx+\kappa\int_{\Omega_R}|\nabla \fnc{V}|^2\,dx&=\dfrac{\kappa}{\varphi}\int_{\partial\Omega_R\backslash\Sigma} \fnc{V}h \,ds-\dfrac{\kappa}{\varphi}\int_{\partial\Omega_R\backslash\Sigma}\fnc{V}^2 \,ds+\kappa\int_{\Sigma} \fnc{V}(\n_R\cdot \nabla \fnc{V})  \,ds.
\end{aligned}
\end{align}
Adding Equations \eqref{eq02d} and \eqref{eq02e} we obtain
\begin{align*}
\begin{aligned}
	\dfrac12\dfrac{\partial}{\partial t}\|\fnc{W}\|^2_{\Omega_L}+\epsilon \|\nabla \fnc{W}\|^2_{\Omega_L}&+\dfrac12\dfrac{\partial}{\partial t}\|\fnc{V}\|^2_{\Omega_R}+\kappa\|\nabla \fnc{V}\|^2_{\Omega_R}
	=\int_{\partial \Omega_L\backslash\Sigma} \fnc{W}g\,ds+\dfrac{\kappa}{\varphi}\int_{\partial\Omega_R\backslash\Sigma} \fnc{V}h \,ds\\
	&-\int_{\partial \Omega_L\backslash\Sigma}\big(\dfrac12\a\cdot \n_L+\zeta\big)\,\fnc{W}^2 \,ds-\dfrac{\kappa}{\varphi}\int_{\partial\Omega_R\backslash\Sigma}\fnc{V}^2 \,ds\\
	&-\dfrac12\int_{\Sigma}(\a\cdot\n_L)\fnc{W}^2 \,ds+\epsilon\int_{\Sigma}\fnc{W}(\n_L\cdot\nabla \fnc{W})\,ds+\kappa\int_{\Sigma} \fnc{V}(\n_R\cdot \nabla \fnc{V})  \,ds,
\end{aligned}
\end{align*}
Then by substituting the interface conditions \eqref{eq03c},
we obtain the energy estimate
\begin{align*}
	\begin{aligned}
		\dfrac12\dfrac{\partial}{\partial t}\|\fnc{W}\|^2_{\Omega_L}&+\epsilon \|\nabla \fnc{W}\|^2_{\Omega_L}+\dfrac12\dfrac{\partial}{\partial t}\|\fnc{V}\|^2_{\Omega_R}+\kappa\|\nabla \fnc{V}\|^2_{\Omega_R}\,dx=\int_{\partial \Omega_L\backslash\Sigma} \fnc{W}g\,ds+\dfrac{\kappa}{\varphi}\int_{\partial\Omega_R\backslash\Sigma} \fnc{V}h \,ds\\
		&\quad-\int_{\partial \Omega_L\backslash\Sigma}\big(\dfrac12\a\cdot \n_L+\zeta\big)\,\fnc{W}^2 \,ds-\dfrac{\kappa}{\varphi}\int_{\partial\Omega_R\backslash\Sigma}\fnc{V}^2 \,ds-\dfrac12\int_{\Sigma}(\a\cdot\n_L)\fnc{W}^2 \,ds.
	\end{aligned}
\end{align*}
Now let $\alpha=\big(\dfrac12\a\cdot \n_L+\zeta\big)> 0$ and $\beta=\dfrac{\kappa}{\varphi}> 0$ and assume that at interface we have $\a\cdot\n_L\ge0$. Then
\begin{align*}
	\begin{aligned}
		\dfrac12\dfrac{\partial}{\partial t}\|\fnc{W}\|^2_{\Omega_L}&+\epsilon \|\nabla \fnc{W}\|^2_{\Omega_L}+\dfrac12\dfrac{\partial}{\partial t}\|\fnc{V}\|^2_{\Omega_R}+\kappa\|\nabla \fnc{V}\|^2_{\Omega_R}\,dx\le
		\dfrac{1}{4\alpha}\int_{\partial \Omega_L\backslash\Sigma}g_L^2\,ds+\dfrac{\beta}{4}\int_{\partial \Omega_R\backslash\Sigma}g_R^2\,ds.
	\end{aligned}
\end{align*}
Therefore
\begin{align*}
	\begin{aligned}
		\dfrac{\partial}{\partial t}\|\fnc{W}\|^2_{\Omega_L}+2\epsilon \|\nabla \fnc{W}\|^2_{\Omega_L}+\dfrac{\partial}{\partial t}\|\fnc{V}\|^2_{\Omega_R}+2\kappa\|\nabla \fnc{V}\|^2_{\Omega_R}
		\le \dfrac{1}{2\alpha}\|g\|^2_{\partial \Omega_L\backslash\Sigma}+\dfrac{\beta}{2}\|h\|^2_{\partial \Omega_R\backslash\Sigma}.
	\end{aligned}
\end{align*} 
Notice that this approach works for any 
$\zeta> -\dfrac12\a\,\cdot\, \n_L$, and $ \varphi> 0$.
For simplicity in calculations, we assume 
$\zeta=\dfrac12\big(|\a\,\cdot\, \n_L|-\a\,\cdot\, \n_L\big) $ and $\varphi=\kappa$. 
Additionally, to ensure a meaningful physical interpretation of the interface condition $\a\cdot\n_{L}\ge0$, we restrict our analysis to the case where $\a\,\cdot\,\n_L=0$. Note that if $\a\cdot\n_L\neq 0$, the energy argument contains an additional interface term and the analysis must be adapted accordingly, using an interface treatment consistent with characteristic (inflow/outflow) directions.

\section{Proof the theorem \eqref{thm:BE}}\label{apdB}

Set $\v^*=\v^{k}$ in \eqref{eq12a} and $\w^*=\w^{k+1}$ in \eqref{eq12b}. 
Assume the advection-diffusion equation  without source terms. 
Multiply the it by ${\w^{k+1}}\Tr \Pnorm{}$,  we obtain
{\setlength{\abovedisplayskip}{7pt}
\setlength{\belowdisplayskip}{7pt} 
\begin{align*}
	\begin{aligned}
		{\w^{k+1}}\Tr\Pnorm{}\MJack{\mat L} \delta_{\delta t} \w^{k+1}&+\dfrac12\sum_{l,m=1}^3 a_m{\w^{k+1}}\Tr\Pnorm{}\bigg\{\D{\xi_l}\MJdxildxmk{l}{m}{\mat L}+\MJdxildxmk{l}{m}{\mat L}\D{\xi_l}\bigg\}\w^{k+1}+\gamma_1^{\mat L}{\w^{k+1}}\Tr \Pnorm{}{\mat S}^{{\mat L},1}_{\Sigma}+\gamma_2^{\mat L}{\w^{k+1}}\Tr \Pnorm{}{\mat S}^{{\mat L},2}_{\Sigma}
		\\=&\epsilon \sum_{l,a=1}^3{\w^{k+1}}\Tr\Pnorm{}\D{\xi_l}\Cmat{l,a}{\mat L}\D{\xi_a}\w^{k+1}+\sum_{l=1}^3{\w^{k+1}}\Tr \Pnorm{}{\mat S}_{\alpha_l}^{\mat L}+\sum_{l=2}^3{\w^{k+1}}\Tr \Pnorm{}{\mat S}_{\beta_l}^{\mat L}.
	\end{aligned}
\end{align*}}
Notice that
{\setlength{\abovedisplayskip}{7pt}
\setlength{\belowdisplayskip}{7pt} 
\begin{align*}
	\begin{aligned}
		{\w^{k+1}}^T\Pnorm{}\MJack{\mat L}\delta_{\delta t} \w^{k+1}
		\ge \dfrac1{2\delta t}\|\w^{k+1}\|^2_{\mat L}-\dfrac1{2\delta t}\|\w^k\|^2_{\mat L}.
	\end{aligned}
\end{align*}}
Therefore
{\setlength{\abovedisplayskip}{7pt}
\setlength{\belowdisplayskip}{7pt} 
\begin{align}\label{eq15a}
	\begin{aligned}
		\dfrac1{2\delta t}\|\w^{k+1}\|^2_{\mat L}&+\dfrac12\sum_{l,m=1}^3 a_m{\w^{k+1}}\Tr\Pnorm{}\bigg\{\D{\xi_l}\MJdxildxmk{l}{m}{\mat L}+\MJdxildxmk{l}{m}{\mat L}\D{\xi_l}\bigg\}\w^{k+1}+\gamma_1^{\mat L}{\w^{k+1}}\Tr \Pnorm{}{\mat S}^{{\mat L},1}_{\Sigma}+\gamma_2^{\mat L}{\w^{k+1}}\Tr \Pnorm{}{\mat S}^{{\mat L},2}_{\Sigma}
		\\\le&\dfrac1{2\delta t}\|\w^{k}\|^2_{\mat L}+\epsilon\sum_{l,a=1}^3{\w^{k+1}}\Tr\Pnorm{}\D{\xi_l}\Cmat{l,a}{\mat L}\D{\xi_a}\w^{k+1}+\sum_{l=1}^3{\w^{k+1}}\Tr \Pnorm{}{\mat S}_{\alpha_l}^{\mat L}+\sum_{l=2}^3{\w^{k+1}}\Tr \Pnorm{}{\mat S}_{\beta_l}^{\mat L}.
	\end{aligned}
\end{align}}
Also 	
{\setlength{\abovedisplayskip}{7pt}
\setlength{\belowdisplayskip}{7pt} 
\begin{align}\label{eq15}
	\begin{aligned}
		\dfrac12\sum_{l,m=1}^3a_m{\w^{k+1}}\Tr\Pnorm{}\bigg\{\D{\xi_l}\MJdxildxmk{l}{m}{\mat L}&+\MJdxildxmk{l}{m}{\mat L}\D{\xi_l}\bigg\}\w^{k+1}
		=\dfrac12\sum_{l,m=1}^3a_m{\w^{k+1}}\Tr\big(\Q{\xi_l}+{\Q{\xi_l}}^T\big)\MJdxildxmk{l}{m}{\mat L}\w^{k+1}
\\&=\dfrac12\sum_{l,m=1}^3a_m{\w^{k+1}}\Tr\big(\R{\beta_l}\Tr\Pnorm{\perp\xi_l}\R{\beta_l}-\R{\alpha_l}\Tr\Pnorm{\perp\xi_l}\R{\alpha_l}\big)\MJdxildxmk{l}{m}{\mat L}\w^{k+1}.
	\end{aligned}
\end{align}}
Also we can simplify
{\setlength{\abovedisplayskip}{7pt}
\setlength{\belowdisplayskip}{7pt} 
\begin{align}\label{eq16}
	\begin{aligned}
		\sum_{l,a=1}^3&{\w^{k+1}}\Tr\Pnorm{}\D{\xi_l}\Cmat{l,a}{\mat L}\D{\xi_a}\w^{k+1}=
		\sum_{l,a=1}^3{\w^{k+1}}\Tr\Q{\xi_l}\Cmat{l,a}{\mat L}\D{\xi_a}\w^{k+1}
		\\&=\sum_{l,a=1}^3{\w^{k+1}}\Tr\big(\E{\xi_l}-\Q{\xi_l}\Tr\big)\Cmat{l,a}{\mat L}\D{\xi_a}\w^{k+1}
		\\&=\sum_{l,a=1}^3{\w^{k+1}}\Tr\E{\xi_l}\Cmat{l,a}{\mat L}\D{\xi_a}\w^{k+1}-\sum_{l,a=1}^3\big(\D{\xi_l} \w^{k+1}\big)\Tr\Pnorm{}\Cmat{l,a}{\mat L}\D{\xi_a}\w^{k+1}
		\\&=\sum_{l,a=1}^3{\w^{k+1}}\Tr\big(\R{\beta_l}\Tr\Pnorm{\perp\xi_l}\R{\beta_l}-\R{\alpha_l}\Tr\Pnorm{\perp\xi_l}\R{\alpha_l}\big)\Cmat{l,a}{\mat L}\D{\xi_a}\w^{k+1}-\sum_{l,a=1}^3\big(\D{\xi_l} \w^{k+1}\big)\Tr\Pnorm{}\Cmat{l,a}{\mat L}\D{\xi_a}\w^{k+1}.
	\end{aligned}
\end{align}}
Substitute \eqref{eq15} and \eqref{eq16} in \eqref{eq15a} we obtain
{\setlength{\abovedisplayskip}{7pt}
	\setlength{\belowdisplayskip}{7pt} 
\begin{align}\label{eq18}
\begin{aligned}
\dfrac1{2\delta t}\|\w^{k+1}\|^2_{\mat L}
&+\epsilon\|\D{}\w^{k+1}\|^2_{\mat L}+\gamma_1^{\mat L}{\w^{k+1}}\Tr \Pnorm{}{\mat S}^{{\mat L},1}_{\Sigma}+\gamma_2^{\mat L}{\w^{k+1}}\Tr \Pnorm{}{\mat S}^{{\mat L},2}_{\Sigma}
\\\le&  \dfrac1{2\delta t}\|\w^{k}\|^2_{\mat L}+\sum_{l=1}^3{\w^{k+1}}\Tr \Pnorm{}{\mat S}_{\alpha_l}^{\mat L}+\sum_{l=2}^3{\w^{k+1}}\Tr \Pnorm{}{\mat S}_{\beta_l}^{\mat L}
\\&-\dfrac12\sum_{l,m=1}^3a_m{\w^{k+1}}\Tr\R{\beta_l}\Tr\Pnorm{\perp\xi_l}\R{\beta_l}\MJdxildxmk{l}{m}{\mat L}\w^{k+1}
+\dfrac1{2}\sum_{l,m=1}^3a_m{\w^{k+1}}\Tr\R{\alpha_l}\Tr\Pnorm{\perp\xi_l}\R{\alpha_l}\MJdxildxmk{l}{m}{\mat L}\w^{k+1}
\\&+\epsilon\sum_{l,a=1}^3{\w^{k+1}}\Tr\R{\beta_l}\Tr\Pnorm{\perp\xi_l}\R{\beta_l}\Cmat{l,a}{\mat L}\D{\xi_a}\w^{k+1}-\epsilon\sum_{l,a=1}^3{\w^{k+1}}\Tr\R{\alpha_l}\Tr\Pnorm{\perp\xi_l}\R{\alpha_l}\Cmat{l,a}{\mat L}\D{\xi_a}\w^{k+1}.
\end{aligned}
\end{align}}
where \(\|\D{}\w^{k+1}\|^2_{\mat L}:=\sum_{l,a=1}^3\big(\D{\xi_l} \w^{k+1}\big)\Tr\Pnorm{}\Cmat{l,a}{\mat L}\D{\xi_a}\w^{k+1}.\)
Next we consider terms at left boundaries in the inequality \eqref{eq18}. First consider terms with $\R{\alpha_l}$ in front and associated  SAT terms multiplied with ${\w^{k+1}}\Tr\Pnorm{}$, a direct calculation yields
{\setlength{\abovedisplayskip}{7pt}
	\setlength{\belowdisplayskip}{7pt} 
\begin{align}\label{eq18b}
	\begin{aligned}		\dfrac12\sum_{l,m=1}^3a_m{\w^{k+1}}\Tr&\R{\alpha_l}\Tr\Pnorm{\perp\xi_l}\R{\alpha_l}\MJdxildxmk{l}{m}{\mat L}\w^{k+1}		\\&+\epsilon\sum_{l,a=1}^3{\w^{k+1}}\Tr\R{\alpha_l}\Tr\Pnorm{\perp\xi_l}\R{\alpha_l}\Cmat{l,a}{\mat L}\D{\xi_a}\w^{k+1}
        +\sum_{l=1}^3{\w^{k+1}}\Tr \Pnorm{}{\mat S}_{\alpha_l}^{\mat L}
		\le \sum_{l=1}^3\dfrac1{2\zeta_{\alpha_l}}\left\|\g_{\alpha_l}^{k+1}\right\|^2_{\MSJack{\alpha_l}{\mat L}\Pnorm{\perp\xi_l}},
	\end{aligned}
\end{align}}
where \(\zeta_{\alpha_l}=\max_{j\in {\cal J}_{\alpha_l}}\{{|(\Normal{\alpha_l})_{jj}|}\}\) for \( l=1,2,3.\)\\
Then consider the  terms with $\R{\beta_2}$ and $\R{\beta_3}$ in front and associated SAT terms in \eqref{eq18},
{\setlength{\abovedisplayskip}{7pt}
\setlength{\belowdisplayskip}{7pt} 
	\begin{align}\label{eq18c}
	\begin{aligned}
		-\dfrac12\sum_{l=2}^3\sum_{m=2}^3 a_m{\w^{k+1}}\Tr&\R{\beta_l}\Tr\Pnorm{\perp\xi_l}\R{\beta_l}\MJdxildxmk{l}{m}{\mat L}\w^{k+1}
\\&+\epsilon\sum_{l=2}^3\sum_{a=1}^3{\w^{k+1}}\Tr\R{\beta_l}\Tr\Pnorm{\perp\xi_l}\R{\beta_l}\Cmat{l,a}{\mat L}\D{\xi_a}\w^{k+1}
+\sum_{l=2}^3{\w^{k+1}}\Tr \Pnorm{}{\mat S}_{\beta_l}^{\mat L}
		\le \sum_{l=2}^3\dfrac1{2\zeta_{\beta_l}}\left\|\g_{\beta_l}^{k+1}\right\|^2_{\MSJack{\beta_l}{\mat L}\Pnorm{\perp\xi_l}},
	\end{aligned}
\end{align}}
where \(\zeta_{\beta_l}=\max_{j\in {\cal J}_{\beta_l}}\{{|(\Normal{\beta_l})_{jj}|}\}\) for \(l=2,3.\) Then apply \eqref{eq18b} and \eqref{eq18c}  in \eqref{eq18} and use the assumption at the interface, $\a\cdot \n_{\mat L}=0$, to obtain
{\setlength{\abovedisplayskip}{7pt}
\setlength{\belowdisplayskip}{7pt} 
\begin{align}\label{eq19}
\begin{aligned}
\dfrac1{2\delta t}&\|\w^{k+1}\|^2_{\mat L}+\epsilon\|\D{}\w^{k+1}\|^2_{\mat L }+\gamma_1{\w^{k+1}}\Tr\Pnorm{}\,{\mat S}_{\Sigma}^{L,1}+\gamma_2^{\mat L}{\w^{k+1}}\Tr\Pnorm{}\,{\mat S}_{\Sigma}^{L,2}
\\\le&  \dfrac1{2\delta t}\|\w^{k}\|^2_{\mat L}+\sum_{l=1}^3\dfrac1{2\zeta_{\alpha_l}}\left\|\g_{\alpha_l}^{k+1}\right\|^2_{\MSJack{\alpha_l}{\mat L}\Pnorm{\perp\xi_l}}+\sum_{l=2}^3\dfrac1{2\zeta_{\beta_l}}\left\|\g_{\beta_l}^{k+1}\right\|^2_{\MSJack{\beta_l}{\mat L}\Pnorm{\perp\xi_l}}
+\epsilon\sum_{a=1}^3{\w^{k+1}}\Tr\R{\beta_1}\Tr\Pnorm{\perp\xi_1}\R{\beta_1}\Cmat{1,a}{\mat L}\D{\xi_a}\w^{k+1}.
\end{aligned}
\end{align}
}
Doing the same for the heat equation on right subdomain,
{\setlength{\abovedisplayskip}{7pt}
\setlength{\belowdisplayskip}{7pt} 
\begin{align*}
	\begin{aligned}
		\dfrac1{2\delta t}\|\v^{k+1}\|^2_{\mat R}&+\kappa\|\D{}\v^{k+1}\|^2_{\mat R}
		+\gamma_1{\v^{k+1}}\Tr\Pnorm{}\,{\mat S}_{\Sigma}^{R,1}
		+ \gamma_2^{\mat R}{\v^{k+1}}\Tr\Pnorm{}\,{\mat S}_{\Sigma}^{R,2}		\\&+\sum_{a=1}^3{\v^{k+1}}^T\R{\alpha_1}\Tr\Pnorm{\perp\xi_1}\bigg(\epsilon\R{\beta_1}\Cmat{1,a}{\mat L}\D{\xi_a}\w^{k+1}-\kappa\R{\alpha_1}\Cmat{1,a}{\mat R}\D{\xi_a}\v^{k+1}\bigg)	
		\\\le& \dfrac1{2\delta t}\|\v^{k}\|^2_{\mat R}
		+\sum_{l=1}^3\dfrac1{4}\left\|\h_{\beta_l}^{k+1}\right\|^2_{\MSJack{\beta_l}{\mat R}\Pnorm{\perp\xi_l}}	+\sum_{l=2}^3\dfrac1{4}\left\|\h_{\alpha_l}^{k+1}\right\|^2_{\MSJack{\alpha_l}{\mat R}\Pnorm{\perp\xi_l}}	
		-\kappa \sum_{a=1}^3{\v^{k+1}}^T\R{\alpha_1}\Tr\Pnorm{\perp\xi_1}\R{\alpha_1}\Cmat{1,a}{\mat R}\D{\xi_a}\v^{k+1}.
	\end{aligned}
\end{align*}}
Then
{\setlength{\abovedisplayskip}{7pt}
\setlength{\belowdisplayskip}{7pt} 
\begin{align}\label{eq20}
	\begin{aligned}
		\dfrac1{2\delta t}\|\v^{k+1}\|^2_{\mat R}&+\kappa\|\D{}\v^{k+1}\|^2_{\mat R}
		+\gamma_1{\v^{k+1}}\Tr\Pnorm{}\,{\mat S}_{\Sigma}^{R,1}
		+\gamma_2^{\mat R}{\v^{k+1}}\Tr\Pnorm{}\,{\mat S}_{\Sigma}^{R,2}				
		\\\le&\dfrac1{2\delta t}\|\v^{k}\|^2_{\mat R}
		+\dfrac1{4}\sum_{l=1}^3\left\|\h_{\beta_l}^{k+1}\right\|^2_{\MSJack{\beta_l}{\mat R}\Pnorm{\perp\xi_l}}	+\dfrac1{4}\sum_{l=2}^3\|
		\h_{\alpha_l}^{k+1}\|_{\MSJack{\alpha_l}{\mat R}\Pnorm{\perp\xi_l}}	
		-\epsilon\sum_{a=1}^3{\v^{k+1}}^T\R{\alpha_1}\Tr\Pnorm{\perp\xi_1}\R{\beta_1}\Cmat{1,a}{\mat L}\D{\xi_a}\w^{k+1}.
	\end{aligned}
\end{align}}
Add \eqref{eq19} and \eqref{eq20}, use notation \eqref{data} and simplify  to obtain
{\setlength{\abovedisplayskip}{7pt}
	\setlength{\belowdisplayskip}{7pt} 
	\begin{align}\label{eq21}
		\begin{aligned}
			\|\w^{k+1}\|^2_{\mat L}&+ \|\v^{k+1}\|^2_{\mat R}+2\delta t\,\epsilon\|\D{}\w^{k+1}\|^2_{\mat L}
			+2\delta t\,\kappa\|\D{}\v^{k+1}\|^2_{\mat R}
			\\&\quad+2\delta t\,\gamma_1{\w^{k+1}}\Tr\Pnorm{}\,{\mat S}_{\Sigma}^{\mat L,1}+2\delta t\,\gamma_1{\v^{k+1}}\Tr\Pnorm{}\,{\mat S}_{\Sigma}^{\mat R,1}+2\delta t\gamma_2^{\mat L}{\w^{k+1}}\Tr\Pnorm{}\,{\mat S}_{\Sigma}^{\mat L,2}+2\delta t\gamma_2^{\mat R}{\v^{k+1}}\Tr\Pnorm{}\,{\mat S}_{\Sigma}^{\mat R,2}
			\\&\le \|\w^{k}\|^2_{\mat L} +\|\v^{k}\|^2_{\mat R}+\calG^{k+1}+2\delta t\,\epsilon\left(\R{\beta_1}\w^{k+1}-\R{\alpha_1}\v^{k+1}\right)\Tr\Pnorm{\perp\xi_1}\left(\sum_{a=1}^3\R{\beta_1}\Cmat{1,a}{\mat L}\D{\xi_a}\w^{k+1}\right),
		\end{aligned}
\end{align}}
Now consider the last line of the inequality \eqref{eq21}. By using Cauchy Schwarz, Young inequalities and Lemma \eqref{lemTrace} for left subdomain we obtain
{\setlength{\abovedisplayskip}{7pt}
	\setlength{\belowdisplayskip}{7pt} 
	\begin{align*}
		\begin{aligned}
			&2\left(\R{\beta_1}\w^{k+1}-\R{\alpha_1}\v^{k+1}\right)\Tr\Pnorm{\perp\xi_1}\bigg(\sum_{a=1}^3\R{\beta_1}\Cmat{1,a}{\mat L}\D{\xi_a}\w^{k+1}\bigg)
			\\&\quad\le \dfrac {1}{\rho_{\mat L}}\bigg(\R{\beta_1}\w^{k+1}-\R{\alpha_1}\v^{k+1}\bigg)\Tr\MSJack{\Sigma}{}\Pnorm{\perp\xi_1}\left(\R{\beta_1}\w^{k+1}-\R{\alpha_1}\v^{k+1}\right)
			\\&\quad\qquad\qquad+
			\rho_{\mat L}\bigg(\sum_{a=1}^3 \R{\beta_1}\Cmat{1,a}{\mat L}\D{\xi_a}\w^{k+1}\bigg)\Tr{\MSJack{\Sigma}{}}^{-1}\Pnorm{\perp\xi_1} \bigg(\sum_{a=1}^3\R{\beta_1}\Cmat{1,a}{\mat L}\D{\xi_a}\w^{k+1}\bigg)
			\\&\quad\le\dfrac{1}{\rho_{\mat L}}\big\|\R{\alpha_1}\v^{k+1}-\R{\beta_1}\w^{k+1}\big\|^2_{\Sigma}
			+\|\D{}\w^{k+1}\|^2_{\mat L}.
		\end{aligned}
\end{align*}}
Using  this in  \eqref{eq21}, we obtain
{\setlength{\abovedisplayskip}{7pt}
	\setlength{\belowdisplayskip}{7pt} 
	\begin{align}\label{eq24aa}
		\begin{aligned}
			\|\w^{k+1}\|^2_{\mat L}&+ \|\v^{k+1}\|^2_{\mat R}
			+\delta t\,\epsilon\|\D{}\w^{k+1}\|^2_{\mat L}
			+2\delta t\,\kappa\|\D{}\v^{k+1}\|^2_{\mat R}
			\\&+2\delta t\,\gamma_1{\w^{k+1}}\Tr\Pnorm{}\,{\mat S}_{\Sigma}^{\mat L,1}
			+2\delta t\,\gamma_1{\v^{k+1}}\Tr\Pnorm{}\,{\mat S}_{\Sigma}^{\mat R,1}
			+2\delta t\,\gamma_2^{\mat L}{\w^{k+1}}\Tr\Pnorm{}\,{\mat S}_{\Sigma}^{\mat L,2}
            +2\delta t\,\gamma_2^{\mat R}{\v^{k+1}}\Tr\Pnorm{}\,{\mat S}_{\Sigma}^{\mat R,2}
			\\&\le \|\w^{k}\|^2_{\mat L} +\|\v^{k}\|^2_{\mat R}+\calG^{k+1}
			+\delta t\,\,\dfrac{\epsilon}{\rho_{\mat L}}\big\|\R{\alpha_1}\v^{k+1}-\R{\beta_1}\w^{k+1}\big\|^2_{\Sigma}.
		\end{aligned}
\end{align}}
Consider SAT terms ${\mat S}_{\Sigma}^{\mat L,1}$ and ${\mat S}_{\Sigma}^{\mat R,1}$ together in \eqref{eq24aa}, we have 
{\setlength{\abovedisplayskip}{7pt}
	\setlength{\belowdisplayskip}{7pt} 
	\begin{align}\label{eq1a}
		\begin{aligned}
			I_1=2 \delta t\gamma_1{\w^{k+1}}\Tr\Pnorm{}\,{\mat S}_{\Sigma}^{\mat L,1}+2 \delta \gamma_1{\v^{k+1}}\Tr\Pnorm{}\,{\mat S}_{\Sigma}^{\mat R,1}
			=2\delta t\gamma_1 \calE_1^{k+1}
			+2\delta t\gamma_1(\R{\beta_1}\w^{k+1})\Tr\MSJack{\Sigma}{}\Pnorm{\perp\xi_1}\big(\R{\alpha_1}\v^{k+1}-\R{\alpha_1}\v^k\big),
		\end{aligned}
\end{align}}
using notations \eqref{E1} and \eqref{gamma}. Next consider ${\mat S}_{\Sigma}^{\mat L,2}$ and ${\mat S}_{\Sigma}^{\mat R,2}$ together in \eqref{eq24aa},  to obtain
{\setlength{\abovedisplayskip}{7pt}
	\setlength{\belowdisplayskip}{7pt} 
	\begin{align*}
		\begin{aligned}
			I_2&=2\delta t\gamma_2^{\mat L}{\w^{k+1}}\Tr\Pnorm{}\,{\mat S}_{\Sigma}^{\mat L,1}+2\delta t \gamma_2^{\mat R}{\v^{k+1}}\Tr\Pnorm{}\,{\mat S}_{\Sigma}^{\mat R,1}
			\\&\ge  2\delta t\gamma_2 \calE_2^{k+1}
			-2\delta t\hat \gamma_2 \bigg(\epsilon\sum_{a=1}^3\R{\beta_1}\Cmat{1,a}{\mat L}\D{\xi_a}\w^{k+1}\bigg)\Tr\MSJack{\Sigma}{}^{-1}{}\Pnorm{\perp\xi_1}\bigg(\kappa\sum_{a=1}^3\R{\alpha_1}\Cmat{1,a}{\mat R}\D{\xi_a}\v^{k+1}\bigg)
			\\&\quad+2\delta t\gamma_2\bigg(\epsilon\sum_{a=1}^3\R{\beta_1}\Cmat{1,a}{\mat L}\D{\xi_a}\w^{k+1}\bigg)\Tr\MSJack{\Sigma}{}^{-1}{}\Pnorm{\perp\xi_1}
			\bigg(\kappa\sum_{a=1}^3\R{\alpha_1}\Cmat{1,a}{\mat R}\D{\xi_a}\v^{k+1}-\kappa\sum_{a=1}^3\R{\alpha_1}\Cmat{1,a}{\mat R}\D{\xi_a}\v^{k}\bigg),
		\end{aligned}
\end{align*}}
using notations \eqref{E2} and \eqref{gamma}. This gives:
{\setlength{\abovedisplayskip}{7pt}
	\setlength{\belowdisplayskip}{7pt} 
	\begin{align}\label{eq1b}
		\begin{aligned}
			- I_2 &\le-2 \delta t\, \gamma_2 \calE_2^{k+1}
			+\delta t\, \hat  \gamma_2\bigg\|\epsilon\sum_{a=1}^3\R{\beta_1}\Cmat{1,a}{\mat L}\D{\xi_a}\w^{k+1}\bigg\|_{\bar\Sigma}^2 +\delta t\, \hat \gamma_2\bigg\|\kappa\sum_{a=1}^3\R{\alpha_1}\Cmat{1,a}{\mat R}\D{\xi_a}\v^{k+1}\bigg\|_{\bar\Sigma}^2
			\\&\quad-2 \delta t \gamma_2\bigg(\epsilon\sum_{a=1}^3\R{\beta_1}\Cmat{1,a}{\mat L}\D{\xi_a}\w^{k+1}\bigg)\Tr\MSJack{\Sigma}{}^{-1}{}\Pnorm{\perp\xi_1}
			\bigg(\kappa\sum_{a=1}^3\R{\alpha_1}\Cmat{1,a}{\mat R}\D{\xi_a}\v^{k+1}-\kappa\sum_{a=1}^3\R{\alpha_1}\Cmat{1,a}{\mat R}\D{\xi_a}\v^{k}\bigg)
			\\&\le-2 \delta t\gamma_2 \calE_2^{k+1}
			+\dfrac{\kappa^2\delta t\,\hat \gamma_2}{\rho_{\mat L}}\big\|\D{}\w^{k+1}\big\|_{\mat L}^2 +\dfrac{\kappa^2\delta t\,\hat \gamma_2}{\rho_{\mat R}}\big\|\D{}\v^{k+1}\big\|_{\mat R}^2
			\\&\quad-2 \delta t\,\gamma_2\bigg(\epsilon\sum_{a=1}^3\R{\beta_1}\Cmat{1,a}{\mat L}\D{\xi_a}\w^{k+1}\bigg)\Tr\MSJack{\Sigma}{}^{-1}{}\Pnorm{\perp\xi_1}
			\bigg(\kappa\sum_{a=1}^3\R{\alpha_1}\Cmat{1,a}{\mat R}\D{\xi_a}\v^{k+1}-\kappa\sum_{a=1}^3\R{\alpha_1}\Cmat{1,a}{\mat R}\D{\xi_a}\v^{k}\bigg).
		\end{aligned}
\end{align}}
Substitute \eqref{eq1a}  and \eqref{eq1b} in \eqref{eq24aa} to obtain
{\setlength{\abovedisplayskip}{7pt}
	\setlength{\belowdisplayskip}{7pt} 
	\begin{align*}
		\begin{aligned}
			\|\w^{k+1}\|^2_{\mat L}&+ \|\v^{k+1}\|^2_{\mat R}+2\delta t\,\gamma_1\calE_1^{k+1}+2\delta t\,\gamma_2\calE_2^{k+1}
			+\delta t\,\epsilon (1-\, \dfrac{\epsilon\hat\gamma_2}{\rho_{\mat L}})\|\D{}\w^{k+1}\|^2_{\mat L}
			+\delta t\, \kappa (1-\, \dfrac{\kappa\hat\gamma_2}{\rho_{\mat R}})\|\D{}\v^{k+1}\|^2_{\mat R}
			\\&\le \|\w^{k}\|^2_{\mat L} +\|\v^{k}\|^2_{\mat R}+\calG^{k+1}
			+\delta t\,\dfrac{\epsilon}{\rho_{\mat L}}\left\|\R{\alpha_1}\v^{k+1}-\R{\beta_1}\w^{k+1}\right\|^2_{\Sigma}
			-2\delta t\,\gamma_1\big(\R{\beta_1}\w^{k+1}\big)\Tr\MSJack{\Sigma}{}\Pnorm{\perp\xi_1}\big(\R{\alpha_1}\v^{k+1}-\R{\alpha_1}\v^k\big)	
			\\&\quad-2\delta t\,\gamma_2\bigg(\epsilon\sum_{a=1}^3\R{\beta_1}\Cmat{1,a}{\mat L}\D{\xi_a}\w^{k+1}\bigg)\Tr\MSJack{\Sigma}{}^{-1}\Pnorm{\perp\xi_1}
			\bigg(\kappa\sum_{a=1}^3\R{\alpha_1}\Cmat{1,a}{\mat R}\D{\xi_a}\v^{k+1}-\kappa\sum_{a=1}^3\R{\alpha_1}\Cmat{1,a}{\mat R}\D{\xi_a}\v^{k}\bigg).
		\end{aligned}
\end{align*}}
Assume assumptions \eqref{a1}-\eqref{a3} hold, then 
{\setlength{\abovedisplayskip}{7pt}
	\setlength{\belowdisplayskip}{7pt} 
	\begin{align}\label{eq27a}
		\begin{aligned}
			&\|\w^{k+1}\|^2_{\mat L}+\|\v^{k+1}\|^2_{\mat R}+\delta t\,\gamma_1\calE_1^{k+1}+2\delta t\,\gamma_2\calE_2^{k+1}
			\le \|\w^{k}\|^2_{\mat L} +\|\v^{k}\|^2_{\mat R}+\calG^{k+1}
			-2\delta t\,\gamma_1\big(\R{\beta_1}\w^{k+1}\big)\Tr\MSJack{\Sigma}{}\Pnorm{\perp\xi_1}\big(\R{\alpha_1}\v^{k+1}-\R{\alpha_1}\v^k\big)	
			\\&\quad-2\delta t\,\gamma_2\bigg(\epsilon\sum_{a=1}^3\R{\beta_1}\Cmat{1,a}{\mat L}\D{\xi_a}\w^{k+1}\bigg)\Tr\MSJack{\Sigma}{}^{-1}\Pnorm{\perp\xi_1}
            \bigg(\kappa\sum_{a=1}^3\R{\alpha_1}\Cmat{1,a}{\mat R}\D{\xi_a}\v^{k+1}-\kappa\sum_{a=1}^3\R{\alpha_1}\Cmat{1,a}{\mat R}\D{\xi_a}\v^{k}\bigg).
		\end{aligned}
\end{align}}
Then consider the last term in  the line 2 of above inequality, we have 
{\setlength{\abovedisplayskip}{7pt}
\setlength{\belowdisplayskip}{7pt} 
	\begin{align*}
		\begin{aligned}
			-\big(\R{\beta_1}\w^{k+1}\big)&\Tr\MSJack{\Sigma}{}\Pnorm{\perp\xi_1}\big(\R{\alpha_1}\v^{k+1}-\R{\alpha_1}\v^k\big)	
			\\&=\big(\R{\beta_1}\w^{k+1}-\R{\alpha_1}\v^{k+1}+\R{\alpha_1}\v^{k+1}\big)\Tr\MSJack{\Sigma}{}\Pnorm{\perp\xi_1}\big(\R{\alpha_1}\v^{k}-\R{\alpha_1}\v^{k+1}\big)		
			\\&\quad+\big(\R{\beta_1}\w^{k+1}-\R{\alpha_1}\v^{k+1}\big)\Tr\MSJack{\Sigma}{}\Pnorm{\perp\xi_1}\big(\R{\alpha_1}\v^{k}-\R{\alpha_1}\v^{k+1}\big)
			\\&\le\big(\R{\alpha_1}\v^{k+1}\big)\Tr\MSJack{\Sigma}{}\Pnorm{\perp\xi_1}\R{\alpha_1}\v^{k}-\big(\R{\alpha_1}\v^{k+1}\big)\Tr\MSJack{\Sigma}{}\Pnorm{\perp\xi_1}\R{\alpha_1}\v^{k+1}
			\\&\quad+\dfrac12\big(\R{\beta_1}\v^{k}-\R{\alpha_1}\v^{k+1}\big)\Tr\MSJack{\Sigma}{}\Pnorm{\perp\xi_1}\big(\R{\beta_1}\v^{k}-\R{\alpha_1}\v^{k+1}\big)
			\\&\quad+\dfrac12\big(\R{\beta_1}\w^{k+1}-\R{\alpha_1}\v^{k+1}\big)\Tr\MSJack{\Sigma}{}\Pnorm{\perp\xi_1}\big(\R{\beta_1}\w^{k+1}-\R{\alpha_1}\v^{k+1}\big)
			\\&\le
			\dfrac12\left\|\R{\alpha_1}\v^{k}\right\|^2_{\Sigma}-\dfrac12\left\|\R{\alpha_1}\v^{k+1}\right\|^2_{\Sigma}
			+\dfrac12\left\|\R{\beta_1}\w^{k+1}-\R{\alpha_1}\v^{k+1}\right\|^2_{\Sigma}.
		\end{aligned}
\end{align*}}
Doing the same for last two lines of the inequality \eqref{eq27a} and substitute them back in \eqref{eq27a}, it simplifies to 
{\setlength{\abovedisplayskip}{7pt}
\setlength{\belowdisplayskip}{7pt} 
\begin{align*}
	\begin{aligned}
		\|\w^{k+1}\|^2_{\mat L}+ \|\v^{k+1}\|^2_{\mat R}+\delta t\,\gamma_2\calE_2^{k+1}
		\le& \|\w^{k}\|^2_{\mat L} +\|\v^{k}\|^2_{\mat R}+\calG^{k+1}
		+\delta t\,\gamma_1\left(\left\|\R{\alpha_1}\v^{k}\right\|^2_{\Sigma}
		-\left\|\R{\alpha_1}\v^{k+1}\right\|^2_{\Sigma}\right)
		\\&+\kappa^2\delta t\,\gamma_2\left(
		\bigg\|\sum_{a=1}^3\R{\alpha_1}\Cmat{1,a}{\mat R}\D{\xi_a}\v^{k}\bigg\|^2_{\bar \Sigma}
		-\bigg\|\sum_{a=1}^3\R{\alpha_1}\Cmat{1,a}{\mat R}\D{\xi_a}\v^{k+1}\bigg\|^2_{\bar \Sigma}\right).
	\end{aligned}
\end{align*}}
Consequently
{\setlength{\abovedisplayskip}{7pt}
\setlength{\belowdisplayskip}{7pt} 
\begin{align*}
	\begin{aligned}
		\|\w^{k+1}\|^2_{\mat L}+ \|\v^{k+1}\|^2_{\mat R}
		+\delta t\,\gamma_2\sum_{j=1}^{k+1}\calE_2^{j}
		\le& \|\w^{0}\|^2_{\mat L} +\|\v^{0}\|^2_{\mat R}
		+\sum_{j=1}^{k+1}\calG^{j}
		+\delta t\,\gamma_1\sum_{j=0}^k\left(\left\|\R{\alpha_1}\v^{j}\right\|^2_{\Sigma}
		-\left\|\R{\alpha_1}\v^{j+1}\right\|^2_{\Sigma}\right)
		\\&+\kappa^2\delta t\,\gamma_2\sum_{j=0}^k\left(
		\bigg\|\sum_{a=1}^3\R{\alpha_1}\Cmat{1,a}{\mat R}\D{\xi_a}\v^{j}\bigg\|^2_{\bar \Sigma}
		-\bigg\|\sum_{a=1}^3\R{\alpha_1}\Cmat{1,a}{\mat R}\D{\xi_a}\v^{j+1}\bigg\|^2_{\bar \Sigma}\right).
	\end{aligned}
\end{align*}}
Observe there are telescoping sums, therefore the above simplifies to
{\setlength{\abovedisplayskip}{7pt}
\setlength{\belowdisplayskip}{7pt} 
\begin{align*}
	\begin{aligned}
		\|\w^{k+1}\|^2_{\mat L}+ \|\v^{k+1}\|^2_{\mat R}
		&+\delta t\gamma_2\sum_{j=1}^{k+1}\calE_2^{j}+\delta t\gamma_1\calF_1^{k+1}+\kappa^2\delta t\gamma_2\calF_2^{k+1}
		\\\le&\|\w^0\|^2_{\mat L} +\|\v^0\|^2_{\mat R}+\sum_{j=1}^{k+1}\calG^{j}
		+\delta t\,\gamma_1\left\|\R{\alpha_1}\v^{0}\right\|^2_{\Sigma}+\kappa^2 \delta t\,\gamma_2\sum_{a=1}^3
		\left\|\R{\alpha_1}\Cmat{1,a}{\mat R}\D{\xi_a}\v^{0}\right\|^2_{\bar \Sigma}
		=\sum_{j=0}^{k+1}\calG^{j},
	\end{aligned}
\end{align*}}
using notation \eqref{F}. Notice that $\w^{0}=\q_L$ and $\v^0=\q_R$. Therefore under the assumption \eqref{a1} and \eqref{a3}, we have the stability estimate.

\section{Proof the theorem \eqref{thm:BE-EXT2}}\label{apdC}
\begin{proof}
We begin by multiplying the advection-diffusion equation by ${\w^{k+1}}\Tr \Pnorm{}$ and the heat equation by  ${\v^{k+1}}\Tr \Pnorm{}$. Performing the same initial calculations as in Theorem \ref{thm:BE}, we obtain
{\setlength{\abovedisplayskip}{7pt}
\setlength{\belowdisplayskip}{7pt} 
\begin{align}\label{eq24a}
	\begin{aligned}
	\|\w^{k+1}\|^2_{\mat L}+& \|\v^{k+1}\|^2_{\mat R}+\delta t\epsilon\|\D{}\w^{k+1}\|^2_{\mat L}+2\delta t\kappa\|\D{}\v^{k+1}\|^2_{\mat R}
	+2\delta t\,\gamma_1{\w^{k+1}}\Tr\Pnorm{}\,{\mat S}_{\Sigma}^{L,1}
	+2\delta t\,\gamma_1{\v^{k+1}}\Tr\Pnorm{}\,{\mat S}_{\Sigma}^{R,1}
	\\&+2\delta t\gamma_2^{\mat L}\epsilon{\w^{k+1}}\Tr\Pnorm{}\,{\mat S}_{\Sigma}^{L,2}+2\delta t\gamma_2^{\mat R}\kappa{\v^{k+1}}\Tr\Pnorm{}\,{\mat S}_{\Sigma}^{R,2}
	\le \|\w^{k}\|^2_{\mat L} +\|\v^{k}\|^2_{\mat R}+\calG^{k+1}+\delta t\,\, \dfrac{\epsilon}{\rho_{\mat L}}\big\|\R{\alpha_1}\v^{k+1}-\R{\beta_1}\w^{k+1}\|^2_{\Sigma}
	\end{aligned}
\end{align}}
Let $\v^*=2\v^{k}-\v^{k-1}$ and  $\w^*=\w^{k+1} $ in SAT terms given by \eqref{eqS1} and \eqref{eqS2}. We have 
{\setlength{\abovedisplayskip}{7pt}
\setlength{\belowdisplayskip}{7pt} 
\begin{align}\label{eq25a}
		\begin{aligned}
			I_1&=\gamma_1{\w^{k+1}}\Tr\Pnorm{}\,{\mat S}_{\Sigma}^{L,1}
			+\gamma_1{\v^{k+1}}\Tr\Pnorm{}\,{\mat S}_{\Sigma}^{R,1}
			\\&=\gamma_1{\w^{k+1}}\Tr\R{\beta_1}\Tr\MSJack{\Sigma}{}\Pnorm{\perp\xi_1}\big(\R{\beta_1}\w^{k+1}-2\R{\alpha_1}\v^k+\R{\alpha_1}\v^{k-1}\big)+\gamma_1{\v^{k+1}}\Tr\R{\alpha_1}\Tr\MSJack{\Sigma}{}\Pnorm{\perp\xi_1}\big(\R{\alpha_1}\v^{k+1}-\R{\beta_1}\w^{k+1}\big)	
			\\&=\gamma_1\big(\R{\alpha_1}\v^{k+1}-\R{\beta_1}\w^{k+1}\big)\Tr\MSJack{\Sigma}{}\Pnorm{\perp\xi_1}\big(\R{\alpha_1}\v^{k+1}-\R{\beta_1}\w^{k+1}\big)
			\\&\quad\;+\gamma_1{\w^{k+1}}\Tr\R{\beta_1}\Tr\MSJack{\Sigma}{}\Pnorm{\perp\xi_1}\big(\R{\alpha_1}\v^{k+1}-\R{\alpha_1}\v^k\big)
			\;+\gamma_1{\w^{k+1}}\Tr\R{\beta_1}\Tr\MSJack{\Sigma}{}\Pnorm{\perp\xi_1}\big(\R{\alpha_1}\v^{k-1}-\R{\alpha_1}\v^k\big)
			\\&= \gamma_1\|\R{\alpha_1}\v^{k+1}-\R{\beta_1}\w^{k+1}\|^2_{\Sigma}
			\\&\quad\;+\gamma_1(\R{\beta_1}\w^{k+1})\Tr\MSJack{\Sigma}{}\Pnorm{\perp\xi_1}\big(\R{\alpha_1}\v^{k+1}-\R{\alpha_1}\v^k\big)
			\;+\gamma_1(\R{\beta_1}\w^{k+1})\Tr\MSJack{\Sigma}{}\Pnorm{\perp\xi_1}\big(\R{\alpha_1}\v^{k-1}-\R{\alpha_1}\v^k\big)
		\end{aligned}
\end{align}}
Also, using the same calculations as in \eqref{eq25a} and \eqref{eq1b} we obtain
{\setlength{\abovedisplayskip}{7pt}
\setlength{\belowdisplayskip}{7pt} 
\begin{align}\label{eq25b}
		\begin{aligned}
			-2I_2&=-2\gamma_2^{\mat L}{\w^{k+1}}\Tr\Pnorm{}\,{\mat S}_{\Sigma}^{L,2}-2\gamma_2^{\mat R}{\v^{k+1}}\Tr\Pnorm{}\,{\mat S}_{\Sigma}^{R,2}
			\le-2 \gamma_2 \calE_2^{k+1}
			+\dfrac{\kappa^2\hat \gamma_2}{\rho_{\mat L}}\big\|\D{}\w^{k+1}\big\|_{\mat L}^2 +\dfrac{\kappa^2\hat \gamma_2}{\rho_{\mat R}}\big\|\D{}\v^{k+1}\big\|_{\mat R}^2
			\\&-2\gamma_2\bigg(\epsilon\sum_{a=1}^3\R{\beta_1}\Cmat{1,a}{\mat L}\D{\xi_a}\w^{k+1}\bigg)\Tr\MSJack{\Sigma}{}^{-1}{}\Pnorm{\perp\xi_1}
			\bigg(\kappa\sum_{a=1}^3\R{\alpha_1}\Cmat{1,a}{\mat R}\D{\xi_a}\v^{k+1}-\kappa\sum_{a=1}^3\R{\alpha_1}\Cmat{1,a}{\mat R}\D{\xi_a}\v^{k}\bigg)	\\&-2\gamma_2\bigg(\epsilon\sum_{a=1}^3\R{\beta_1}\Cmat{1,a}{\mat L}\D{\xi_a}\w^{k+1}\bigg)\Tr\MSJack{\Sigma}{}^{-1}{}\Pnorm{\perp\xi_1}
			\bigg(\kappa\sum_{a=1}^3\R{\alpha_1}\Cmat{1,a}{\mat R}\D{\xi_a}\v^{k-1}-\kappa\sum_{a=1}^3\R{\alpha_1}\Cmat{1,a}{\mat R}\D{\xi_a}\v^{k}\bigg).
		\end{aligned}
\end{align}}
Substitute  \eqref{eq25a} and \eqref{eq25b} in \eqref{eq24a} to obtain
{\setlength{\abovedisplayskip}{7pt}
\setlength{\belowdisplayskip}{7pt} 
	\begin{align}\label{g1}
		\begin{aligned}
			&\|\w^{k+1}\|^2_{\mat L}+ \|\v^{k+1}\|^2_{\mat R}+2\delta t\,\gamma_1\calE_1^{k+1}+\delta t\,\gamma_2\calE_2^{k+1}
			+\delta t\,\epsilon (1-\, \dfrac{\epsilon\hat\gamma_2}{\rho_{\mat L}})\|\D{}\w^{k+1}\|^2_{\mat L}
			+\delta t\, \kappa (2-\, \dfrac{\kappa\hat\gamma_2}{\rho_{\mat R}})\|\D{}\v^{k+1}\|^2_{\mat R}
			\\&\le \|\w^{k}\|^2_{\mat L} +\|\v^{k}\|^2_{\mat R}+\calG^{k+1}
			+\delta t\,\dfrac{\epsilon}{\rho_{\mat L}}\left\|\R{\alpha_1}\v^{k+1}-\R{\beta_1}\w^{k+1}\right\|^2_{\Sigma}
			\\&\quad-2\delta t\,\gamma_1{\w^{k+1}}\Tr\R{\beta_1}\Tr\MSJack{\Sigma}{}\Pnorm{\perp\xi_1}\big(\R{\alpha_1}\v^{k+1}-\R{\alpha_1}\v^k\big)	
			-2\delta t\,\gamma_1{\w^{k+1}}\Tr\R{\beta_1}\Tr\MSJack{\Sigma}{}\Pnorm{\perp\xi_1}\big(\R{\alpha_1}\v^{k-1}-\R{\alpha_1}\v^{k}\big)	
			\\&\quad-2\delta t\,\gamma_2\bigg(\epsilon\sum_{a=1}^3\R{\beta_1}\Cmat{1,a}{\mat L}\D{\xi_a}\w^{k+1}\bigg)\Tr\MSJack{\Sigma}{}^{-1}\Pnorm{\perp\xi_1}
			\bigg(\kappa\sum_{a=1}^3\R{\alpha_1}\Cmat{1,a}{\mat R}\D{\xi_a}\v^{k+1}-\kappa\sum_{a=1}^3\R{\alpha_1}\Cmat{1,a}{\mat R}\D{\xi_a}\v^{k}\bigg)
			\\&\quad-2\delta t\,\gamma_2\bigg(\epsilon\sum_{a=1}^3\R{\beta_1}\Cmat{1,a}{\mat L}\D{\xi_a}\w^{k+1}\bigg)\Tr\MSJack{\Sigma}{}^{-1}\Pnorm{\perp\xi_1}
			\bigg(\kappa\sum_{a=1}^3\R{\alpha_1}\Cmat{1,a}{\mat R}\D{\xi_a}\v^{k-1}-\kappa\sum_{a=1}^3\R{\alpha_1}\Cmat{1,a}{\mat R}\D{\xi_a}\v^{k}\bigg).
		\end{aligned}
	\end{align}}
Then consider the term in  line 5 of the above inequality, we have 
{\setlength{\abovedisplayskip}{7pt}
\setlength{\belowdisplayskip}{7pt} 
	\begin{align*}
			\begin{aligned}
				-{\w^{k+1}}\Tr&\R{\beta_1}\Tr\MSJack{\Sigma}{}\Pnorm{\perp\xi_1}\big(\R{\alpha_1}\v^{k-1}-\R{\alpha_1}\v^{k}\big)	
				\\&=\big(\R{\alpha_1}\v^{k-1}\big)\Tr\MSJack{\Sigma}{}\Pnorm{\perp\xi_1}\big(\R{\alpha_1}\v^{k}-\R{\alpha_1}\v^{k-1}\big)	
				+\big(\R{\beta_1}\w^{k+1}-\R{\alpha_1}\v^{k-1}\big)\Tr\MSJack{\Sigma}{}\Pnorm{\perp\xi_1}\big(\R{\alpha_1}\v^{k}-\R{\alpha_1}\v^{k-1}\big)
				\\&\le\big(\R{\alpha_1}\v^{k-1}\big)\Tr\MSJack{\Sigma}{}\Pnorm{\perp\xi_1}\R{\alpha_1}\v^{k}-\big(\R{\alpha_1}\v^{k-1}\big)\Tr\MSJack{\Sigma}{}\Pnorm{\perp\xi_1}\R{\alpha_1}\v^{k-1}
				\\&\quad
				+\dfrac12\big\|\R{\alpha_1}\v^{k}-\R{\alpha_1}\v^{k-1}\big\|^2_{\Sigma}+\dfrac{\rho_{\mat R}}{4}\big\|\R{\beta_1}\w^{k+1}-\R{\alpha_1}\v^{k-1}\big\|^2_{\Sigma}
				+(\dfrac{1}{\rho_{\mat R}}-\dfrac12)\big\|\R{\alpha_1}\v^{k}-\R{\alpha_1}\v^{k-1}\big\|^2_{\Sigma}
				\\&\le\big(\R{\alpha_1}\v^{k-1}\big)\Tr\MSJack{\Sigma}{}\Pnorm{\perp\xi_1}\R{\alpha_1}\v^{k}-\big(\R{\alpha_1}\v^{k-1}\big)\Tr\MSJack{\Sigma}{}\Pnorm{\perp\xi_1}\R{\alpha_1}\v^{k-1}
				\\&\quad+\dfrac{1}{2}\|\R{\alpha_1}\v^{k}\|^2_\Sigma+\dfrac{1}{2}\|\R{\alpha_1}\v^{k-1}\|^2_\Sigma-\big(\R{\alpha_1}\v^{k-1}\big)\Tr\MSJack{\Sigma}{}\Pnorm{\perp\xi_1}\R{\alpha_1}\v^{k}
				\\&\quad+\dfrac{\rho_{\mat R}}{4}\big\|\R{\beta_1}\w^{k+1}-\R{\alpha_1}\v^{k+1}+\R{\alpha_1}\v^{k+1}-\R{\alpha_1}\v^{k-1}\big\|^2_{\Sigma}
				+(\dfrac{1}{\rho_{\mat R}}-\dfrac12)\big\|\R{\alpha_1}\v^{k}-\R{\alpha_1}\v^{k-1}\big\|^2_{\Sigma}
				\\&\le\dfrac{1}{2}\|\R{\alpha_1}\v^{k}\|^2_\Sigma-\dfrac{1}{2}\|\R{\alpha_1}\v^{k-1}\|^2_\Sigma
				+\dfrac{\rho_{\mat R}}{2}\big\|\R{\beta_1}\w^{k+1}-\R{\alpha_1}\v^{k+1}\big\|^2_{\Sigma}
				\\&\quad+\dfrac{\rho_{\mat R}}{2}\big\|\R{\alpha}\v^{k+1}-\R{\alpha_1}\v^{k-1}\big\|^2_{\Sigma}
				+\dfrac12(\dfrac{2}{\rho_{\mat R}}-1)\big\|\R{\alpha_1}\v^{k}-\R{\alpha_1}\v^{k-1}\big\|^2_{\Sigma}.
			\end{aligned}
\end{align*}}
Therefore
{\setlength{\abovedisplayskip}{7pt}
\setlength{\belowdisplayskip}{7pt} 
\begin{align*}
		\begin{aligned}
				-2{\w^{k+1}}\Tr&\R{\beta_1}\Tr\MSJack{\Sigma}{}\Pnorm{\perp\xi_1}\big(\R{\alpha_1}\v^{k-1}-\R{\alpha_1}\v^{k}\big)\le\|\R{\alpha}\v^{k}\|^2_{\Sigma}-\|\R{\alpha}\v^{k-1}\|^2_{\Sigma}
				\\&+\rho_{\mat R}\|\R{\beta_1}\w^{k+1}-\R{\alpha}\v^{k+1}\|^2
				+\rho_{\mat R}\|\R{\alpha_1}\v^{k+1}-\R{\alpha}\v^{k-1}\|^2_{\Sigma}
				+(\dfrac{2}{\rho_{\mat R}}-1)\big\|\R{\alpha_1}\v^{k}-\R{\alpha_1}\v^{k-1}\big\|^2_{\Sigma}.
		\end{aligned}
\end{align*}}
Doing the same kind of calculation for last three lines of the inequality \eqref{g1} and substitute, it simplifies to 
{\setlength{\abovedisplayskip}{7pt}
\setlength{\belowdisplayskip}{7pt}     
		\begin{align*}
			\begin{aligned}
				\|\w^{k+1}\|^2_{\mat L}&+ \|\v^{k+1}\|^2_{\mat R}
				+((1-\rho_{\mat R})\gamma_1-\dfrac{\epsilon }{\rho_{\mat L}})\delta t\calE_1^{k+1}
				+\delta t\,\epsilon (1-\, \dfrac{\epsilon\hat\gamma_2}{\rho_{\mat L}})\D{}\w^{k+1}\|^2_{\mat L}
				+\delta t\, \kappa (2-\, \dfrac{\kappa\hat\gamma_2}{\rho_{\mat R}})\|\D{}\v^{k+1}\|^2_{\mat R}
				\\\le&\|\w^{k}\|^2_{\mat L} +\|\v^{k}\|^2_{\mat R}
				+\calG^{k+1}
				\\&+\delta t\,\gamma_1\left(\left\|\R{\alpha_1}\v^{k}\right\|^2_{\Sigma}
				-\left\|\R{\alpha_1}\v^{k+1}\right\|^2_{\Sigma}\right)
				+\delta t\,\gamma_1\left(\left\|\R{\alpha_1}\v^{k}\right\|^2_{\Sigma}
				-\left\|\R{\alpha_1}\v^{k-1}\right\|^2_{\Sigma}\right)
				\\&+\delta t\,\gamma_1\rho_{\mat R}\left\|\R{\alpha_1}\v^{k+1}-\R{\alpha_1}\v^{k-1}\right\|^2_{\Sigma}
				+\delta t\,\gamma_1(\dfrac{2}{\rho_{\mat R}}-1)\big\|\R{\alpha_1}\v^{k}-\R{\alpha_1}\v^{k-1}\big\|^2_{\Sigma}
				\\&+\delta t\,\kappa^2\gamma_2\left(
				\bigg\|\sum_{a=1}^3\R{\alpha_1}\Cmat{1,a}{\mat R}\D{\xi_a}\v^{k}\bigg\|^2_{\bar \Sigma}
				-\bigg\|\sum_{a=1}^3\R{\alpha_1}\Cmat{1,a}{\mat R}\D{\xi_a}\v^{k+1}\bigg\|^2_{\bar \Sigma}\right)
				\\&+\delta t\,\kappa^2\gamma_2\left(
				\bigg\|\sum_{a=1}^3\R{\alpha_1}\Cmat{1,a}{\mat R}\D{\xi_a}\v^{k}\bigg\|^2_{\bar \Sigma}
				-\bigg\|\sum_{a=1}^3\R{\alpha_1}\Cmat{1,a}{\mat R}\D{\xi_a}\v^{k-1}\bigg\|^2_{\bar \Sigma}\right)
				\\&+\delta t\,\kappa^2\gamma_2\,
				\bigg\|\sum_{a=1}^3\R{\alpha_1}\Cmat{1,a}{\mat R}\D{\xi_a}\v^{k+1}
				-\sum_{a=1}^3\R{\alpha_1}\Cmat{1,a}{\mat R}\D{\xi_a}\v^{k-1}\bigg\|^2_{\bar \Sigma}
				\\&+\delta t\,\kappa^2\gamma_2\,
				\bigg\|\sum_{a=1}^3\R{\alpha_1}\Cmat{1,a}{\mat R}\D{\xi_a}\v^{k}
				-\sum_{a=1}^3\R{\alpha_1}\Cmat{1,a}{\mat R}\D{\xi_a}\v^{k-1}\bigg\|^2_{\bar \Sigma}.
			\end{aligned}
\end{align*}}
We assume $\rho_{\mat R}\le 1$ and \eqref{b1}-\eqref{b4} hold, and use Lemma \eqref{lemTrace}, 
{\setlength{\abovedisplayskip}{7pt}
\setlength{\belowdisplayskip}{7pt} 
	\begin{align*}
			\begin{aligned}
				\|\w^{k+1}\|^2_{\mat L}&+ \|\v^{k+1}\|^2_{\mat R}
				+((1-\rho_{\mat R})\gamma_1-\dfrac{\epsilon }{\rho_{\mat L}})\delta t\calE_1^{k+1}
				\\&+\delta t\,\epsilon (1-\, \dfrac{\epsilon\hat\gamma_2}{\rho_{\mat L}})\D{}\w^{k+1}\|^2_{\mat L}
				+\delta t\, \kappa (2-\, \dfrac{\kappa\hat\gamma_2}{\rho_{\mat R}})\|\D{}\v^{k+1}\|^2_{\mat R}
				\\\le& \|\w^{k}\|^2_{\mat L} +\|\v^{k}\|^2_{\mat R}
				+\calG^{k+1}
				\\&+\delta t\,\gamma_1\left(\left\|\R{\alpha_1}\v^{k}\right\|^2_{\Sigma}
				-\left\|\R{\alpha_1}\v^{k+1}\right\|^2_{\Sigma}\right)
				+\delta t\,\gamma_1\left(\left\|\R{\alpha_1}\v^{k}\right\|^2_{\Sigma}
				-\left\|\R{\alpha_1}\v^{k-1}\right\|^2_{\Sigma}\right)
				\\&+\delta t\,\gamma_1\big(\left\|\v^{k+1}\|_{\mat R}+\|\v^{k-1}\right\|^2_{\mat R}\big)
				+\delta t\,\gamma_1\,\dfrac{1}{\rho_{\mat R}}(\dfrac{2}{\rho_{\mat R}}-1)\big(\left\|\v^{k}\|^2_{\mat R}+\|\v^{k-1}\right\|^2_{\mat R}\big)
				\\&+\delta t\,\kappa^2\gamma_2\left(
				\bigg\|\sum_{a=1}^3\R{\alpha_1}\Cmat{1,a}{\mat R}\D{\xi_a}\v^{k}\bigg\|^2_{\bar \Sigma}
				-\bigg\|\sum_{a=1}^3\R{\alpha_1}\Cmat{1,a}{\mat R}\D{\xi_a}\v^{k+1}\bigg\|^2_{\bar \Sigma}\right)
				\\&+\delta t\,\kappa^2\gamma_2\left(
				\bigg\|\sum_{a=1}^3\R{\alpha_1}\Cmat{1,a}{\mat R}\D{\xi_a}\v^{k}\bigg\|^2_{\bar \Sigma}
				-\bigg\|\sum_{a=1}^3\R{\alpha_1}\Cmat{1,a}{\mat R}\D{\xi_a}\v^{k-1}\bigg\|^2_{\bar \Sigma}\right)
				\\&+\delta t\,\kappa^2\,\dfrac{\gamma_2}{\rho_{\mat R}}\bigg(\big
				\|\D{}\v^{k+1}\big\|^2_{\mat R}
				+\big\|\D{}\v^{k-1}\big\|^2_{\mat R}\bigg)
				+\delta t\,\kappa^2\,\dfrac{\gamma_2}{\rho_{\mat R}}
				\bigg(\big
				\|\D{}\v^{k}\big\|^2_{\mat R}
				+\big\|\D{}\v^{k-1}\big\|^2_{\mat R}\bigg).
			\end{aligned}
\end{align*}}
Then  simplify and again by assuming that  conditions \eqref{b1}-\eqref{b4} hold, we have
{\setlength{\abovedisplayskip}{7pt}
\setlength{\belowdisplayskip}{7pt}  
\begin{align*}
		\begin{aligned}
			\|\w^{k+1}\|^2_{\mat L}&+ (1-\delta t\gamma_1) \|\v^{k+1}\|^2_{\mat R} 
			+((1-\rho_{\mat R})\gamma_1-\dfrac{\epsilon }{\rho_{\mat L}})\delta t\calE_1^{k+1}
			\\&+\delta t\,\epsilon (1-\, \dfrac{\epsilon\hat\gamma_2}{\rho_{\mat L}})\|\D{}\w^{k+1}\|^2_{\mat L}
			+\delta t\, \kappa (2-\,\dfrac{\kappa\gamma_2}{\rho_{\mat R}}-\, \dfrac{\kappa\hat\gamma_2}{\rho_{\mat R}})\|\D{}\v^{k+1}\|^2_{\mat R}
			\\\le&\|\w^{k}\|^2_{\mat L} +(1-\delta t\gamma_1)\|\v^{k}\|^2_{\mat R} +\calG^{k+1}
			\\&+\delta t\,\gamma_1\left(\left\|\R{\alpha_1}\v^{k}\right\|^2_{\Sigma}
			-\left\|\R{\alpha_1}\v^{k+1}\right\|^2_{\Sigma}\right)
			+\delta t\,\gamma_1\left(\left\|\R{\alpha_1}\v^{k}\right\|^2_{\Sigma}
			-\left\|\R{\alpha_1}\v^{k-1}\right\|^2_{\Sigma}\right)
			\\&+\delta t\,\kappa^2\gamma_2\left(
			\bigg\|\sum_{a=1}^3\R{\alpha_1}\Cmat{1,a}{\mat R}\D{\xi_a}\v^{k}\bigg\|^2_{\bar \Sigma}
			-\bigg\|\sum_{a=1}^3\R{\alpha_1}\Cmat{1,a}{\mat R}\D{\xi_a}\v^{k+1}\bigg\|^2_{\bar \Sigma}\right)
			\\&+\delta t\,\kappa^2\gamma_2\left(
			\bigg\|\sum_{a=1}^3\R{\alpha_1}\Cmat{1,a}{\mat R}\D{\xi_a}\v^{k}\bigg\|^2_{\bar \Sigma}
			-\bigg\|\sum_{a=1}^3\R{\alpha_1}\Cmat{1,a}{\mat R}\D{\xi_a}\v^{k-1}\bigg\|^2_{\bar \Sigma}\right)
			\\&
			+\delta t\,\dfrac{2\gamma_1}{\rho^2_{\mat R}}\left(\|\v^{k-1}\|^2_{\mat R}+\|\v^{k}\|^2_{\mat R}\right)
			+\delta t\,\kappa^2\,\dfrac{2\gamma_2}{\rho_{\mat R}}
			\left(\big\|\D{}\v^{k-1}\big\|^2_{\mat R}+\big\|\D{}\v^{k}\big\|^2_{\mat R}\right).
		\end{aligned}
\end{align*}}
Consequently
{\setlength{\abovedisplayskip}{7pt}
\setlength{\belowdisplayskip}{7pt}  
	\begin{align*}
		\begin{aligned}
			\|\w^{k+1}&\|^2_{\mat L}+ (1-\delta t\gamma_1) \|\v^{k+1}\|^2_{\mat R} 
			+((1-\rho_{\mat R})\gamma_1-\dfrac{\epsilon }{\rho_{\mat L}})\delta t\sum_{k=2}^{k+1}\calE_1^{j}
			+\delta t\,\epsilon (1-\, \dfrac{\epsilon\hat\gamma_2}{\rho_{\mat L}})\|D\w^{k+1}\|^2_{\mat L}
			+\delta t\, \kappa (2-\,\dfrac{\kappa\gamma_2}{\rho_{\mat R}}-\, \dfrac{\kappa\hat\gamma_2}{\rho_{\mat R}})\|\D{}\v^{k+1}\|^2_{\mat R}
			\\&\le \|\w^{1}\|^2_{\mat L} +\|\v^{1}\|^2_{\mat R} 
			+\sum_{j=2}^{k+1}\calG^{j}
			+\delta t\,\gamma_1\sum_{j=1}^k\left(\left\|\R{\alpha_1}\v^{j}\right\|^2_{\Sigma}
			-\left\|\R{\alpha_1}\v^{j+1}\right\|^2_{\Sigma}\right)
			+\delta t\,\gamma_1\sum_{j=1}^k\left(\left\|\R{\alpha_1}\v^{j}\right\|^2_{\Sigma}
			-\left\|\R{\alpha_1}\v^{j-1}\right\|^2_{\Sigma}\right)
			\\&\qquad\qquad\qquad\qquad\qquad\qquad+\delta t\,\kappa^2\gamma_2\sum_{j=1}^k\left(
			\bigg\|\sum_{a=1}^3\R{\alpha_1}\Cmat{1,a}{\mat R}\D{\xi_a}\v^{j}\bigg\|^2_{\bar \Sigma}
			-\bigg\|\sum_{a=1}^3\R{\alpha_1}\Cmat{1,a}{\mat R}\D{\xi_a}\v^{j+1}\bigg\|^2_{\bar \Sigma}\right)
			\\&\qquad\qquad\qquad\qquad\qquad\qquad+\delta t\,\kappa^2\gamma_2\sum_{j=1}^k\left(
			\bigg\|\sum_{a=1}^3\R{\alpha_1}\Cmat{1,a}{\mat R}\D{\xi_a}\v^{j}\bigg\|^2_{\bar \Sigma}
			-\bigg\|\sum_{a=1}^3\R{\alpha_1}\Cmat{1,a}{\mat R}\D{\xi_a}\v^{j-1}\bigg\|^2_{\bar \Sigma}\right)
			\\&\qquad\qquad\qquad\qquad\qquad\qquad+\delta t\,\dfrac{2\gamma_1}{\rho^2_{\mat R}}\sum_{j=1}^k \left(\|\v^{j-1}\|^2_{\mat R}+\|\v^{j}\|^2_{\mat R}\right)+\delta t\,\kappa^2\,\dfrac{2\gamma_2}{\rho_{\mat R}}\sum_{j=1}^k\left(\|\D{}\v^{j-1}\|^2_{\mat R}+\|\D{}\v^{j}\|^2_{\mat R}\right).
		\end{aligned}
\end{align*}}
Since there are telescoping sums, the above simplifies to
{\setlength{\abovedisplayskip}{7pt}
\setlength{\belowdisplayskip}{7pt}  
\begin{align*}
		\begin{aligned}
			\|\w^{k+1}&\|^2_{\mat L}+ (1-\delta t\gamma_1)\|\v^{k+1}\|^2_{\mat R}
			+\delta t\,\gamma_1(\calF_1^0+\calF_1^{k+1})
			+\delta t\,\kappa^2\gamma_2(\calF_2^{0}+\calF_2^{k+1})
			\\&\quad+((1-\rho_{\mat R})\gamma_1-\dfrac{\epsilon }{\rho_{\mat L}})\delta t\sum_{k=2}^{k+1}\calE_1^{j}
			+\delta t\,\epsilon (1-\, \dfrac{\epsilon\hat\gamma_2}{\rho_{\mat L}})\|\D{}\w^{k+1}\|^2_{\mat L}
			+\delta t\, \kappa (2-\,\dfrac{\kappa\gamma_2}{\rho_{\mat R}}-\, \dfrac{\kappa\hat\gamma_2}{\rho_{\mat R}})\|\D{}\v^{k+1}\|^2_{\mat R}
			\\&\le \|\w^{1}\|^2_{\mat L} +\|\v^{1}\|^2_{\mat R} +\sum_{j=2}^{k+1}\calG^{j}
			+\delta t\,\gamma_1\left\|\R{\alpha_1}\v^{1}\right\|^2_{\Sigma}
			+\delta t\,\gamma_1\left\|\R{\alpha_1}\v^{k}\right\|^2_{\Sigma}
			\\&\quad+\delta t\,\kappa^2\gamma_2
			\bigg\|\sum_{a=1}^3\R{\alpha_1}\Cmat{1,a}{\mat R}\D{\xi_a}\v^{1}\bigg\|^2_{\bar \Sigma}
			+\delta t\,\kappa^2\gamma_2
			\bigg\|\sum_{a=1}^3\R{\alpha_1}\Cmat{1,a}{\mat R}\D{\xi_a}\v^{k}\bigg\|^2_{\bar \Sigma}
			\\&\quad+\delta t\,\dfrac{2\gamma_1}{\rho^2_{\mat R}}\sum_{j=1}^k \left(\|\v^{j-1}\|^2_{\mat R}+\|\v^{j}\|^2_{\mat R}\right)+\delta t\,\kappa^2\,\dfrac{2\gamma_2}{\rho_{\mat R}}\sum_{j=1}^k\left\|\D{}\v^{j-1}\|^2_{\mat R}+\|\D{}\v^{j}\|^2_{\mat R}\right).
		\end{aligned}
\end{align*}}
Using Lemma \eqref{lemTrace} and summing up
{\setlength{\abovedisplayskip}{7pt}
\setlength{\belowdisplayskip}{7pt}  
\begin{align*}
		\begin{aligned}
			\|\w^{k+1}&\|^2_{\mat L}+ (1-\delta t\gamma_1)\|\v^{k+1}\|^2_{\mat R}
			+\delta t\,\gamma_1(\calF_1^0+\calF_1^{k+1})
			+\delta t\,\kappa^2\gamma_2(\calF_2^{0}+\calF_2^{k+1})
			\\&\quad+((1-\rho_{\mat R})\gamma_1-\dfrac{\epsilon }{\rho_{\mat L}})\delta t\sum_{k=2}^{k+1}\calE_1^{j}
			+\delta t\,\epsilon (1-\, \dfrac{\epsilon\hat\gamma_2}{\rho_{\mat L}})\|\D{}\w^{k+1}\|^2_{\mat L}
			+\delta t\, \kappa (2-\,\dfrac{\kappa\gamma_2}{\rho_{\mat R}}-\, \dfrac{\kappa\hat\gamma_2}{\rho_{\mat R}})\|\D{}\v^{k+1}\|^2_{\mat R}
			\\&\;\le \|\w^{1}\|^2_{\mat L} +\|\v^{1}\|^2_{\mat R} +\sum_{j=2}^{k+1}\calG^{j}
			+\delta t\, \,\dfrac{4\gamma_1}{\rho_{\mat R}^2}\sum_{j=1}^k\left\|\v^{j}\right\|^2_{\mat R}
			+\delta t\,\kappa^2\,\dfrac{4\gamma_2}{\rho_{\mat R}}\sum_{j=1}^k\big\|\D{}\v^{j}\big\|^2_{\mat R}.
		\end{aligned}
\end{align*}}
Notice that by theorem \eqref{thm:BE} under assumption \eqref{b3}
{\setlength{\abovedisplayskip}{7pt}
\setlength{\belowdisplayskip}{7pt}  
\begin{align*}
		\begin{aligned}
			& \|\w^{1}\|^2_{\mat L}+ \|\v^{1}\|^2_{\mat R}+\delta t\gamma_1\calE_1^{1}+\delta t\gamma_2\calE_2^{1}
			+\delta t \gamma_1\calF_1^0+\kappa^2\delta t\gamma_2\calF_2^0\le \calG^{1}+\calG^{0}.
		\end{aligned}
\end{align*}}
Assuming this and assumptions \eqref{b1}-\eqref{b3} holds, then we obtain
{\setlength{\abovedisplayskip}{7pt}
\setlength{\belowdisplayskip}{7pt}  
\begin{align*}
		\begin{aligned}
			\|\w^{k+1}\|^2_{\mat L}+ (1-\delta t\gamma_1)\|\v^{k+1}\|^2_{\mat R}
			&+\delta t\, \kappa (2-\,\dfrac{\kappa\gamma_2}{\rho_{\mat R}}-\, \dfrac{\kappa\hat\gamma_2}{\rho_{\mat R}})\|\D{}\v^{k+1}\|^2_{\mat R}
			\\&\quad\le \mathcal{S}^{k+1}
			+(1-\delta t\gamma_1)\sum_{j=1}^k\left\|\v^{j}\right\|^2_{\mat R}
			+\kappa\delta t\,(2-\,\dfrac{\kappa\gamma_2}{\rho_{\mat R}}-\, \dfrac{\kappa\hat\gamma_2}{\rho_{\mat R}})\sum_{j=1}^k\big\|\D{}\v^{j}\big\|^2_{\mat R}.
		\end{aligned}
\end{align*}}
where $\mathcal{S}^{j}=\sum_{i=0}^j\calG^{i}$. Therefore
{\setlength{\abovedisplayskip}{7pt}
\setlength{\belowdisplayskip}{7pt}  
\begin{align*}
		\begin{aligned}
			\|\w^{k+1}\|^2_{\mat L}+ (1-\delta t\gamma_1)\|\v^{k+1}\|^2_{\mat R}
			\le \mathcal{H}^{k+1},
		\end{aligned}
\end{align*}}
where
\[\mathcal{H}^{k+1}=\mathcal{S}^{k+1}+(k+1)\mathcal{S}^1+\sum_{j=1}^{k}(k+1-j)\mathcal{S}^j,\]
assuming that the conditions \eqref{b1}-\eqref{b4} hold.
\end{proof}

\end{document}